%% file: comparingcombinatorial.tex
\documentclass[a4paper]{amsart}

\input{preamble}
\title[Comparing models of moduli space and their compactifications]{Comparing combinatorial models of moduli space and their compactifications}
\author{Daniela Egas Santander}
\author{Alexander Kupers}
\date{\today}

\begin{document}

\begin{abstract}
We compare two combinatorial models for the moduli space of two-dimensional cobordisms: B\"odigheimer's radial slit configurations and Godin's admissible fat graphs, producing an explicit homotopy equivalence using a ``critical graph'' map. We also discuss natural compactifications of these two models, the unilevel harmonic compactification and Sullivan diagrams respectively, and prove that the homotopy equivalence induces a cellular homeomorphism between these compactifications.
\end{abstract}

\maketitle

\tableofcontents

\section{Introduction} 
\input{introduction}

\section{Radial slit configurations and the harmonic compactification}
\label{sec_radialslits}
\input{radialslits}

\section{Admissible fat graphs and string diagrams}
\label{sec_fatgraphs}

\input{fatgraphs}

\section{The critical graph equivalence} 
\label{sec_critmap}
\input{projection}

\section{Sullivan diagrams and the harmonic compactification}
\label{sec_sdharmonic}
\input{sullivanandrad}

\bibliographystyle{amsalpha}
\bibliography{stringtopology}
\end{document}

%% file: preamble.tex
\newcommand{\cB}{\mathcal{B}}
\newcommand{\cD}{\mathcal{D}}
\newcommand{\cF}{\mathcal{F}}
\newcommand{\cG}{\mathcal{G}}
\newcommand{\cI}{\mathcal{I}}
\newcommand{\cL}{\mathcal{L}}
\newcommand{\cM}{\mathcal{M}}

\newcommand{\cT}{\mathcal{T}}

\newcommand{\bA}{\mathbb{A}}
\newcommand{\bC}{\mathbb{C}}

\newcommand{\bR}{\mathbb{R}}

\usepackage{amsmath}%to controll numbering conditions
\usepackage{amsfonts}
\usepackage{amsthm}%nice thm and proof environment
\usepackage[pdftex]{color, graphicx}
\usepackage{mathrsfs}%special letters REVISAR
\usepackage[final]{pdfpages}%to always see graphics
\usepackage{yfonts} %for Rad letters
\numberwithin{equation}{section}
\numberwithin{figure}{section}
\numberwithin{table}{section}
\usepackage{lmodern}
\usepackage[shortlabels]{enumitem}
\usepackage[dvipsnames,svgnames,x11names,hyperref]{xcolor}
\usepackage{url,graphicx,verbatim,amssymb,stmaryrd,mathtools,mathabx,microtype}
\usepackage[pagebackref,colorlinks,citecolor=blue,linkcolor=blue,urlcolor=blue,filecolor=blue]{hyperref}
\usepackage[hmargin=1.37in,vmargin=1.35in]{geometry}
\usepackage{tikz,tikz-cd}
\usepackage{caption}
\AtBeginDocument{%
	\def\MR#1{}}

\theoremstyle{plain}
\newtheorem{theorem}{Theorem}[section]

\newtheorem{lemma}[theorem]{Lemma}
\newtheorem{proposition}[theorem]{Proposition}
\newtheorem{corollary}[theorem]{Corollary}

\theoremstyle{definition}
\newtheorem{definition}[theorem]{Definition}
\newtheorem{example}[theorem]{Example}

\theoremstyle{remark}
\newtheorem{remark}[theorem]{Remark}
\newtheorem{notation}[theorem]{Notation}

\usetikzlibrary{arrows,shapes,shapes.symbols,decorations.markings,decorations.pathmorphing,chains,calc,patterns,backgrounds,fit,positioning,matrix,snakes}
\tikzset{dot/.style={circle,fill=black,thick,inner sep=0pt,minimum size=1mm,draw}}
\tikzset{arrow/.style={semithick,>=stealth',shorten >=1pt,shorten <=1pt}}
\tikzset{equal/.style={kant,double distance=2pt}}

\DeclareMathAlphabet{\mathpzc}{OT1}{pzc}{m}{it}
\newcommand{\Z}{\mathbb Z}
\newcommand{\R}{\mathbb R}

\newcommand{\C}{\mathbb C}

\newcommand{\Acat}{\mathscr{A}}

\newcommand{\cof}{\hookrightarrow}

\newcommand{\adj}{\leftrightarrows}

\newcommand{\Sg}{S_{g,n+m}}
\newcommand{\Modgpq}{\mathrm{Mod}(S_{g,n+m})}

\newcommand{\Diff}{\mathrm{Diff}}

\newcommand{\Mark}{\mathrm{Mark}}

\newcommand{\Fatad}{\mathpzc{Fat}^{\mathpzc{ad}}}
\newcommand{\Fat}{\mathpzc{Fat}}

\newcommand{\Ead}{\mathpzc{EFat}^{\mathpzc{ad}}}

\newcommand{\Fatadg}{\mathpzc{Fat}^{\mathpzc{ad}}_{g,n+m}}

\newcommand{\Eadg}{\mathpzc{EFat}^{\mathpzc{ad}}_{g,n+m}}
\newcommand{\SD}{\mathcal{SD}}

\newcommand{\MFatad}{\mathpzc{MFat}^{\mathpzc{ad}}}

\newcommand{\MFat}{\mathpzc{MFat}}

\newcommand{\EMad}{\mathpzc{EMFat}^{\mathpzc{ad}}}
\newcommand{\MFatadg}{\mathpzc{MFat}^{\mathpzc{ad}}_{g,n+m}}
\newcommand{\EMadg}{\mathpzc{EMFat}^{\mathpzc{ad}}_{g,n+m}}

\newcommand{\Conf}{\mathrm{Conf}}

\newcommand{\Rad}{\textfrak{Rad}}
\newcommand{\RAD}{\mathbf{RAD}}
\newcommand{\bRAD}{\overline{\mathbf{RAD}}}
\newcommand{\UbRad}{\overline{\textfrak{URad}}}
\newcommand{\QRad}{\textfrak{QRad}}
\newcommand{\bQRad}{\overline{\textfrak{QRad}}}
\newcommand{\Radt}{\textfrak{Rad}^{\sim}}
\newcommand{\bPRad}{\overline{\textfrak{PRad}}}
\newcommand{\PRad}{\textfrak{PRad}}
\newcommand{\bRad}{\overline{\textfrak{Rad}}}
\newcommand{\ERad}{E\textfrak{Rad}}
\newcommand{\bprad}{\overline{\mathrm{PRad}}}
\newcommand{\prad}{\mathrm{PRad}}
\newcommand{\bqrad}{\overline{\mathrm{QRad}}}
\newcommand{\qrad}{\mathrm{QRad}}
\newcommand{\brad}{\overline{\mathrm{Rad}}}
\newcommand{\rad}{\mathrm{Rad}}

\newcommand{\mr}[1]{{\rm #1}}

\newcommand{\lra}{\longrightarrow}

%% file: introduction.tex
In this paper we compare two combinatorial models of the moduli space of cobordisms. We start with an introduction to moduli space, giving a conformal description of it. After that we describe various combinatorial models and how they relate to each other, which includes our main result, Theorem \ref{thm_main}. Finally we describe two applications.

\subsection{The moduli space of cobordisms} \label{subsec_confmodel} The study of families of surfaces, known as ``moduli theory'', goes back to the nineteenth century. One of the main points of this theory is the construction of a \textit{moduli space}; informally, this is a space of all surfaces isomorphic to a given one, characterized by the property that equivalence classes of maps into it correspond to equivalence classes of families of surfaces. For applications to field theories, the surfaces of interest are two-dimensional oriented cobordisms; an oriented surface $\smash{S}$ with parametrized boundary divided into an incoming and outgoing part. More precisely, there is a pair of maps $\iota_\mr{in} \colon \sqcup_{i=1}^n S^1 \to \partial S$ and $\iota_\mr{out} \colon \sqcup_{j=1}^m S^1 \to \partial S$ such that $\iota_\mr{in} \sqcup \iota_\mr{out}$ is a diffeomorphism onto $\partial S$. 

We will now give a conformal definition of the moduli space of these cobordisms, following \cite[Section 2]{bodigheimer} and \cite{hamenstadt}. Let $S$ be an isomorphism class of connected two-dimensional oriented cobordisms with non-empty incoming and outgoing boundary. As we will later endow $S$ with a metric, a parametrization of its boundary is given by a point in each boundary component. So $S = S_{g,n+m}$ is a connected oriented surface of genus $g$ with $n+m$ boundary components, each containing a single point $p_i$ for $1 \leq i \leq n+m$.  The marked points are ordered and divided into an incoming set (which contains the first $n \geq 1$ marked points) and an outgoing set (which contains the last $m \geq 1$ marked points).

To define the moduli space we start by considering the set of metrics $g$ on $S$. Two metrics are said to be conformally equivalent if they are equal up to a pointwise rescaling by a continuous function. This is equivalent to having the same notion of angle. A diffeomorphism $f \colon S_1 \to S_2$ between two-dimensional manifolds $(S_1,[g]_1)$, $(S_2,[g]_2)$ with conformal classes of metrics such that $f^*[g]_2 = [g]_1$, is said to be a conformal diffeomorphism. This is equivalent to each of its differentials $D_p f$ for $p \in S_1$ being a linear map that preserves angles.

We will restrict our attention to those conformal classes of metrics on $S$ so that each incoming  boundary component has a neighborhood that is conformally diffeomorphic to a neighborhood of the boundary of $\{z \in \bC\,|\,||z|| \geq 1\}$ and each outgoing boundary component has a neighborhood that is conformally diffeomorphic to a neighborhood of the boundary of $\{z \in \bC\,|\,||z|| \leq 1\}$. We say that these conformal classes have good boundary.

The moduli space $\cM_g(n,m)$ will have as underlying set the conformal classes of metrics on $S$ with good boundary modulo the equivalence relation of conformal diffeomorphism fixing the points $p_i$. To topologize it, we introduce the Teichm\"uller metric. With respect to this metric, two equivalence classes of metrics on $S$ are close if they are related by a homeomorphism that---away from a finite set---is not only differentiable, but also conformal up to a small error. To make this precise, note that a linear map $D \colon \bR^2 \to \bR^2$ is conformal if and only if $\max \frac{||Dv||}{||v||} = \min \frac{||Dv||}{||v||}$, both the maximum and minimum taken over non-zero vectors. Hence we can quantify the deviation of a linear map from being conformal by its eccentricity:
\[\mr{Ecc}(D) \coloneqq \frac{\max ||Dv||/||v||}{\min ||Dv||/||v||}.\]

If $f \colon (S,[g]_1) \to (S,[g]_2)$ is a homeomorphism that is continuously differentiable outside a finite set of points $\Sigma  \subset S$, then its quasi-conformal constant $K_f$ is defined to be
\[K_f \coloneqq \sup_{p \in S \backslash \Sigma}\,\mr{Ecc}(D_p f),\]
and $f$ is said to be quasi-conformal if $K_f$ is finite. If $QC([g]_1,[g]_2)$ denotes the set of all quasiconformal homeomorphisms between $(S,[g]_1)$ and $(S,[g]_2)$ fixing the points $p_i$, then we can define the Teichm\"uller distance between $[g]_1$ and $[g]_2$ as follows:
\[d_\cT((S,[g]_1),(S,[g]_2)) \coloneqq \log \inf\{K_f\,|\,f \in QC([g]_1,[g]_2)\}.\]
The moduli space of two-dimensional oriented cobordisms isomorphic to $S$ is then defined to be the following metric space:
\[\cM_g(n,m) \coloneqq \left(\frac{\text{conformal classes of metrics on $S$ with good boundary}}{\text{conformal diffeomorphisms fixing the points $p_i$}},d_\cT\right).\]

For $S$ that are not connected, we take the product of these spaces over all components. An alternative definition of these spaces is as the quotient of Teichm\"uller space (the space of quasiconformal maps modulo conformal equivalence) by the action of the mapping class group $\mr{Mod}(S,\partial S)$, i.e., the group of components of the diffeomorphism group $\mr{Diff}(S,\partial S)$. This is a free proper action on a contractible space and hence $\cM_g(n,m) \simeq B\mr{Mod}(S,\partial S)$. All connected components of $\mr{Diff}(S,\partial S)$ are contractible and we can thus conclude that
\[\cM_g(n,m) \simeq B\mr{Mod}(S,\partial S) \simeq B\mr{Diff}(S,\partial S).\] 
This explains why $\cM_g(n,m)$ is a model for the moduli space of two-dimensional oriented cobordisms; any bundle of cobordisms over a paracompact space $B$ with transition functions given by diffeomorphisms, can be obtained up to isomorphism by pulling back a universal bundle over $\cM_g(n,m)$ along a map $B \to \cM_g(n,m)$. This universal bundle is the quotient of the space consisting of pairs $([g],x)$ of a conformal class of metrics and a point $x \in S$, by conformal diffeomorphisms acting diagonally.

\subsection{Combinatorial models of moduli space} In this paper we discuss several combinatorial models of the moduli spaces $\cM_g(n,m)$, as well as certain compactifications. The following diagram spells out the relations between them (we fix $g$, $n$ and $m$ and drop them from the notation):
\[\begin{tikzcd} \cM   & & \\[-5pt]
	\RAD \uar[swap]{\cong} \dar[two heads]{\simeq} & &  \vert\Fatad\vert \dar{\simeq} \rar{\simeq} & \vert\Fat\vert \dar{\simeq} \\[-5pt]
	\Rad \arrow{dd}[description]{\text{compactification}} & \Radt \lar[two heads]{(\ref{subsec_blowup})}[swap]{\simeq} \rar{\simeq}[swap]{(\ref{subsec_critgraph})} & \MFatad \ar[two heads]{ddd}[description]{\text{quotient by slides}} \rar{\simeq} & \MFat &\\[-5pt]
	& & & \\[-5pt]
	\bRad \dar[two heads]{\simeq} &  &  & \\[-5pt]
	\UbRad \arrow{rr}{\cong}[swap]{(\ref{sec_sdharmonic})} & & \SD. & \end{tikzcd}\]
Each arrow is a continuous map; if decorated by $\simeq$ it is homotopy equivalence, if it is double-headed it is a surjection, and if decorated by $\cong$ it is a homeomorphism. The objects that appear in this diagram are summarized below:
\begin{description}
\item[Moduli space $\cM$] This is the archetypical ``space of cobordisms,'' a conformal model of which was discussed in Section \ref{subsec_confmodel}. It consists of conformal classes of metrics modulo conformal diffeomorphisms, with the Teichm\"uller metric.
\item[The radial slit configurations $\RAD$ and $\Rad$] This model is due to B\"odigheimer, consisting of glueing data to construct a conformal class of metric by glueing together annuli in $\bC$. The main theorem of \cite{bodigheimer} is that there is a homeomorphism $\cM \cong \RAD$. There is a deformation retraction of $\RAD$ onto $\Rad$ by fixing the radii of the annuli. This and related models will be discussed in Section \ref{sec_radialslits}, and $\Rad$ will be defined in Definition \ref{def_radcw}.
\item[The fat graphs $\Fat$] Fat graphs are graphs with the additional structure of a cyclic ordering of the edges going into each vertex and data encoding the parametrization of its ``boundary components.'' Taking as morphisms maps of fat graphs that collapse a disjoint union of trees defines a category of fat graphs $\Fat$. The space $\vert \Fat\vert$ is the geometric realization of this category. This and related models will be discussed in Section \ref{sec_fatgraphs}, and $\Fat$ will be defined in Definition \ref{def_fatfatad}.
\item[The admissible fat graphs $\Fatad$] A fat graph is said to be admissible if its incoming boundary graph embeds in it. The space $\vert \Fatad \vert$ is the geometric realization of the full subcategory on the admissible fat graphs. It is defined in Definition \ref{def_fatfatad}.
\item[The metric fat graphs $\MFat$] Closely related to $\Fat$ is the space of metric fat graphs $\MFat$. This is the space of fat graphs with the additional data of lengths of their edges. The topology is described in terms of these lengths and it contains the realization of $\Fat$ as a deformation retract.
\item[The admissible metric fat graphs $\MFatad$] Just like $\Fatad$ is the subcategory of $\Fat$ consisting of fat graphs that are admissible, $\MFatad$ is the subspace of $\MFat$ consisting of metric fat graphs that are admissible. It is defined in Definition \ref{def_mfatad}.
\item[The fattening of the radial slit configurations $\Radt$] To discuss the relation between $\Rad$ and $\MFat$, in this paper we introduce $\Radt$ as a thicker version of $\Rad$ by including resolutions of the critical graph for non-generic radial slit configurations. This is done in Subsection \ref{subsec_blowup}.
\item[The harmonic compactification $\bRad$] Naturally $\Rad$ arises as an open subspace of a compact space $\bRad$. In this compactification we allow identifications of points on the outgoing boundary and allow handles to degenerate to intervals. It is defined in Definition \ref{def_radcw}.
\item[The unilevel harmonic compactification $\UbRad$] The space $\UbRad$ is a deformation retract of $\bRad$ obtained by making all slits equal length. It is defined in Definition \ref{def_ubrad}.
\item[The Sullivan diagrams $\SD$] The space of Sullivan diagrams are the quotient of $\MFatad$ by the equivalence relation of slides away from the admissible boundary. It is defined in Definition \ref{def_sd}.
\end{description}

We will focus on the bottom square; that is, the relations between radial slit configurations, admissible metric fat graphs and their compactifications. Our main result is:

\begin{theorem}\label{thm_main} We define a space $\Radt$ and maps \eqref{cor_pi1homeq}, \eqref{cor_pi2homeq} and \eqref{prop_sdubradhomeo}
such that there is a commutative square
\[\begin{tikzcd} \Rad \arrow{d} & \Radt \lar{\simeq}[swap]{(\ref{cor_pi1homeq})} \rar{(\ref{cor_pi2homeq})}[swap]{\simeq} & \MFatad \arrow[two heads]{dd} \\[-5pt]
	\bRad \dar[two heads]{\simeq}[swap]{(\ref{lem_bradubrad})} &  &  \\[-5pt]
	\UbRad \arrow{rr}{\cong}[swap]{(\ref{prop_sdubradhomeo})} & & \SD.\end{tikzcd}\]
Furthermore, all maps that are decorated by $\simeq$ are homotopy equivalences and the map decorated by $\cong$ is a cellular homeomorphism.\end{theorem}

There exist other combinatorial models related to the moduli space of cobordisms which are not discussed in this paper. We will describe six such models in the following remarks.

\begin{remark} To describe an action of the chains of the moduli space of surfaces on the Hochschild homology of $\Acat_\infty$-Frobenius algebras, Costello constructed a chain complex that models the homology of the moduli space (\cite{costellorg, tcft}). In \cite{wahlwesterland}, Wahl and Westerland described this chain complex in terms of fat graphs with two types of vertices, which they called \emph{black and white fat graphs}. There is an equivalence relation of black and white graphs given by slides away from the white vertices. The quotient chain complex is the cellular chain complex of $\SD$.  Furthermore, in \cite{egas} it was shown that $\MFatad$ has a quasi-cell structure of which \emph{black and white fat graphs} is its cellular complex and the quotient map to $\SD$ respects this cell structure.
\end{remark}

\begin{remark}In \cite{cohengodin} Cohen and Godin defined \emph{Sullivan chord diagrams} of genus $g$ with $p$ incoming and $q$ outgoing boundary components, which were also used in \cite{felixthomasgorenstein}. These are fat graphs obtained from glueing trees to circles. These fit together into a space $\mathcal{CF}(g;p,q)$ which is a subspace of $\MFatad$.  They are \emph{not} the same as Sullivan diagrams as in Definition \ref{def_sd}, though they do admit a map to $\SD$. The space of metric chord diagrams is not homotopy equivalent to moduli space, see Remark 3 of \cite{godin}.\end{remark}

\begin{remark} In \cite{poirier}, Poirier defined a space $\overline{SD}(g,k,l)/{\sim}$ of \emph{string diagrams modulo slide equivalence} of genus $g$ with $k$ incoming and $l$ outgoing boundary components and more generally she defined \emph{string diagrams with many levels modulo slide equivalence} $\overline{LD}(g,k,l)/{\sim}$. Proposition 2.3 of \cite{poirier} says that $\overline{SD}(g,k,l)/{\sim} \simeq \overline{LD}(g,k,l)/{\sim}$. She also defined a subspace $SD(g,k,l)$ of $\overline{SD}(g,k,l)$. Both $\overline{SD}(g,k,l)$ and $SD(g,k,l)$ are subspaces of $\MFatad$ and by counting components one can see that these inclusions can not be homotopy equivalences.  However, there is an induced map $\overline{SD}(g,k,l)/{\sim} \to \SD$ which is a homeomorphism.
\end{remark}

\begin{remark}
In \cite{drummond_poirier_rounds} Drummond-Cole, Poirier, and Rounds defined a space of \emph{string diagrams} $SD$ which generalized the spaces of chord diagrams constructed in \cite{poirier}. They conjectured that this space is homotopy equivalent to the moduli space of Riemann surfaces. There is an embedding $SD \cof \MFatad$ but it is not clear this is a homotopy equivalence. Furthermore, there is an equivalence relation $\sim$ on $SD$, which is not discussed in their paper, and they conjectured that $SD/{\sim}$ is homotopy equivalent to the harmonic compactification.
\end{remark}

\begin{remark}
Following the ideas of Wahl, Klamt constructed a chain complex of \emph{looped diagrams} denoted $l\cD$  in \cite{Klamt_comfrob}.  This complex gives operations on the Hochschild homology of commutative Frobenius algebras. Moreover, she gave a chain map from cellular complex of the space of Sullivan diagrams to looped diagrams. However, a geometric interpretation of a space underlying the complex $l\cD$ and its possible relation to moduli space is still unknown.
\end{remark}

\begin{remark}
In \cite{kaufmann_sullivan}, Kaufman described a space of open-closed Sullivan diagrams $\mathrm{Sull^{c/o}_1}$ in terms of arcs embedded in a surface. The closed part, $\mathrm{Sull^c_1}$, is a space whose points correspond to weighted families of embedded arcs in the surface that flow from the incoming boundary to the outgoing boundary. This space has a natural cell structure and there is a cellular homeomorphism $\mathrm{Sull}^c_1 \stackrel{\cong}{\longrightarrow} \SD$ \cite[Remark 2.12]{wahlwesterland} .
\end{remark}

\subsection{Applications of these models} 
We will next explain two of the applications of combinatorial models for moduli spaces. 

\subsubsection{Explicit computations of the homology of moduli spaces} Combinatorial models provide cell decompositions for moduli spaces, allowing for explicit computations of the (co)homology groups of moduli spaces using cellular (co)homology. Instead of studying $\cM_g(n,m)$, it is more convenient to study the closely related moduli space $\cM^{1,n}_g$ of surfaces of genus $g$ with one parametrized boundary component and $n$ permutable punctures. There are variations of $\Rad$ and $\MFatad$ that are models for $\cM^{1,n}_g$.

Much is known about the homology of $\cM_g^{1,n}$ and much is unknown about it. Harer stability tells us $H_*(\cM_g^{1,n})$ stabilizes as $g \to \infty$ \cite{Harer_stable, wahlmcg}; as a consequence of homological stability for configuration spaces it also stabilizes as $n \to \infty$. The Madsen-Weiss theorem gives the stable homology \cite{MadsenWeiss,galatius_modp} (see \cite{BodigheimerTillmann} for the increasing the number of punctures). Less is known outside of the stable range; explicit computations of $H_*(\cM_g^{1,n})$ for low $g$ and $n$ can help inform and test conjectures about the homology of moduli spaces.
 
The computation of the homology of moduli spaces using radial slit configurations, or the closely related parallel slit configurations, is a long-term project of B\"odigheimer and his students. The first example of this is Ehrenfried's thesis \cite{ehrenfried} where he computes $\cM_2^{1,0}$. See \cite{abhau} for computations of the integral homology of $\cM_g^{1,n}$ for $2g+n \leq 5$ using parallel slits. An example of an explicit computation using fat graphs is \cite{godinunstable}, in which Godin computes the integral homology of $\cM_g^{1,0}$ for $g=1,2$ and $\cM_{g}^{2,0}$ for $g=1$. 

\subsubsection{Two-dimensional field theories, in particular string topology} 
Combinatorial models of moduli spaces have been an important tool in the study of two-dimensional field theories. Two applications are Kontsevich's proof of the Witten conjecture \cite{kontsevich}, and Costello's classification of topological conformal field theories \cite{tcft}. More concretely, combinatorial models for the moduli space of cobordisms play a role in the construction of string operations; these are operations $H_*(\cM_g(n,m);\cL^{\otimes d}) \otimes H_*(LM)^{\otimes n} \to H_*(LM)^{\otimes m}$ for compact oriented manifolds $M$. Chas and Sullivan thought of the pair of pants cobordism as a figure-eight graph \cite{chassullivan}, and many of the constructions of string operations since have used graphs. An important example is Godin's work \cite{godin}, which uses $\Fatad$. Using Costello's model for moduli space together with a Hochschild homology model for $H^*(LM)$, Wahl and Westerland \cite{wahlwesterland, wahluniversal} not only constructed string operations, but showed that these factor through $\SD$. One can also use radial slit configurations to construct string operations.

A problem in string topology is that there are many constructions but few comparisons between them. The critical graph equivalence of Section \ref{sec_critmap} may help to compare constructions involving fat graphs and Sullivan diagrams to those involving radial slit configurations and the harmonic compactification.

\subsection{Outline of paper} In Sections \ref{sec_radialslits} and \ref{sec_fatgraphs} we define radial slit configurations, fat graphs and their compactifications in detail. In Section \ref{sec_critmap} we use the critical graph of a radial slit configuration to construct a zigzag of homotopy equivalences between $\Rad$ and $\MFatad$. In Section \ref{sec_sdharmonic} we show this descends to a homeomorphism between $\UbRad$ and $\SD$.

\subsection{Acknowledgments} This paper grew out of discussions at the String Topology and Related Topics at the Center for Symmetry and Deformation at the University of Copenhagen and was finished during the Hausdorff Trimester Program on Homotopy Theory, Manifolds, and Field Theories. The authors would like to thank Carl-Friedrich B\"{o}digheimer and Nathalie Wahl for helpful conversations and comments. The authors would also like to thank the anonymous referees for helpful comments.  DES was supported by the Danish National Research Foundation through the Centre for Symmetry and Deformation (DNRF92). AK was supported by a William R. Hewlett Stanford Graduate Fellowship, Department of Mathematics, Stanford University.

%% file: radialslits.tex
\subsection{The definition} In this subsection we introduce B\"odigheimer's radial slit configuration model for the moduli space of two-dimensional cobordisms with non-empty incoming and outgoing boundary. All material in this subsection is due to B\"odigheimer, and references include \cite{bodigheimerold}, \cite{bodigheimer}, \cite{abhau}, \cite{ebertradial} and \cite{bodighpreprint}. The last one is of particular interest, as it describes in a related setting an elegant alternative to the construction below, using subspaces of bar complexes associated to symmetric groups. It however leads to a different compactification of moduli space than the harmonic compactification, so we use \cite{bodigheimer}.

\subsubsection{Spaces of radial slit configurations} Before giving a definition of the radial slit configuration space $\Rad$, we explain how to arrive at it from the perspective of building cobordisms by glueing annuli along cuts. The reader may prefer to skip this motivation and go directly to Definition \ref{defpreconf}.

The simplest cobordism with non-empty incoming and outgoing boundary is the cylinder, with one incoming and one outgoing boundary component. Using the theory of harmonic functions, one sees each annulus is conformally equivalent to one of the following annuli for $R \in (\frac{1}{2\pi},\infty)$ \cite[Corollary 2.13]{hamenstadt} (the reason for the choice of $\frac{1}{2\pi}$ is to facilitate comparison with fat graphs later on):
\[\bA_R \coloneqq \left\{z \in \bC\,\middle|\,\frac{1}{2\pi} \leq |z| \leq R\right\}.\]
We take these as our basic building blocks. Each of them has an inner boundary $\partial_\mr{in} \bA_R = \{z \in \bC\,|\,|z| = \frac{1}{2\pi}\}$ and an outer boundary $\partial_\mr{out} \bA_R = \{z \in \bC\,|\,|z| = R\}$. They come with a canonical metric, as subsets of the complex plane.

To construct a cobordism with $n$ incoming boundary components, we start with an ordered disjoint union of $n$ annuli $\bA^{(i)}_{R_i}$, whose inner boundaries will be the incoming boundary of our cobordism. Next we make cuts radially inward from the outer boundaries of the annuli. Such cuts are uniquely specified by points $\zeta \in \sqcup_{i=1}^n \bA^{(i)}_{R_i}$, which we will call slits. They need not be distinct. As will become clear, the number of slits must always be an even number $2h$ and we thus number them $\zeta_1,\ldots,\zeta_{2h}$. For a total genus $g$ cobordism with $n$ incoming and $m$ outgoing boundary components we need $2h = 2(2g-2+n+m)$ slits.

We want to glue the different sides of the cuts back together. To get a metric on the surface from the metric on the cut annuli, the two cuts that we glue together must be of the same length. To get an orientation on the surface from the orientations on the cut annuli, we must glue a side clockwise from a cut to a side counterclockwise from a cut. To avoid singularities, if one side of the cut corresponding to $\zeta_i$ is glued to a side of the cut corresponding to $\zeta_j$, the same must be true for the other two sides. Thus our gluing procedure is described by a pairing on $\{1,\ldots,2h\}$, encoded by a permutation 
\[\lambda \colon  \{1,\ldots,2h\} \to \{1,\ldots,2h\}\]
consisting of $h$ cycles of length 2. We should demand that if $\zeta_i$ lies on the annulus $\bA^{(j)}_{R_j}$ and $\zeta_{\lambda(i)}$ lies on the annulus $\smash{\bA^{(j')}_{R_{j'}}}$, then $R_j - |\zeta_i| = R_{j'} - |\zeta_{\lambda(i)}|$. See Figure \ref{figradialglue} for an example.

\begin{figure}[h!]
  \centering
    \begin{tikzpicture}[scale=.9]
    	\node at (0,0) {$1$};
    	\draw (90:1.5cm) -- (90:1.6cm);
    	\node at (90:1.6cm) [above] {$1$};
    	\draw (-90:1.5cm) -- (-90:1.6cm);
    	\node at (-90:1.6cm) [below] {$2$};
    	
    	\draw [black!50!white,densely dotted,thick] (0:0.5cm) -- (0:0.7cm);
    	\draw [blue,thick] (0:0.7cm) -- (0:1.5cm);
    	
    	\draw [black!50!white,densely dotted,thick] (180:0.5cm) -- (180:0.7cm);
    	\draw [blue,thick] (180:0.7cm) -- (180:1.5cm);
    	
    	\draw[decoration={markings, mark=at position 0 with {\arrow{<}}}, 	postaction={decorate}] (0,0) circle (1.5cm);
    	\draw[decoration={markings, mark=at position 0 with {\arrow{<}}},	postaction={decorate}]  (0,0) circle (.5cm);
    	
		\node at (1.5,-2.4) [left] {outgoing boundary};
		\node at (1.5,2.4) [left] {incoming boundary};
		\draw [->] (1.5,2.4) to[out=0,in=45] (45:0.6);
		\draw [->] (1.5,-2.4) to[out=0,in=-45] (-45:1.6);
		
		\begin{scope}[xshift=4cm]		
		\draw [black!50!white,densely dotted,thick] (0:0.5cm) -- (0:0.7cm);
		\draw [black!50!white,densely dotted,thick] (180:0.5cm) -- (180:0.7cm);
		\draw [black!10!white,line width=2mm,yshift=1mm] (0:0.7cm) -- (5:1.5);
		\draw [black!10!white,line width=2mm,yshift=1mm] (180:0.7cm) -- (175:1.5);
		\draw [blue,thick] (0:0.7cm) -- (5:1.5);
		\draw [blue,thick] (0:0.7cm) -- (-5:1.5);
		\draw [blue,thick] (180:0.7cm) -- (185:1.5);
		\draw [blue,thick] (180:0.7cm) -- (175:1.5);

    	\draw (0,0) circle (.5cm);
		\draw (0,0) ++(5:1.5) arc (5:175:1.5);
		\draw (0,0) ++(-5:1.5)  arc (-5:-175:1.5);
		
		\node at (0,0) {$1$};
		\draw (90:1.5cm) -- (90:1.6cm);
		\node at (90:1.6cm) [above] {$1$};
		\draw (-90:1.5cm) -- (-90:1.6cm);
		\node at (-90:1.6cm) [below] {$2$};
		\end{scope}
		
		\begin{scope}[xshift=7cm]
		\draw [xscale=.7,decoration={markings, mark=at position 0.5 with {\arrow{<}}},	postaction={decorate}] (0,0.5) arc (90:270:.5);
		\draw [xscale=.7,densely dotted] (0,-0.5) arc (-90:90:.5);
		\node at (0,0) {$1$};
		
		\begin{scope}[xshift=1.5cm,yshift=1.5cm]
		\draw [xscale=.7,decoration={markings, mark=at position 0.5 with {\arrow{<}}},	postaction={decorate}] (0,0) circle (.8cm);
		\draw (90:0.8cm) -- (90:0.9cm);
		\node at (90:0.9cm) [above] {$1$};
		\end{scope}
		
		\begin{scope}[xshift=1.7cm,yshift=-1cm]
		\draw [xscale=.7,decoration={markings, mark=at position 0.5 with {\arrow{<}}},	postaction={decorate}] (0,0) circle (.8cm);
		\draw (-90:0.8cm) -- (-90:0.9cm);
		\node at (-90:0.9cm) [below] {$2$};
		\end{scope}
		
		\draw (1.5,{1.5+.8}) to[out=180,in=0] (0,.5);
		\draw (1.7,{-1+-.8}) to[out=180,in=0] (0,-.5);
		\draw [thick,blue] (1.5,{1.5-.8}) to[out=180,in=90] (1,.25) to[out=-90,in=180] (1.7,{-1+.8});
		\draw [black!25!white,densely dotted,thick] ({-.7*.5},0) to[out=0,in=90] (1,.25) to[out=-90,in=180] ({.7*.5},0);
		\end{scope}
    \end{tikzpicture}
  \caption{An example of constructing a cobordism by cutting and glueing slits in annuli. We start with the annulus on the left, cut along the blue lines to obtain the middle figure, and finally glue both the gray sides and the white sides of the cuts to get the cobordism on the right. In this simple example the pairing $\lambda$ and the successor permutation $\omega$ are uniquely determined.}
  \label{figradialglue}
\end{figure}

However, several problematic situations could occur. Firstly, if two slits $\zeta_i$ and $\zeta_j$ lie on the same \emph{radial segment}, by definition a subset of the annulus $\smash{\bA^{(j)}_{R_j}}$ of the form 
\[\{z \in \bA^{(j)}_{R_j} \,|\,\arg(z) = \theta\} \qquad \text{for some $\theta$,}\] then our cutting and glueing procedure is not well-defined: we need to keep track of whether $\zeta_i$ lies clockwise or counterclockwise from $\zeta_j$. To do this we include the data of a successor permutation 
\[\omega \colon \{1,\ldots,2h\} \to \{1,\ldots,2h\}.\]
This has $n$ cycles, corresponding to the $n$ annuli, and we should demand that each cycle contains the numbers of the slits in one of the annuli and is compatible with the weak cyclic ordering on these coming from the argument of the slits. The successor permutation keeps track of the fact that when two slits coincide, one lies actually ``infinitesimally counterclockwise'' from the other. See Figure \ref{figsuccessor}.

\begin{figure}[h]
  \centering
    \begin{tikzpicture}[scale=.9]
      	\node at (0,0) {$1$};
  \draw (90:1.5cm) -- (90:1.6cm);
  \node at (90:1.6cm) [above] {$1$};

  \draw [red,thick] (-30:1cm) -- (-30:1.5cm);
  \node at (-30:1cm) [left] {\footnotesize $\zeta_4$};
  \draw [blue,thick] (-100:1.2cm) -- (-100:1.5cm);
  \node at (-100:1.2cm) [right] {\footnotesize $\zeta_3$};
  \draw [red,thick] (150:1cm) -- (150:1.5cm);
  \node at (150:1cm) [right] {\footnotesize $\zeta_2$}; 
  \draw [blue,thick] (150:1.2cm) -- (150:1.5cm);
  \node at (150:1.2cm) [below] {\footnotesize $\zeta_1$}; 
  
  \draw[decoration={markings, mark=at position 0 with {\arrow{<}}}, 	postaction={decorate}] (0,0) circle (1.5cm);
  \draw[decoration={markings, mark=at position 0 with {\arrow{<}}},	postaction={decorate}]  (0,0) circle (.5cm);
  
  \draw [->] (25:1.8) -- (25:2.6);
  \node at (35:2.8cm) [fill=white,above left] {\footnotesize $\omega = (1\,2\,3\,4)$};
  \draw [->] (-25:1.8) -- (-25:2.6);
  \node at (-35:2.8cm) [fill=white,below left] {\footnotesize $\omega = (2\,1\,3\,4)$};
  
  \begin{scope}[xshift=4cm,yshift=1.9cm]
    	\node at (0,0) {$1$};
  	  \draw (90:1.5cm) -- (90:1.6cm);
  	  \node at (90:1.6cm) [above] {$1$};

  	  \draw [red,thick] (-30:1cm) -- (-30:1.5cm);
  	  \draw [blue,thick] (-100:1.2cm) -- (-100:1.5cm);
  	  \draw [red,thick] (150:1cm) -- (150:1.5cm);
  	  \draw [blue,thick,yshift=.04cm,xshift=.02cm] (150:1.2cm) -- (150:1.5cm);
  	  
  	  \draw[decoration={markings, mark=at position 0 with {\arrow{<}}}, 	postaction={decorate}] (0,0) circle (1.5cm);
  	  \draw[decoration={markings, mark=at position 0 with {\arrow{<}}},	postaction={decorate}]  (0,0) circle (.5cm);
  \end{scope} 
  
    \begin{scope}[xshift=4cm,yshift=-1.9cm]
  \node at (0,0) {$1$};
  \draw (90:1.5cm) -- (90:1.6cm);
  \node at (90:1.6cm) [above] {$1$};

  \draw [red,thick] (-30:1cm) -- (-30:1.5cm);
  \draw [blue,thick] (-100:1.2cm) -- (-100:1.5cm);
  \draw [red,thick] (150:1cm) -- (150:1.5cm);
  \draw [blue,thick,yshift=-.04cm,xshift=-.02cm] (150:1.2cm) -- (150:1.5cm);
  
  \draw[decoration={markings, mark=at position 0 with {\arrow{<}}}, 	postaction={decorate}] (0,0) circle (1.5cm);
  \draw[decoration={markings, mark=at position 0 with {\arrow{<}}},	postaction={decorate}]  (0,0) circle (.5cm);
  \end{scope} 
    \end{tikzpicture}
  \caption{An example of a radial slit preconfiguration with two slits on the same radial segment; $\zeta_1$ is the shorter blue slit and $\zeta_2$ is the longer red slit. The successor permutation $\omega$ allows us to think of $\zeta_1$ as either infinitesimally clockwise or counterclockwise from $\zeta_2$.}
  \label{figsuccessor}
\end{figure}

This is not enough, because if all slits on an annulus lie on the same radial segment we can only deduce the ordering of the slits up to a cyclic permutation. To amend this, we add additional data; the angular distance $r_i \in [0,2\pi]$ in counterclockwise direction from $\zeta_i$ to $\zeta_{\omega(i)}$. In almost all cases one can deduce this from the locations of the $\zeta_i$ and $\omega$, but in the case where all slits on an annulus lie on the same radial segment, one of them will have to be $r_i = 2\pi$, while the others will have to be $r_j = 0$. This allows one to determine the ordering of the slits, since the slit $\zeta_i$ with $r_i = 2\pi$ should be first in clockwise direction from the angular gap between the slits.

We have almost described enough data to construct a cobordism. We can build a possibly degenerate surface, which has among its boundary components the inner boundaries of the annuli. Since we wanted $m$ outgoing boundary components, we restrict to the subset of data that gives us $m$ boundary components in addition to these inner boundaries of annuli. The inner boundaries of the annuli come with a canonical parametrization, but the outer ones do not come with such a parametrization. Because they already have a canonical orientation coming from the orientation of the outer boundary of the annuli, it suffices to add one point $P_i$ in each of them, $m$ in total. Thus, we need to include these new parametrization points in $\omega$ and the $r_i$'s. To do this, we write $\xi_i = \zeta_i$ for $1 \leq i \leq 2h$ and $\xi_{2h+i} = P_i$ for $1 \leq i \leq m$, and expand our definition of $\omega$ to a permutation $\overline{\omega} \in \mathfrak{S}_{2h+m}$ and add additional $r_{2h+i} \in [0,2\pi]$ for $1 \leq i \leq m$. It is also convenient to extend the definition of $\lambda$ to a permutation $\overline{\lambda} \in \mathfrak{S}_{2h+m}$ by setting $\overline{\lambda}(2h+i) = 2h+i$ for $1 \leq i \leq m$.

Now we can state the definition of a radial slit configuration by collecting all the above data, identifying those configurations yielding the same conformal surface, and discarding those configurations yielding degenerate surfaces.  Actually, it is only necessary to consider configurations with a fixed outer radius; we will say more on this towards the end of the section.  Therefore, from now on we take $\vec{R}=(R,R,\ldots,R)$ and $R=\frac{1}{2\pi}+\frac{1}{2}$ unless stated otherwise. This choice of outer radius is arbitrary, but it makes the connection with metric fat graphs cleanest.

\begin{definition}\label{defpreconf} The space of \emph{possibly degenerate radial slit preconfigurations} $\bprad_h(n,m)$ is the subspace of 
	\[L = (\vec{\xi},\overline{\lambda},\overline{\omega},\vec{r}) \in (\sqcup_{j=1}^n \bC)^{2h+m} \times \mathfrak S_{2h+m} \times \mathfrak S_{2h+m} \times [0,2\pi]^{2h+m}\]
with the following properties. For notation, let $\zeta_i \coloneqq \xi_i$ for $1\leq i\leq 2h$ and $P_i \coloneqq \zeta_{2h+i}$ for $1\leq i\leq m$. Then we have:	
\begin{itemize}
	\item $\vec{\zeta}\in (\bigsqcup_{j=1}^n \bC)^{2h}$ are the endpoints of the \emph{slits},
	\item $\vec{P}\in (\bigsqcup_{j=1}^n \bC)^{m}$ are the \emph{parametrization points},
	\item $\overline{\lambda}\in \mathfrak S_{2h}$  is the \emph{extended slit pairing},
	\item $\overline{\omega}\in \mathfrak S_{2h+m}$ is the \emph{extended successor permutation},
	\item $\vec{r}\in [0,2\pi]^{2h+m}$ are the \emph{angular distances}.
\end{itemize} 
These are subject to six conditions:
	\begin{enumerate}[(i)]
		\item Each slit $\zeta_i$ lies in $\bigsqcup_{j=1}^n \bA^{(j)}_{R} \subset \bigsqcup_{j=1}^n \bC$ and each parametrization point $P_i$ lies in $\bigsqcup_{j=1}^n \partial_{out}\bA^{(j)}_{R} $.
		\item The extended slit pairing $\overline{\lambda}$ consists of $h$ 2-cycles and $m$ 1-cycles. The latter are given by $2h+i$ for $1 \leq i \leq m$. We demand for all $1 \leq i \leq 2h$ we have that $|\zeta_i|=|\zeta_{\overline{\lambda}(i)}|$.
		\item The successor permutation $\overline{\omega}$ consists of a disjoint union of $n$ cycles and these cycles consist exactly of the indices of the $\xi_i$ lying on each of the annuli. We demand that the permutation action of $\overline{\omega}$ on these $\xi_i$ preserves the weakly cyclic ordering which comes from the argument (as usual taken in counterclockwise direction).
		\item The \emph{boundary component permutation} $\overline{\lambda} \circ \overline{\omega}$ consists of $m$ cycles. We will see its cycles correspond to the outgoing boundary components.
		\item We demand that $P_i$ lies in the subset $O_i$ of $\bigsqcup_{j=1}^n \partial_\mr{out} \bA^{(j)}_R$ which we will now define. The $m$ cycles of $\overline{\lambda} \circ \overline{\omega}$ allow one to write the outer boundaries of the annuli as a union of $m$ subsets, overlapping only in isolated points. We demand that each of these contains exactly one $P_i$ and denote that subset by $O_i$. To be precise, each $O_i$ is the union of the parts in the outer boundary between the radial segments $\xi_j$ and $\xi_{\overline{\omega}(j)}$ in counter-clockwise direction, for all $j$ in a cycle of $\overline{\lambda} \circ \overline{\omega}$.
		\item The angular distances $r_i$ must be compatible with the location of the $\xi_i$ and the successor permutation $\overline{\omega}$ in the following sense. If $\xi_i$ does not lie on an annulus with all slits and parametrization points coinciding, then $r_i$ is equal to the angular distance in counterclockwise direction from $\xi_i$ to $\xi_{\overline{\omega}(i)}$. If $\xi_i$ lies on an annulus with all slits and parametrization points coinciding, then $r_i$ is equal to either $0$ or $2\pi$ and exactly one $\xi_j$ on that annulus has $r_j = 2\pi$.
	\end{enumerate}
\end{definition}

In terms of the previous notation, $\omega$ and $\lambda$ are obtained from $\overline{\omega}$ and $\overline{\lambda}$ by deleting the elements $2h+i$ for $1\leq i \leq m$ from the cycles.

\begin{figure}[h]
  \centering
      \begin{tikzpicture}[scale=.9]
  \node at (0,0) {$1$};
  \draw (90:1.5cm) -- (90:1.6cm);
  \node at (90:1.6cm) [above] {$1$};
  \draw (-90:1.5cm) -- (-90:1.6cm);
  \node at (-90:1.6cm) [below] {$2$};
  
  \draw [blue,thick] (0:1cm) -- (0:1.5cm);
  \node at (0:1cm) [below] {\footnotesize $\zeta_1$};
  \draw [blue,thick] (180:1cm) -- (180:1.5cm);
  \node at (180:1cm) [below] {\footnotesize $\zeta_2$}; 
  
  \draw [line width=1.5mm,red!40!white](1.5,0) arc (0:180:1.5cm);
  \draw [line width=1.5mm,blue!40!white,dashed](1.5,0) arc (0:-180:1.5cm);
  \draw[decoration={markings, mark=at position 0 with {\arrow{<}}}, 	postaction={decorate}] (0,0) circle (1.5cm);
  \draw[decoration={markings, mark=at position 0 with {\arrow{<}}},	postaction={decorate}]  (0,0) circle (.5cm);

  \node at (-1,1) [left] {\parbox{5cm}{\footnotesize successor permutation $\omega = (1\,2)$ \\
  angular distances $r_1=r_2=\pi$}};
  \node at (2,0) [right] {\parbox{4cm}{\footnotesize parametrization points in each outgoing boundary component}};
  \draw [->] (3,.5) to[out=90,in=0] (.2,1.8);
  \draw [->] (3,-.5) to[out=-90,in=0] (.2,-1.8);
  
  \node at (0,-2.5) [below] {\footnotesize labeled incoming boundary};
  \draw [->,line width=1mm,white] (1,-2.5) to[out=90,in=-60] (-60:.6);
  \draw [->] (1,-2.5) to[out=90,in=-60] (-60:.6);
  \node at (-2,-1) [left] {\parbox{4cm}{\footnotesize outer boundary of annulus divided into two outgoing boundary components (here solid and dashed)}};
  \node at (0,2.5) [above] {\footnotesize radial slits with pairing $(1\,2)$};
  \draw [->,line width=1mm,white] (.2,2.5) to[out=-90,in=80] (5:1.03);
  \draw [->,line width=1mm,white] (-.2,2.5) to[out=-90,in=100] (175:1.03);
  \draw [->] (.2,2.5) to[out=-90,in=80] (5:1.1);
  \draw [->] (-.2,2.5) to[out=-90,in=100] (175:1.1);
  \end{tikzpicture}
  \caption{The configuration of Figure \ref{figradialglue} with all its data pointed out.}
  \label{figradialglossary}
\end{figure}

We now give a construction of a possibly degenerate cobordism $S(L)$ for a preconfiguration $L$. To do so, we first define the sector space $\overline{\Sigma}(L)$, the pieces used in the glueing construction. We slightly depart from our informal discussion by making cuts from the outer boundary to the inner boundary of the annuli and reglueing these later. See Figure \ref{figradialsectors} for examples of the different types of sectors.

\begin{definition}Let $l$ be the number of annuli containing no elements of $\vec{\xi}$. Then $\overline{\Sigma}(L)$ will have $2h+m+l$ components $F_i$ for $1 \leq i \leq 2h+m+l$. These come in four types:
	\begin{description}
		\item[Ordinary sectors] If $\arg(\xi_i) \neq \arg(\xi_{\overline{\omega}(i)})$ and $\xi_i$ lies on the $j$th annulus $\bA^{(j)}_{R}$, then we set 
		\[F_i = \{z \in \bA^{(j)}_{R}\,|\,\arg(\xi_i) \leq \arg(z) \leq \arg(\xi_{\overline{\omega}(i)})\}.\]
		\item[Thin sectors] If $\arg(\xi_i) = \arg(\xi_{\overline{\omega}(i)})$, $r_i = 0$ and $\xi_i$ lies on the $j$th annulus $\bA^{(j)}_{R}$, then we set 
		\[F_i = \{z \in \bA^{(j)}_{R}\,|\,\arg(\xi_i) = \arg(z)\}.\]
		\item[Full sectors] If $\arg(\xi_i) = \arg(\xi_{\overline{\omega}(i)})$, $r_i = 2\pi$ and $\xi_i$ lies on the $j$th annulus $\smash{\bA^{(j)}_{R}}$, then we set $F_i$ to be the annulus $\bA^{(j)}_{R}$ cut open along the segment $\arg(z) = \arg(\xi_i)$, with that segment doubled so that it is homeomorphic to a closed rectangle.
		\item[Entire sectors] If the $j$th annulus $\bA^{(j)}_{R}$ does not contain any elements of $\vec{\xi}$ and is $j'$th in the induced ordering on the $r$ annuli that do not contain any slits, we set $F_{2h+m+j'} = \bA^{(j)}_{R}$.
	\end{description}
\end{definition}

\begin{figure}[h]
  \centering
    \begin{tikzpicture}
    	\draw (0,0) ++ (40:1cm) arc (40:140:1cm);
    	\draw (0,0) ++ (40:3cm) arc (40:140:3cm);
    	\draw (40:1cm) -- (40:3cm);
    	\draw (140:1cm) -- (140:3cm);
    	\draw [blue,line width=.5mm] (40:1cm) -- (40:2.2cm);
    	\draw [blue,line width=.5mm,dashed] (40:2.2cm) -- (40:3cm);
    	\draw [red,line width=.5mm] (140:1cm) -- (140:1.6cm);
		\draw [red,line width=.5mm,dashed] (140:1.6cm) -- (140:3cm);
    	\draw (0,0) ++ (36:2.2cm) arc (36:44:2.2cm);
    	\draw (0,0) ++ (144:1.6cm) arc (144:136:1.6cm);
    	\node [blue,xshift=-.2cm,yshift=.2cm] at (40:1.6cm) {$\alpha^-$};
    	\node [blue,xshift=-.2cm,yshift=.2cm] at (40:2.6cm) {$\beta^-$};
    	\node [red,xshift=-.2cm,yshift=-.2cm] at (140:1.4cm) {$\alpha^+$};
    	\node [red,xshift=-.2cm,yshift=-.2cm] at (140:2.3cm) {$\beta^+$};
    	\node at (0,0) {ordinary};
    	
    	\begin{scope}[xshift=6cm]
    		\draw (140:1cm) -- (140:3cm);
    		\draw [xshift=.1cm,yshift=.1cm] (140:1cm) -- (140:3cm);
    		\draw (140:1cm) -- ++(.1cm,.1cm);
    		\draw (140:3cm) -- ++(.1cm,.1cm);
    		\draw (0,0) ++ (144:1.6cm) arc (144:136:1.6cm);
     		\draw [xshift=.1cm,yshift=.1cm] (0,0) ++ (144:2.2cm) arc (144:136:2.2cm);
    		\draw [red,line width=.5mm] (140:1cm) -- (140:1.6cm);
			\draw [red,line width=.5mm,dashed] (140:1.6cm) -- (140:3cm);
    		\node [red,xshift=-.2cm,yshift=-.2cm] at (140:1.4cm) {$\alpha^+$};
    		\node [red,xshift=-.2cm,yshift=-.2cm] at (140:2.3cm) {$\beta^+$};
      		\draw [blue,line width=.5mm,xshift=.1cm,yshift=.1cm] (140:1cm) -- (140:2.2cm);
 			\draw [blue,line width=.5mm,dashed,xshift=.1cm,yshift=.1cm] (140:2.2cm) -- (140:3cm);
    		\node [blue,xshift=.4cm,yshift=.4cm] at (140:1.6cm) {$\alpha^-$};
			\node [blue,xshift=.4cm,yshift=.4cm] at (140:2.6cm) {$\beta^-$};
    		\node at (-1.5,0) {thin};
    	\end{scope}
    
    	\begin{scope}[xshift=4.5cm,yshift=-2cm]
    		\draw [yscale=.7] (0,0) circle (.5cm);    		
    		\draw [yscale=.7] (0,0) circle (2cm);
    		\node at (0,-2) {entire};
    	\end{scope}
    	
    	\begin{scope}[yshift=-2cm,yscale=.35,xscale=.5]
    	\draw (0,0) ++ (140:1cm) arc (140:490:1cm);
    	\draw (0,0) ++ (140:4cm) arc (140:490:4cm);
    	\draw (490:1cm) -- (490:4cm);
    	\draw (140:1cm) -- (140:4cm);
    	\draw [red,line width=.5mm] (490:1cm) -- (490:3.2cm);
    	\draw [red,line width=.5mm,dashed] (490:3.2cm) -- (490:4cm);
    	\draw [blue,line width=.5mm] (140:1cm) -- (140:1.6cm);
    	\draw [blue,line width=.5mm,dashed] (140:1.6cm) -- (140:4cm);
    	\draw (0,0) ++ (486:3.2cm) arc (486:494:3.2cm);
    	\draw (0,0) ++ (144:1.6cm) arc (144:136:1.6cm);
    	\node [red,xshift=.25cm,yshift=.25cm] at (490:2cm) {$\alpha^-$};
    	\node [red,xshift=.3cm,yshift=.35cm,fill=white] at (490:3.2cm) {$\beta^-$};
    	\node [blue,xshift=-.2cm,yshift=-.2cm] at (140:1.4cm) {$\alpha^+$};
    	\node [blue,xshift=-.2cm,yshift=-.2cm] at (140:2.8cm) {$\beta^+$};
    	\node at (0,-{2/.35}) {full};
    	\end{scope}
    \end{tikzpicture}
  \caption{Examples of the different types of radial sectors with subsets $\alpha^\pm$ and $\beta^\pm$.}
  \label{figradialsectors}
\end{figure}

The surface $\Sigma(L)$ underlying the cobordism $S(L)$ will be obtained as a quotient space of the sector space by an equivalence relation that makes identifications on the boundary of the sectors. We next define the subsets involved in those identifications.

\begin{definition}\label{def_alphabeta} If $F_i$ is an ordinary or thin sector corresponding to the element $\xi_i$ on the $j$th annulus $\bA^{(j)}_{R}$, then we define the following subspaces of $F_i$:
\begin{align*}\alpha^+_i &\coloneqq \{z \in \bA^{(j)}_{R}\,|\,\arg(z) = \arg(\xi_{\overline{\omega}(i)}) \text{ and } |z| \leq |\xi_{\overline{\omega}(i)}|\}, \\
\alpha^-_i &\coloneqq \{z \in \bA^{(j)}_{R}\,|\,\arg(z) = \arg(\xi_{i}) \text{ and } |z| \leq |\xi_{i}|\},\\
\beta^+_i &\coloneqq \{z \in \bA^{(j)}_{R}\,|\,\arg(z) = \arg(\xi_{\overline{\omega}(i)}) \text{ and } |z| \geq |\xi_{\overline{\omega}(i)}|\}, \\
\beta^-_i &\coloneqq \{z \in \bA^{(j)}_{R}\,|\,\arg(z) = \arg(\xi_{i}) \text{ and } |z| \geq |\xi_{i}|\}.\end{align*}
If $F_i$ is a full sector then our definitions are different, because the two radial segments in the boundary have the same argument. Let $S^+_i$ be the radial segment bounding $F_i$ in counterclockwise direction and $S^-_i$ be the radial segment bounding it in clockwise direction, then we define the following subspaces of $F_i$:
\begin{align*}\alpha^+_i &\coloneqq \{z \in S^+_i\,|\,|z| \leq |\xi_{\overline{\omega}(i)}|\}, \qquad
\alpha^-_i \coloneqq \{z \in S^-_i\,|\,|z| \leq |\xi_{i}|\},\\
\beta^+_i &\coloneqq \{z \in S^+_i\,|\, |z| \geq |\xi_{\overline{\omega}(i)}|\}, \qquad
\beta^-_i \coloneqq \{z \in S^-_i\,|\, |z| \geq |\xi_{\overline{\omega}{i}}|\}.\end{align*}
These subspaces are empty for entire sectors.
\end{definition}

%We now define the equivalence relation $\approx_L$ and the surface $\Sigma(L)$.

\begin{definition}\label{def_approxl}The equivalence relation $\approx_L$ on $\overline{\Sigma}(L)$ is the one generated by:
\begin{enumerate}[(i)]
\item We identify $z \in \alpha^+_i$ with $z \in \alpha^-_{\overline{\omega}(i)}$.
\item We identify $z \in \beta^+_i$ with $z \in \beta^-_{\overline{\lambda}(i)}$.
\end{enumerate}
We define the surface $\Sigma(L)$ to be $\overline{\Sigma}(L)/{\approx_L}$.
\end{definition}

\begin{definition}\label{defconstrsl} The cobordism $S(L)$ has underlying surface $\Sigma(L)$. It has a map from each inner boundary $\partial_\mr{in} \bA^{(j)}_{R}$
\[\iota^\mr{in}_j \colon S^1 \cong \partial_\mr{in} \bA^{(j)}_{R} \lra \Sigma(L),\] 
and these are inclusions of subspaces if none of the slits lie on the inner boundary of an annulus. One can define the outgoing boundary components as a subspace of $\Sigma(L)$ by considering the intersection of the outer boundary of the annuli with the sectors. For each cycle in $\lambda \circ \omega$ these intersections form a circle with canonical orientation and starting point $P_k$. This yields for the cycle $\lambda \circ \omega$ corresponding to $P_k$ a map
\[\iota^\mr{out}_k \colon S^1 \lra \Sigma(L),\] 
and these are inclusions of subspaces if none of the slits lie on the outer boundary of an annulus.\end{definition}

As mentioned before, this definition may result in a degenerate cobordism for some $L$.  Moreover, two different pre-configurations might give the same conformal classes of cobordism. In fact, each conformal class of cobordisms occurs at least $(2h)!$ times, because the labeling on the slits does not matter. To see that degenerate surfaces can occur, consider the example in Figure \ref{figdegexample}.  Now we explain how to resolve both issues.

\begin{figure}[h]
  \centering
    \begin{tikzpicture}
    \node at (0,0) {$1$};
    \draw (90:1.5cm) -- (90:1.6cm);
    \node at (90:1.6cm) [above] {$1$};

    \draw [red,thick,yshift=.07cm] (0:.8cm) -- (0:1.5cm);
    \draw [blue,thick] (180:1.2cm) -- (180:1.5cm);   
    \draw [blue,thick] (0:1.2cm) -- (0:1.5cm);
    \draw [red,thick,yshift=-.07cm] (0:.8cm) -- (0:1.5cm);
    
    \draw[decoration={markings, mark=at position 0 with {\arrow{<}}}, 	postaction={decorate}] (0,0) circle (1.5cm);
    \draw[decoration={markings, mark=at position 0 with {\arrow{<}}},	postaction={decorate}]  (0,0) circle (.5cm);
    
    \begin{scope}[xshift=4cm]

    \draw [blue,thick] (180:1.2cm) -- (180:1.5cm);   
    \draw [red,thick] (0:.8cm) -- (0:1.5cm);
    \node at (180:1.2cm) {$\bullet$};
    \node at (0:.8cm) {$\bullet$};
    
    \draw[decoration={markings, mark=at position 0 with {\arrow{<}}}, 	postaction={decorate},yscale=.7] (0,0) circle (1.5cm);
    \draw[decoration={markings, mark=at position 0 with {\arrow{<}}},	postaction={decorate},yscale=.7]  (0,0) circle (.5cm);
    
    \draw [line width=2mm,white] (180:1.2cm) to[out=90,in=90,looseness=2.5] (0:.8cm);
    \draw (180:1.2cm) to[out=90,in=90,looseness=2.5] (0:.8cm);
    \end{scope}
    \end{tikzpicture}
  \caption{An example of a radial slit preconfiguration leading to a degenerate surface. The black arc connecting two points on the surface on the right was the line segment between the two red slits.}
  \label{figdegexample}
\end{figure}

We have already explained that one should identify configurations obtained by permuting the labels on the slits. We only need to make two additional identifications. For the first additional identification, instead of doing all the cutting and gluing simultaneously, do it in order of increasing modulus of the slits. This results in the same cobordism but doing so makes clear it that if $\zeta_i$ lies on the same radial segment as $\zeta_j$ and satisfies $|\zeta_i| \geq |\zeta_j|$, it might as well be on the other side of $\zeta_{\lambda(j)}$. That is, it might as well have ``jumped'' over the slit $\zeta_j$ to $\zeta_{\lambda(j)}$. For the second additional identification, note that if a parametrization point similarly ``jumps'' over a slit, this does not change the parametrization of the outgoing boundary. These will turn out to be all required identifications, and we now use them to define equivalence relations on $\mathrm{PRad}_h(n,m)$.

\begin{definition}Let $\equiv'$ be the equivalence relation on $\bprad_h(n,m)$ generated by
\begin{description}
	\item[Relabeling of the slits] We identify two preconfigurations if they can be obtained from each other by relabeling the slits. More precisely for every permutation $\sigma \in \mathfrak S_{2h}$, extended by the identity to a permutation $\overline{\sigma} \in \mathfrak{S}_{2h+m}$, and $L = (\vec{\xi},\overline{\lambda},\overline{\omega},\vec{r}) \in \mathrm{PRad}_h(n,m)$ we say that $L \equiv' \sigma(L)$, with 
	\[\sigma(L)=\big((\vec{\xi})^{\overline{\sigma}},(\overline{\lambda})^{\overline{\sigma}},(\overline{\omega})^{\overline{\sigma}},(\vec{r})^{\overline{\sigma}}\big),\]
	whose components defined as follows:
	\begin{itemize}
		\item $(\vec{\xi})^{\overline{\sigma}}$ is given by $(\xi)^{\overline{\sigma}}_i = \xi_{\overline{\sigma}(i)}$,
		\item $(\overline{\lambda})^{\overline{\sigma}}=\overline{\sigma}\circ\overline{\lambda}\circ\overline{\sigma}^{-1}$,
		\item $(\overline{\omega})^{\overline{\sigma}}=\overline{\sigma}\circ\overline{\omega}\circ\overline{\sigma}^{-1}$,
		\item $(\vec{r})^{\overline{\sigma}}$ is given by $(r)^{\overline{\sigma}}_i = r_{\overline{\sigma}(i)}$.
	\end{itemize}
\end{description}
Let $\equiv$ be the equivalence relation on $\bprad_h(n,m)$ generated by relabeling of the slits (as above) and the following two identifications:
\begin{description}
	\item[Slit jumps] We say $L \equiv L'$ if $L'$ can be obtained from $L$ by a slit jump, see Figure \ref{figslitjump}. More precisely, if we are given a preconfiguration $L$ and two indices $i$ and $j$ such that $j = \omega(i)$, $r_i = 0$ and $|\zeta_i| \geq |\zeta_j|$, then we can obtain a new preconfiguration $L'$ as follows. We replace $\zeta_i$ by the point $\zeta'_i = \frac{|\zeta_i|}{|\zeta_{\lambda(j)}|}{\zeta_{\lambda(j)}}$ and keep all the other slits the same. We then put $i$ after of $\lambda(j)$ in $\overline{\omega}$ to obtain $\overline{\omega}'$ and set $r'_i=r_{\lambda(j)}$ and $r'_{\lambda(j)}=0$.  The rest of the data remains the same.	\item[Parametrization point jumps] We say $L \equiv L'$ if $L'$ can be obtained from $L$ by a jump of a parametrization point, see Figure \ref{figparamjump}. More precisely, if we are given a preconfiguration $L$ in which there is a $P_i$ such that $j = \overline{\omega}(i+2h)$ for some $j$ and $r_{i+2h} = 0$, then we can obtain a new preconfiguration $L'$ by keeping all the data the same except replacing $P_i$ with $P'_i$ lying at the radial segment through $\zeta_{\lambda(j)}$ and setting $r'_{i+2h}=r_{\lambda(j)}$ and $r'_{\lambda(j)}=0$.
\end{description}

\end{definition}

\begin{figure}[t]
  \centering
    \begin{tikzpicture}[scale=.9]
    \node at (0,0) {$1$};
    \draw (90:1.5cm) -- (90:1.7cm);
    \node at (90:1.7cm) [above] {$1$};

    \draw [red,thick] (100:.8cm) -- (100:1.5cm);
    \draw [blue,thick] (-110:1.2cm) -- (-110:1.5cm);   
    \draw [blue,thick,xshift=-.5mm,yshift=-.5mm] (-45:1.2cm) -- (-45:1.5cm);
    \draw [red,thick] (-45:.8cm) -- (-45:1.5cm);
    
    \draw[decoration={markings, mark=at position 0 with {\arrow{<}}}, 	postaction={decorate}] (0,0) circle (1.5cm);
    \draw[decoration={markings, mark=at position 0 with {\arrow{<}}},	postaction={decorate}]  (0,0) circle (.5cm);
    
    \draw [<->] (2,0) -- (3,0);
    \node at (2.5,0) [above] {$\equiv$};
    
    \begin{scope}[xshift=5cm]
    \node at (0,0) {$1$};
    \draw (90:1.5cm) -- (90:1.7cm);
    \node at (90:1.7cm) [above] {$1$};

    \draw [red,thick] (100:.8cm) -- (100:1.5cm);
    \draw [blue,thick] (-110:1.2cm) -- (-110:1.5cm);   
    \draw [blue,thick,xshift=-.7mm,yshift=-.1mm] (100:1.2cm) -- (100:1.5cm);
    \draw [red,thick] (-45:.8cm) -- (-45:1.5cm);
    
    \draw[decoration={markings, mark=at position 0 with {\arrow{<}}}, 	postaction={decorate}] (0,0) circle (1.5cm);
    \draw[decoration={markings, mark=at position 0 with {\arrow{<}}},	postaction={decorate}]  (0,0) circle (.5cm);
   	\end{scope}
    \end{tikzpicture}
  \caption{A jump of a slit. The pairing $\lambda$ is given by the colors, but is uniquely determined by the configuration.}
  \label{figslitjump}
\end{figure}

\begin{figure}[t]
  \centering
    \begin{tikzpicture}[scale=.9]
    \node at (0,0) {$1$};
    \draw (90:1.5cm) -- (90:1.7cm);
    \node at (90:1.7cm) [above] {$1$};
    \draw (200:1.5cm) -- (200:1.7cm);
    \node at (200:1.7cm) [left] {$2$};    
    \draw (-90:1.5cm) -- (-90:1.7cm);
    \node at (-90:1.7cm) [below] {$3$};
    
    \draw [red,thick] (90:.8cm) -- (90:1.5cm);
    \draw [blue,thick] (-110:1.2cm) -- (-110:1.5cm);   
    \draw [blue,thick] (170:1.2cm) -- (170:1.5cm);
    \draw [red,thick] (-45:.8cm) -- (-45:1.5cm);
    
    \draw[decoration={markings, mark=at position 0 with {\arrow{<}}}, 	postaction={decorate}] (0,0) circle (1.5cm);
    \draw[decoration={markings, mark=at position 0 with {\arrow{<}}},	postaction={decorate}]  (0,0) circle (.5cm);
    
    \draw [<->] (2,0) -- (3,0);
    \node at (2.5,0) [above] {$\equiv$};
    
    \begin{scope}[xshift=5.25cm]
    \node at (0,0) {$1$};
	\draw (-45:1.5cm) -- (-45:1.7cm);
	\node at (-45:1.7cm) [right] {$1$};
	\draw (200:1.5cm) -- (200:1.7cm);
	\node at (200:1.7cm) [left] {$2$};    
	\draw (-90:1.5cm) -- (-90:1.7cm);
	\node at (-90:1.7cm) [below] {$3$};
	
	\draw [red,thick] (90:.8cm) -- (90:1.5cm);
	\draw [blue,thick] (-110:1.2cm) -- (-110:1.5cm);   
	\draw [blue,thick] (170:1.2cm) -- (170:1.5cm);
	\draw [red,thick] (-45:.8cm) -- (-45:1.5cm);

	\draw[decoration={markings, mark=at position 0 with {\arrow{<}}}, 	postaction={decorate}] (0,0) circle (1.5cm);
	\draw[decoration={markings, mark=at position 0 with {\arrow{<}}},	postaction={decorate}]  (0,0) circle (.5cm);
    \end{scope}
    \end{tikzpicture}
  \caption{A jump of a parametrization point.}
  \label{figparamjump}
\end{figure}

\begin{definition} We now define certain quotient spaces using these equivalence relations. \label{defbradslit}
\begin{itemize}
\item The space $\bqrad_h(n,m)$ of \emph{unlabeled possibly degenerate radial slit configurations} is the quotient of $\bprad_h(n,m)$ by $\equiv '$. 
\item The space $\brad_h(n,m)$ of \emph{possibly degenerate radial slit configurations} is the quotient of $\bprad_h(n,m)$ by $\equiv$.
\end{itemize}
\end{definition}

We will denote by $[L]$ the radial slit configuration represented by a preconfiguration $L$. We are left to deal with the problem that certain preconfigurations give cobordisms whose underlying surface is degenerate. We call such preconfigurations \emph{degenerate}. In \cite{bodigheimer}, B\"odigheimer gave a necessary and sufficient criterion for a (pre)configuration to lead to a degenerate surface:

\begin{proposition}\label{propdegcriterion}
The surface underlying the cobordism $\Sigma(L)$ constructed out of a preconfiguration $L$ is degenerate if and only if it is equivalent under $\equiv$ to a preconfiguration satisfying at least one of the following three conditions:
\begin{description}
\item[Slit hitting inner boundary] There is a slit $\zeta_i$ with $|\zeta_i| = \frac{1}{2\pi}$.
\item[Slit hitting outer boundary] There is a slit $\zeta_i$ on an annulus $\bA^{(j)}_{R}$ with $|\zeta_i| = R_j$.
\item[Slits are ``squeezed''] There is a pair $i,j$ such that $j= \lambda(i)$, $\zeta_i$ and $\zeta_j$ lie on the same annulus, $\zeta_i = \zeta_j$ and such that for all $k$ between $i$ and $j$ in the cyclic ordering coming from $\omega$, we have that $|\zeta_k| \geq |\zeta_i| = |\zeta_j|$ (see Figure \ref{figdegexample} for an example). If all slits on the annulus containing $\zeta_i$ and $\zeta_j$ lie at the same point, we additionally require that $r_k =0 $ for all of the $k$ between $i$ and $j$.
\end{description}
\end{proposition}

\begin{definition}A radial slit preconfiguration is said to be \emph{generic} if it is not equivalent to any other by slit or parametrization point jumps, i.e. all the slits are disjoint.\end{definition}

\begin{definition}\label{def_rad} We define the following spaces:
\begin{itemize}
\item The space $\prad_h(n,m)$ of \emph{unlabeled radial slit configurations} is the subspace of $\bprad_h(n,m)$ consisting of non-degenerate preconfigurations. 
\item The space $\qrad_h(n,m)$ of \emph{labeled radial slit configurations} is the subspace of $\bqrad_h(n,m)$ consisting of equivalence classes with non-degenerate representatives. 
\item The space $\rad_h(n,m)$ of \emph{radial slit configurations} is the subspace of $\brad_h(n,m)$ consisting of equivalence classes with non-degenerate representatives. 
\end{itemize}
\end{definition}

\subsubsection{Cell complexes of radial slit configurations} Next we give CW complexes $\bRad$ and $\Rad$ homeomorphic to  the spaces of radial slit configurations given before. On $\bRad$ this is the CW structure given in Section 8.2 of \cite{bodigheimer} and on the subspace $\Rad$ it coincides with the radial analogue of \cite{bodighpreprint}. The cells will be indexed by so-called combinatorial types, which we define first.

\begin{definition}
Fix an $L$ in $\prad_h(n,m)$. 
\begin{itemize}
\item The radial segments of the slits, the parametrization points and the positive real lines, divide the annuli of the preconfiguration $L$ radially into different pieces, which we will call \emph{radial chambers} (see Figure \ref{chambers_radial}).
\item Each slit $\zeta_i$ in $L$ defines a circle of radius $|\zeta_i|$ on all of the $n$ annuli.  These circles divide the $n$ annuli into different pieces, which we will call \emph{annular chambers} (see Figure \ref{chambers_radial}). 
\end{itemize}
\end{definition}

\begin{remark}
The orientation of the complex plane endows the radial chambers on each annulus with a natural ordering, and similarly the modulus endows the annular chambers with a natural ordering (see Figure \ref{chambers_radial}).
\end{remark}

\begin{figure}[h!]
  \centering
   \begin{tikzpicture}[scale=1.05]
   	\begin{scope}
   	\node at (0,0) {\footnotesize $1$};
   	\draw (45:1.2cm) -- (45:1.4cm);
   	\node at (45:1.4cm) [right] {\footnotesize $1$};
   	\draw (135:1.2cm) -- (135:1.4cm);
   	\node at (135:1.3cm) [left] {\footnotesize $4$};
   	\draw (-135:1.2cm) -- (-135:1.4cm);
   	\node at (-135:1.4cm) [left] {\footnotesize $3$};

   	\draw[dotted] (0,0) circle (.5cm);   	
   	\draw[dotted] (0,0) circle (.8cm);
   	\draw[dotted] (0,0) circle (1cm);
   	\draw[dashed] (0:.3) -- (0:1.2);
   	\draw[dotted] (45:.3) -- (45:1.2);
   	\draw[dotted] (90:.3) -- (90:1.2);
   	\draw[dotted] (135:.3) -- (135:1.2);
	\draw[dotted] (180:.3) -- (180:1.2);
   	\draw[dotted] (-135:.3) -- (-135:1.2);
   	\node[red] at (-90:.4) {\tiny $0$};
   	\node[red] at (-90:.65) {\tiny $1$};
   	\node[red] at (-90:.9) {\tiny $2$};
   	\node[red] at (-90:1.1) {\tiny $3$};
   	\node[blue] at (22.5:1.3) {\tiny $0$};
   	\node[blue] at (67.5:1.3) {\tiny $1$};
   	\node[blue] at (112.5:1.3) {\tiny $2$};
   	\node[blue] at (157.5:1.3) {\tiny $3$};
   	\node[blue] at (-157.5:1.3) {\tiny $4$};
   	\node[blue] at (-{135/2}:1.3) {\tiny $5$};
   	
   	\draw [red,thick] (180:.8cm) -- (180:1.2cm);
   	\draw [blue,thick] (90:.5cm) -- (90:1.2cm);   
   	\draw [red,thick,yshift=-.5mm] (0:.8cm) -- (0:1.2cm);
   	\draw [blue,thick] (0:.5cm) -- (0:1.2cm);
   	
   	\draw (0,0) circle (1.2cm);
   	\draw[decoration={markings, mark=at position 0.5 with {\arrow{<}}},	postaction={decorate}]  (0,0) circle (.3cm);
   	
   	\node at (1.5,-1.5) {(a)};
   	\node at (1.5,1.5) {$L$};
   	
   	\begin{scope}[xshift=3cm]
   	\node at (0,0) {\footnotesize $2$};
   	\draw (135:1.2cm) -- (135:1.4cm);
   	\node at (135:1.4cm) [left] {\footnotesize $5$};
   	\draw (-135:1.2cm) -- (-135:1.4cm);
   	\node at (-135:1.4cm) [below] {\footnotesize $2$};
   	
   	\draw[dotted] (0,0) circle (.5cm);   	
   	\draw[dotted] (0,0) circle (.8cm);
   	\draw[dotted] (0,0) circle (1cm);
   	\draw[dashed] (0:.3) -- (0:1.2);
   	\draw[dotted] (135:.3) -- (135:1.2);
   	\draw[dotted] (-135:.3) -- (-135:1.2);
   	\draw[dotted] (-45:.3) -- (-45:1.2);
   	\node[red] at (-90:.4) {\tiny $0$};
   	\node[red] at (-90:.65) {\tiny $1$};
   	\node[red] at (-90:.9) {\tiny $2$};
   	\node[red] at (-90:1.1) {\tiny $3$};
   	\node[blue] at ({135/2}:1.3) {\tiny $0$};
   	\node[blue] at (-180:1.3) {\tiny $1$};
   	\node[blue] at (-22.5:1.3) {\tiny $3$};
   	\node[blue] at (-90:1.3) {\tiny $2$};
   	
   	\draw [PineGreen,thick] (-135:1cm) -- (-135:1.2cm);
   	\draw [PineGreen,thick] (-45:1cm) -- (-45:1.2cm);   
   	
   	\draw (0,0) circle (1.2cm);
   	\draw[decoration={markings, mark=at position 0.5 with {\arrow{<}}},	postaction={decorate}]  (0,0) circle (.3cm);
   	\end{scope}
   	\end{scope}
   	
   	\begin{scope}[xshift=7cm]
   	\fill[yellow!20!white] (0,0) ++ (135:.3) arc (135:180:.3) -- (180:1.2) arc (180:135:1.2) -- cycle;
   	\node at (0,0) {\footnotesize $1$};
   	\draw (45:1.2cm) -- (45:1.4cm);
   	\node at (45:1.4cm) [right] {\footnotesize $1$};
   	\draw (180:1.2cm) -- (180:1.4cm);
   	\node at (180:1.4cm) [left] {\footnotesize $4$};
   	\draw (-135:1.2cm) -- (-135:1.4cm);
   	\node at (-135:1.4cm) [left] {\footnotesize $3$};
   	
   	\draw[dotted] (0,0) circle (.5cm);   	
   	\draw[dotted] (0,0) circle (.8cm);
   	\draw[dotted] (0,0) circle (1cm);
   	\draw[dashed] (0:.3) -- (0:1.2);
   	\draw[dotted] (45:.3) -- (45:1.2);
   	\draw[dotted] (90:.3) -- (90:1.2);
   	\draw[dotted] (180:.3) -- (180:1.2);
   	\draw[dotted] (-135:.3) -- (-135:1.2);
   	\node[red] at (-90:.4) {\tiny $0$};
   	\node[red] at (-90:.65) {\tiny $1$};
   	\node[red] at (-90:.9) {\tiny $2$};
   	\node[red] at (-90:1.1) {\tiny $3$};
   	\node[blue] at (22.5:1.3) {\tiny $0$};
   	\node[blue] at (67.5:1.3) {\tiny $1$};
   	\node[blue] at (135:1.3) {\tiny $2$};
   	\node[blue] at (-157.5:1.3) {\tiny $3$};
   	\node[blue] at (-{135/2}:1.3) {\tiny $4$};
   	
   	\draw [red,thick] (180:.8cm) -- (180:1.2cm);
   	\draw [blue,thick] (90:.5cm) -- (90:1.2cm);   
   	\draw [red,thick,yshift=-.5mm] (0:.8cm) -- (0:1.2cm);
   	\draw [blue,thick] (0:.5cm) -- (0:1.2cm);
   	
   	\draw (0,0) circle (1.2cm);
   	\draw[decoration={markings, mark=at position 0.5 with {\arrow{<}}},	postaction={decorate}]  (0,0) circle (.3cm);
   	
   	\node at (1.5,-1.5) {(b)};
   	\node at (1.5,1.5) {$\partial^1_3(L)$};
   	
   	\begin{scope}[xshift=3cm]
	\node at (0,0) {\footnotesize $2$};
	\draw (135:1.2cm) -- (135:1.4cm);
	\node at (135:1.4cm) [left] {\footnotesize $5$};
	\draw (-135:1.2cm) -- (-135:1.4cm);
	\node at (-135:1.4cm) [below] {\footnotesize $2$};

	\draw[dotted] (0,0) circle (.5cm);   	
	\draw[dotted] (0,0) circle (.8cm);
	\draw[dotted] (0,0) circle (1cm);
	\draw[dashed] (0:.3) -- (0:1.2);
	\draw[dotted] (135:.3) -- (135:1.2);
	\draw[dotted] (-135:.3) -- (-135:1.2);
	\draw[dotted] (-45:.3) -- (-45:1.2);
	\node[red] at (-90:.4) {\tiny $0$};
	\node[red] at (-90:.65) {\tiny $1$};
	\node[red] at (-90:.9) {\tiny $2$};
	\node[red] at (-90:1.1) {\tiny $3$};
	\node[blue] at ({135/2}:1.3) {\tiny $0$};
	\node[blue] at (-180:1.3) {\tiny $1$};
	\node[blue] at (-22.5:1.3) {\tiny $3$};
	\node[blue] at (-90:1.3) {\tiny $2$};

	\draw [PineGreen,thick] (-135:1cm) -- (-135:1.2cm);
	\draw [PineGreen,thick] (-45:1cm) -- (-45:1.2cm);   

	\draw (0,0) circle (1.2cm);
	\draw[decoration={markings, mark=at position 0.5 with {\arrow{<}}},	postaction={decorate}]  (0,0) circle (.3cm);
	\end{scope}
   	\end{scope}
   	
   	\begin{scope}[yshift=-4.5cm]
   	\begin{scope}
   	\node at (0,0) {\footnotesize $1$};
   	\draw (45:1.2cm) -- (45:1.4cm);
   	\node at (45:1.4cm) [right] {\footnotesize $1$};
   	\draw (135:1.2cm) -- (135:1.4cm);
   	\node at (135:1.3cm) [left] {\footnotesize $4$};
   	\draw (-135:1.2cm) -- (-135:1.4cm);
   	\node at (-135:1.4cm) [left] {\footnotesize $3$};
   	
   	\draw[dotted] (0,0) circle (.5cm);   	
   	\draw[dotted] (0,0) circle (.8cm);
   	\draw[dotted] (0,0) circle (1cm);
   	\draw[dashed] (0:.3) -- (0:1.2);
   	\draw[dotted] (45:.3) -- (45:1.2);
   	\draw[dotted] (90:.3) -- (90:1.2);
   	\draw[dotted] (135:.3) -- (135:1.2);
   	\draw[dotted] (180:.3) -- (180:1.2);
   	\draw[dotted] (-135:.3) -- (-135:1.2);
   	\node[red] at (-90:.4) {\tiny $0$};
   	\node[red] at (-90:.65) {\tiny $1$};
   	\node[red] at (-90:.9) {\tiny $2$};
   	\node[red] at (-90:1.1) {\tiny $3$};
   	\node[blue] at (22.5:1.3) {\tiny $0$};
   	\node[blue] at (67.5:1.3) {\tiny $1$};
   	\node[blue] at (112.5:1.3) {\tiny $2$};
   	\node[blue] at (157.5:1.3) {\tiny $3$};
   	\node[blue] at (-157.5:1.3) {\tiny $4$};
   	\node[blue] at (-{135/2}:1.3) {\tiny $5$};
   	
   	\draw [red,thick] (180:.8cm) -- (180:1.2cm);
   	\draw [blue,thick] (90:.5cm) -- (90:1.2cm);   
   	\draw [red,thick,yshift=-.5mm] (0:.8cm) -- (0:1.2cm);
   	\draw [blue,thick] (0:.5cm) -- (0:1.2cm);
   	
   	\draw (0,0) circle (1.2cm);
   	\draw[decoration={markings, mark=at position 0.5 with {\arrow{<}}},	postaction={decorate}]  (0,0) circle (.3cm);
   	
   	\node at (1.5,-1.5) {(c)};
   	\node at (1.5,1.5) {$\partial^2_3(L)$};
   	
   	\begin{scope}[xshift=3cm]
   	\fill[yellow!20!white] (0,0) ++ (0:.3) arc (0:-45:.3) -- (-45:1.2) arc (-45:0:1.2) -- cycle;
   	\node at (0,0) {\footnotesize $2$};
   	\draw (135:1.2cm) -- (135:1.4cm);
   	\node at (135:1.4cm) [left] {\footnotesize $5$};
   	\draw (-135:1.2cm) -- (-135:1.4cm);
   	\node at (-135:1.4cm) [below] {\footnotesize $2$};
   	
   	\draw[dotted] (0,0) circle (.5cm);   	
   	\draw[dotted] (0,0) circle (.8cm);
   	\draw[dotted] (0,0) circle (1cm);
   	\draw[dashed] (0:.3) -- (0:1.2);
   	\draw[dotted] (135:.3) -- (135:1.2);
   	\draw[dotted] (-135:.3) -- (-135:1.2);
   	\node[red] at (-90:.4) {\tiny $0$};
   	\node[red] at (-90:.65) {\tiny $1$};
   	\node[red] at (-90:.9) {\tiny $2$};
   	\node[red] at (-90:1.1) {\tiny $3$};
   	\node[blue] at ({135/2}:1.3) {\tiny $0$};
   	\node[blue] at (180:1.3) {\tiny $1$};
   	\node[blue] at (-45.5:1.3) {\tiny $2$};
   	
   	\draw [PineGreen,thick] (-135:1cm) -- (-135:1.2cm);
   	\draw [PineGreen,thick] (0:1cm) -- (0:1.2cm);   
   	
   	\draw (0,0) circle (1.2cm);
   	\draw[decoration={markings, mark=at position 0.5 with {\arrow{<}}},	postaction={decorate}]  (0,0) circle (.3cm);
   	\end{scope}
   	\end{scope}
   	
   	\begin{scope}[xshift=7cm]
   	\fill [yellow!20!white] (0,0) circle (.8cm);
   	\fill [white] (0,0) circle (.5cm);
   	\node at (0,0) {\footnotesize $1$};
   	\draw (45:1.2cm) -- (45:1.4cm);
   	\node at (45:1.4cm) [right] {\footnotesize $1$};
   	\draw (180:1.2cm) -- (180:1.4cm);
   	\node at (180:1.4cm) [left] {\footnotesize $4$};
   	\draw (-135:1.2cm) -- (-135:1.4cm);
   	\node at (-135:1.4cm) [left] {\footnotesize $3$};
   	
   	\draw[dotted] (0,0) circle (.5cm);   	
   	\draw[dotted] (0,0) circle (1cm);
   	\draw[dashed] (0:.3) -- (0:1.2);
   	\draw[dotted] (45:.3) -- (45:1.2);
   	\draw[dotted] (90:.3) -- (90:1.2);
   	\draw[dotted] (180:.3) -- (180:1.2);
   	\draw[dotted] (-135:.3) -- (-135:1.2);
   	\node[red] at (-90:.4) {\tiny $0$};
   	\node[red] at (-90:.75) {\tiny $1$};
   	\node[red] at (-90:1.1) {\tiny $2$};
   	\node[blue] at (22.5:1.3) {\tiny $0$};
   	\node[blue] at (67.5:1.3) {\tiny $1$};
   	\node[blue] at (135:1.3) {\tiny $2$};
   	\node[blue] at (-157.5:1.3) {\tiny $3$};
   	\node[blue] at (-{135/2}:1.3) {\tiny $4$};
   	
   	\draw [red,thick] (180:.5cm) -- (180:1.2cm);
   	\draw [blue,thick] (90:.5cm) -- (90:1.2cm);   
   	\draw [red,thick,yshift=-.5mm] (0:.5cm) -- (0:1.2cm);
   	\draw [blue,thick] (0:.5cm) -- (0:1.2cm);
   	
   	\draw (0,0) circle (1.2cm);
   	\draw[decoration={markings, mark=at position 0.5 with {\arrow{<}}},	postaction={decorate}]  (0,0) circle (.3cm);
   	
   	\node at (1.5,-1.5) {(d)};
   	\node at (1.5,1.5) {$\partial^3_1(L)$};
   	
   	\begin{scope}[xshift=3cm]
  	\fill [yellow!20!white] (0,0) circle (.8cm);
   	\fill [white] (0,0) circle (.5cm);
   	\node at (0,0) {\footnotesize $2$};
   	\draw (135:1.2cm) -- (135:1.4cm);
   	\node at (135:1.4cm) [left] {\footnotesize $5$};
   	\draw (-135:1.2cm) -- (-135:1.4cm);
   	\node at (-135:1.4cm) [below] {\footnotesize $2$};
   	
   	\draw[dotted] (0,0) circle (.5cm);   	
   	\draw[dotted] (0,0) circle (1cm);
   	\draw[dashed] (0:.3) -- (0:1.2);
   	\draw[dotted] (135:.3) -- (135:1.2);
   	\draw[dotted] (-135:.3) -- (-135:1.2);
   	\draw[dotted] (-45:.3) -- (-45:1.2);
   	\node[red] at (-90:.4) {\tiny $0$};
   	\node[red] at (-90:.75) {\tiny $1$};
   	\node[red] at (-90:1.1) {\tiny $2$};
   	\node[blue] at ({135/2}:1.3) {\tiny $0$};
   	\node[blue] at (-180:1.3) {\tiny $1$};
   	\node[blue] at (-90:1.3) {\tiny $2$};
   	\node[blue] at (-22.5:1.3) {\tiny $2$};
   	
   	\draw [PineGreen,thick] (-135:1cm) -- (-135:1.2cm);
   	\draw [PineGreen,thick] (-45:1cm) -- (-45:1.2cm);   
   	
   	\draw (0,0) circle (1.2cm);
   	\draw[decoration={markings, mark=at position 0.5 with {\arrow{<}}},	postaction={decorate}]  (0,0) circle (.3cm);
   	\end{scope}
   	\end{scope}
   	\end{scope}
   \end{tikzpicture}
   \caption{(a) A configuration $L$ and its radial and annular chambers divided by dotted lines. The radial chambers are numbered in blue (there are 6 radial chambers on the left annulus and 4 on the right annulus) and the annular chambers are numbered in red (there are 3 annular chambers consisting of a pair small annuli, one on each of the annuli). This combinatorial type gives an $11$-cell in $\bRad$ given by a $\Delta^5\times\Delta^3 \times \Delta^3$.  (b), (c), (d) show part of the boundary of $L$ and their chambers. The modified parts are marked in light yellow.}
  \label{chambers_radial}
\end{figure} 

Each of the annular chambers is homeomorphic to a disjoint union of $n$ annuli, while each of the radial chambers is homeomorphic to a rectangle.

\begin{definition}Two preconfigurations $L$ and $L'$ in $\prad_h(n,m)$ are said to \emph{have the same combinatorial data} if $L'$ can be obtained from $L$ by continuously moving the slits and parametrization points in each complex plane without collapsing any chamber.  This defines an equivalence relation on $\prad_h(n,m)$.

A \emph{combinatorial type of preconfigurations} $\cL$ is an equivalence class of preconfigurations under this relation. Informally, a combinatorial type is the data carried over by the picture of a preconfiguration without remembering the precise placement of the slits. Notice that this equivalence relation is also well defined on the sets of radial slit configurations $[L]$.  Thus one can similarly define a \emph{combinatorial type of configurations} $[\cL]$ to be an equivalence class of configurations under this relation.  Similarly for the case of unlabeled radial slit configurations. 

We will use $\Upsilon$ for the \emph{set of all combinatorial types of configurations}.
\end{definition}

\begin{remark}Notice that if $L$ is a degenerate (respectively non-degenerate) preconfiguration then so is any preconfiguration of the same combinatorial type.  Thus, we can talk about a degenerate or non-degenerate combinatorial type.  
\end{remark}

Now we give definitions of cell complexes of (pre)configurations and their compactifications. Note that the meaning of $p$ and $q$ is different from \cite{bodigheimer}. 

\begin{definition}
\label{def_radcw}
The \emph{multi-degree of a combinatorial type} $[\cL]$ on $n$ annuli is the $(n+1)$-tuple of integers $(q_1,\ldots,q_n,p)$ where $q_i+1$ is the number of radial chambers in the $i$th annulus and $p+1$ is the number of annular chambers.  For $0 \leq j \leq q_i$ and $0 \leq i \leq n$, we denote by $d^i_j([\cL])$, the combinatorial type obtained by collapsing the $j$th radial chamber on the $i$th annulus, see Figure \ref{chambers_radial}.  For $0 \leq j \leq p$, we denote by $d^{n+1}_j([\cL])$, the combinatorial type obtained by collapsing the $j$th annular chamber, see Figure \ref{chambers_radial}.

The \emph{cell complex of possibly degenerate radial slit configurations} $\bRad_h(n,m)$ is the realization of the multisimplicial set with:

\begin{itemize}
\item  $(q_1,\ldots,q_n,p)$-simplices given by \[\left\{ e_{[\cL]}\, \middle|\,[\cL] \text{ combinatorial type of multi-degree } (q_1,\ldots,q_n,p)\right\},\]
\item the faces of $e_{[\cL]}$ given by $d^i_j(\sigma_{[\cL]})\coloneqq \sigma_{d^i_j([\cL])}$.
\end{itemize}

That is, $\bRad_h(n,m)$ is a CW-complex with cells indexed by combinatorial types of radial slits configurations as follows.  Let $e_{[\cL]} \coloneqq \Delta^{q_1} \times \ldots \times \Delta^{q_n} \times \Delta^{p}$, then: 
\[\bRad_h(n,m)\coloneqq \frac{\bigsqcup_{[\cL] \in \Upsilon} e_{[\cL]}}{\sim}\]
where the equivalence relation is generated by
\[(e_{[\cL]},(\vec{t}_{1},\ldots,\delta^j(\vec{t}_i),\ldots,\vec{t}_{n+1}))\sim
(e_{d^i_j([\cL])},(\vec{t}_1,\ldots,\vec{t}_i,\ldots,\vec{t}_{n+1}))\]
where $\delta^j$ is the map $\Delta^{q_i-1} \to \Delta^{q_i}$ including $0$ as the $(j+1)$st coordinate, and $\Upsilon$ is the set of combinatorial types of radial slit configurations.

The cell complexes of possibly degenerate radial slit preconfigurations $\bPRad_h(n,m)$ and unlabeled configurations $\bQRad_h(n,m)$ are defined in similar ways.\end{definition}

\begin{figure}[h!]
	\centering
    \begin{tikzpicture}[scale=1.15]
    \node at (0,0) {\footnotesize $1$};
    \draw (180:1cm) -- (180:1.2cm);
    \node at (180:1.2cm) [left] {\footnotesize $1$};
    
    \draw[dotted] (0,0) circle (.65cm);   	
    \draw[dashed] (0:.3) -- (0:1);
    \draw[dotted] (180:.3) -- (180:1);
    \node[red] at (-90:{.3/2+.65/2}) {\tiny $0$};
    \node[red] at (-90:{1/2+.65/2}) {\tiny $1$};
    \node[blue] at (90:1.1) {\tiny $0$};
    \node[blue] at (-90:1.1) {\tiny $1$};

    \draw [red,thick] (0:.65cm) -- (0:1cm);
    
    \draw (0,0) circle (1cm);
    \draw[decoration={markings, mark=at position 0.5 with {\arrow{<}}},	postaction={decorate}]  (0,0) circle (.3cm);
    
    \node at (1.25,1) {$L$};
    
    \begin{scope}[xshift=2.5cm]
    \node at (0,0) {\footnotesize $2$};

	\draw[dotted] (0,0) circle (.65cm);   	
	\draw[dashed] (0:.3) -- (0:1);
	\draw[dotted] ({360/5}:.3) -- ({360/5}:1);
	\draw[dotted] ({2*360/5}:.3) -- ({2*360/5}:1);
	\draw[dotted] ({3*360/5}:.3) -- ({3*360/5}:1);
	\draw[dotted] ({4*360/5}:.3) -- ({4*360/5}:1);
	\node[red] at (-90:{.3/2+.65/2}) {\tiny $0$};
	\node[red] at (-90:{1/2+.65/2}) {\tiny $1$};
	\node[blue] at ({360/5-180/5}:1.1) {\tiny $0$};
	\node[blue] at ({2*360/5-180/5}:1.1) {\tiny $1$};
	\node[blue] at ({3*360/5-180/5}:1.1) {\tiny $2$};
	\node[blue] at ({4*360/5-180/5}:1.1) {\tiny $3$};
	\node[blue] at ({5*360/5-180/5}:1.1) {\tiny $4$};

	\draw [red,thick] ({360/5}:.65cm) -- ({360/5}:1cm);
	\draw [PineGreen,thick] ({2*360/5}:.65cm) -- ({2*360/5}:1cm);
	\draw [yellow!50!orange,thick] ({3*360/5}:.65cm) -- ({3*360/5}:1cm);
	\draw [PineGreen,thick,xshift=.5mm,yshift=-.5mm] ({3*360/5}:.65cm) -- ({3*360/5}:1cm);
	\draw [yellow!50!orange,thick] ({4*360/5}:.65cm) -- ({4*360/5}:1cm);

	\draw (0,0) circle (1cm);
	\draw[decoration={markings, mark=at position 0.5 with {\arrow{<}}},	postaction={decorate}]  (0,0) circle (.3cm);    
    \end{scope}
    
    \begin{scope}[xshift=-3cm,yshift=-3cm]
    \node at (0,0) {\footnotesize $1$};
    \draw (180:1cm) -- (180:1.2cm);
    \node at (180:1.2cm) [left] {\footnotesize $1$};
    
    \draw[dotted] (0,0) circle (.65cm);   	
    \draw[dashed] (0:.3) -- (0:1);
    \draw[dotted] (180:.3) -- (180:1);
    \node[red] at (-90:{.3/2+.65/2}) {\tiny $0$};
    \node[red] at (-90:{1/2+.65/2}) {\tiny $1$};
    \node[blue] at (90:1.1) {\tiny $0$};
    \node[blue] at (-90:1.1) {\tiny $1$};
    
    \draw [red,thick] (0:.65cm) -- (0:1cm);
    
    \draw (0,0) circle (1cm);
    \draw[decoration={markings, mark=at position 0.5 with {\arrow{<}}},	postaction={decorate}]  (0,0) circle (.3cm);
    
    \node at (1.25,1.1) {$\partial^2_0(L)$};
    
    \begin{scope}[xshift=2.5cm]
    \node at (0,0) {\footnotesize $2$};
    
    \draw[dotted] (0,0) circle (.65cm);   	
    \draw[dashed] (0:.3) -- (0:1);
    \draw[dotted] ({360/4}:.3) -- ({360/4}:1);
    \draw[dotted] ({2*360/4}:.3) -- ({2*360/4}:1);
    \draw[dotted] ({3*360/4}:.3) -- ({3*360/4}:1);
    \node[red] at (-90:{.3/2+.65/2}) {\tiny $0$};
    \node[blue] at ({360/4-180/4}:1.1) {\tiny $0$};
    \node[blue] at ({2*360/4-180/4}:1.1) {\tiny $1$};
    \node[blue] at ({3*360/4-180/4}:1.1) {\tiny $2$};
    \node[blue] at ({4*360/4-180/4}:1.1) {\tiny $3$};
    
    \draw [red,thick] (0:.65cm) -- (0:1cm);
    \draw [PineGreen,thick] ({360/4}:.65cm) -- ({360/4}:1cm);
    \draw [yellow!50!orange,thick] ({2*360/4}:.65cm) -- ({2*360/4}:1cm);
    \draw [PineGreen,thick,yshift=-.7mm] ({2*360/4}:.65cm) -- ({2*360/4}:1cm);
    \draw [yellow!50!orange,thick] ({3*360/4}:.65cm) -- ({3*360/4}:1cm);
    
    \draw (0,0) circle (1cm);
    \draw[decoration={markings, mark=at position 0.5 with {\arrow{<}}},	postaction={decorate}]  (0,0) circle (.3cm);    
    \end{scope}
    \end{scope}
    
    \begin{scope}[xshift=3cm,yshift=-3cm]
    \node at (0,0) {\footnotesize $1$};
    \draw (180:1cm) -- (180:1.2cm);
    \node at (180:1.2cm) [left] {\footnotesize $1$};
    
    \draw[dotted] (0,0) circle (.65cm);   	
    \draw[dashed] (0:.3) -- (0:1);
    \draw[dotted] (180:.3) -- (180:1);
    \node[red] at (-90:{.3/2+.65/2}) {\tiny $0$};
    \node[red] at (-90:{1/2+.65/2}) {\tiny $1$};
    \node[blue] at (90:1.1) {\tiny $0$};
    \node[blue] at (-90:1.1) {\tiny $1$};
    
    \draw [red,thick] (0:.65cm) -- (0:1cm);
    
    \draw (0,0) circle (1cm);
    \draw[decoration={markings, mark=at position 0.5 with {\arrow{<}}},	postaction={decorate}]  (0,0) circle (.3cm);
    
    \node at (1.25,1.1) {$\partial^2_3(L)$};
    
    \begin{scope}[xshift=2.5cm]
    \node at (0,0) {\footnotesize $2$};
    
    \draw[dotted] (0,0) circle (.65cm);   	
    \draw[dashed] (0:.3) -- (0:1);
    \draw[dotted] ({360/4}:.3) -- ({360/4}:1);
    \draw[dotted] ({2*360/4}:.3) -- ({2*360/4}:1);
    \draw[dotted] ({3*360/4}:.3) -- ({3*360/4}:1);
    \node[red] at (-90:{.3/2+.65/2}) {\tiny $0$};
    \node[blue] at ({360/4-180/4}:1.1) {\tiny $0$};
    \node[blue] at ({2*360/4-180/4}:1.1) {\tiny $1$};
    \node[blue] at ({3*360/4-180/4}:1.1) {\tiny $2$};
    \node[blue] at ({4*360/4-180/4}:1.1) {\tiny $3$};
    
    \draw [red,thick] ({360/4}:.65cm) -- ({360/4}:1cm);
    \draw [PineGreen,thick] ({2*360/4}:.65cm) -- ({2*360/4}:1cm);
    \draw [yellow!50!orange,thick,xshift=-.7mm] ({3*360/4}:.65cm) -- ({3*360/4}:1cm);
    \draw [PineGreen,thick] ({3*360/4}:.65cm) -- ({3*360/4}:1cm);
    \draw [yellow!50!orange,thick,xshift=.7mm] ({3*360/4}:.65cm) -- ({3*360/4}:1cm);
    
    \draw (0,0) circle (1cm);
    \draw[decoration={markings, mark=at position 0.5 with {\arrow{<}}},	postaction={decorate}]  (0,0) circle (.3cm);    
    \end{scope}
    \end{scope}
    \end{tikzpicture}
    \caption{A second example of a cell and parts of its boundary. Here all slits have the same length.}
    \label{sUbRad_example}
\end{figure}

\begin{definition}\label{deffrrad} If a combinatorial type $[\cL]$ is degenerate, then $d^i_j([\cL])$ is also degenerate.  Thus, we define the \emph{cell complex of degenerate radial slit configurations} as the subcomplex $\bRad_h(n,m)'\subset \bRad_h(n,m)$ obtained as the realization of the degenerate simplices.  Finally, the $\Rad_h(n,m)$ is the complement.  That is 
\[\Rad_h(n,m)\coloneqq \bRad_h(n,m) \setminus \bRad_h(n,m)'\]
The spaces $\PRad_h(n,m)$ and $\QRad_h(n,m)$ are defined in a similar way.\end{definition}

We introduce notation for the image of $e_{[\cL]}$ in $\Rad$.

\begin{definition}\label{defcombtypeclosedcell} 
Let $[\cL]$ be a combinatorial type, we define the subspace $\Rad_{[\cL]}$ as image of the interior of $e_{[\cL]}$. We also let $\bRad_{[\cL]}$ be the closure of $\Rad_{[\cL]}$ in $\bRad$ and define $\partial \bRad_{[\cL]} = \Rad \cap (\bRad_{[\cL]} \setminus \Rad_{[\cL]})$.
\end{definition}

\subsubsection{Relationships} Our final goal for this section is to explain the relationship between the spaces and cell complexes of radial slit configurations, and the moduli space of cobordisms. The first relationship is straightforward, as there are obvious continuous bijections 
\begin{align*}\rad_h(n,m) \lra \Rad_h(n,m), \qquad & \qquad \brad_h(n,m) \lra \bRad_h(n,m), \\
\qrad_h(n,m) \lra \QRad_h(n,m), \qquad & \qquad \bqrad_h(n,m) \lra \bQRad_h(n,m), \\
\prad_h(n,m) \lra \PRad_h(n,m), \qquad & \qquad \bprad_h(n,m) \lra \bPRad_h(n,m),
\end{align*} 
compatible with the quotient maps and inclusions. These are given by sending a point to its combinatorial type and the simplicial coordinates obtained by rescaling the angles of the slits (for the first $n$ coordinates) and their radii (for the last coordinate).  The following Lemma follows from \cite{bodigheimer} and we sketch a proof below. 

\begin{lemma}These maps are homeomorphisms.\end{lemma}

\begin{proof}We start by noting that $\bprad_h(n,m)$ and $\bPRad_h(n,m)$ are both compact Hausdorff spaces; the former is a closed subset of a compact Hausdorff space and the latter is a finite CW-complex. A continuous bijection between compact Hausdorff spaces is a homeomorphism. Next note that the maps $\brad_h(n,m) \to \bRad_h(n,m)$ and $\bqrad_h(n,m) \to \bQRad_h(n,m)$ are induced by passing to quotients, as are their inverses, so they are also homeomorphisms.

Thus the right maps are homeomorphisms and the left maps are obtained by restricting these homeomorphisms to open subsets and replacing their codomain with their image. Hence they are also homeomorphisms.\end{proof}

The relationship to moduli space is less straightforward. In Section 9 of \cite{bodigheimer}, B\"{o}digheimer defined a space $\bRAD_h(n,m)$ of all radial slit configurations with varying inner radii, but fixed outer radii and a subspace $\RAD_h(n,m)$ of all non-degenerate radial slit configurations. He also proved a version of the previous lemma.

\begin{lemma}There are homotopy equivalences
\begin{align*}\RAD_h(n,m) \simeq \rad_h(n,m) \qquad  & \qquad \bRAD_h(n,m) \simeq \brad_h(n,m)
\end{align*}
\end{lemma}

\begin{proof}[Sketch of proof]To explain the existence of these homotopy equivalences, we note that B\"odigheimer's $\bRAD$ and $\RAD$ differ from $\brad$ and $\rad$ only in the following two ways:
\begin{enumerate}[(i)]
	\item In $\bRAD$ and $\RAD$, the inner radii are allowed to vary in $(0,R_0)$ for some choice of $R_0 > 0$, while in $\brad$ and $\rad$ they are fixed to $\frac{1}{2\pi}$.
	\item In $\bRAD$ and $\RAD$, an exceptional set $\Omega$ is used to remove ambiguity when all slits on an annulus lie on two segments, while in $\brad$ and $\rad$ this role is played by the angular distances $\vec{r}$.
\end{enumerate}
The second of these encodes equivalent data; given the rest of the data of a radial slit configuration, $\Omega$ can be reconstructed from $\vec{r}$ and vice versa. The first says that the difference between the two spaces is in the choices of radii. More precisely, there is an inclusion $\brad \hookrightarrow \bRAD$ with homotopy inverse given by decreasing all radii to $\min(R_i)$ and changing the radial coordinates of all the data by an affine transformation that sends $\min(R_i)$ to $\frac{1}{2\pi}$ and fixes 1. This homotopy equivalence restricts to one between $\RAD$ and $\rad$.
\end{proof}

B\"odigheimer proved in Section 7.5 of \cite{bodigheimer}, with additional details in \cite{ebertradial}, that a version of $\RAD_h(n,m)$ without parametrization points on the outgoing boundary, is a model for the moduli space of cobordisms without parametrization of the outgoing boundary. This uses that $\Sigma(L)$ comes with a canonical conformal structure, being obtained by gluing subsets of $\bC$. Adding in the parametrizations for the outer boundary, this result implies:

\begin{theorem}[B\"odigheimer] \label{thm_bodig} The map that assigns to each $[L] \in \RAD_h(n,m)$ the conformal class of the cobordism $S(L)$ gives a homeomorphism
\[\RAD_h(n,m) \cong \bigsqcup \cM_g(n,m),\]
where the disjoint union is over triples $(g,n,m)$ satisfying $h = 2g-2+n+m$.
\end{theorem}

By the remarks above we have
\[\Rad_h(n,m) \simeq \bigsqcup_{[\Sigma]} B\mr{Diff}(\Sigma,\partial \Sigma),\]
where the disjoint union is over two-dimensional cobordisms with $n \geq 1$ incoming boundary components, $m \geq 1$ outgoing boundary components, and total genus $g \geq 0$.

B\"odigheimer proved Theorem \ref{thm_bodig} for connected cobordisms with no parametrization of the outgoing boundary, but this version of the theorem is an easy consequence of his. His proof amounts to checking that $\RAD_h(n,m)$ is a manifold of dimension $3h+m+n$ (see also \cite{ebertfriedrich} for remarks on the real-analytic structure). It sits as a dense open subset in $\bRAD_h(n,m)$. In this way we can think of $\bRAD_h(n,m)$ as a ``compactification'' of $\RAD_h(n,m)$. Informally, it is the compactification where handles or boundary components can degenerate to radius zero, as long as there is always a path from each incoming boundary component to an outgoing boundary component that does not pass through any degenerate handles or boundary components. Colloquially, ``the water must always be able to leave the tap.'' B\"{o}digheimer calls this the \emph{harmonic compactification of moduli space}. We now describe a deformation retract of it:

\begin{definition} \label{def_ubrad} The \emph{unilevel harmonic compactification} $\UbRad_h(n,m)$ is the subspace of $\bRad_h(n,m)$ given by cells corresponding to configurations satisfying $|\zeta_i| = R$ for all $i \in \{1,\ldots,2h\}$, i.e. all slits lie on the outer radius.\end{definition}

In addition to the inclusion $\iota \colon \UbRad_h(n,m) \hookrightarrow \bRad_h(n,m)$, there is also a projection $p \colon \bRad_h(n,m) \to \UbRad_h(n,m)$ which makes all slits have modulus $R$.

\begin{lemma}\label{lem_bradubrad} The maps $\iota$ and $p$ are mutually inverse up to homotopy.\end{lemma}

\begin{proof}The map $p \circ \iota$ is equal to the identity on $\UbRad$. For $\iota \circ p$, a homotopy from the identity on $\bRad$ to $\iota \circ p$ is given at time $t \in [0,1]$ sending each slit $\zeta_i$ to $\frac{(1-t)|\zeta_i| + Rt}{|\zeta_i|}\zeta_i$ under the homeomorphism with $\rad$.
\end{proof}

The spaces constructed in this section fit together in the following diagram
\[\begin{tikzcd}
	\PRad_h(n,m) \arrow{rrr}{\text{compactification}} \dar & & & \bPRad_h(n,m) \dar &  \\[-5pt]
	\QRad_h(n,m) \arrow{rrr}{\text{compactification}} \dar & & & \bQRad_h(n,m) \dar &  \\[-5pt]
	\Rad_h(n,m) \arrow{rrr}{\text{compactification}} & & & \bRad_h(n,m) \ar[shift left=.5 ex]{r}{\simeq} & \UbRad_h(n,m), \lar[shift left=.5ex] \end{tikzcd}\]
where all the horizontal maps are inclusions.

\begin{remark}One can make sense of glueing of cobordisms on the level of radial slits, see \cite{bodigheimer}. This construction gives $\RAD_h(n,m)$ the structure of a prop in topological spaces. One of the advantages of the radial slit configurations over fat graphs is the ease with which one can describe the prop structure.\end{remark}

\subsection{The universal surface bundle} In the previous section, we motivated radial slit configurations by explaining that a preconfiguration consists of data to construct a cobordism $S(L)$. The topology on the collection of radial slit configurations was guided by the idea that this construction produces a conformal family of cobordisms. In this section we make this precise by defining a universal surface bundle over $\Rad$ via its homeomorphism with $\rad$.

The equivalence relation $\equiv$ on $\mr{PRad}_h(n,m)$ is such that there is a canonical isomorphism of cobordisms with conformal structure between $S(L)$ and $S(L')$ if $L \equiv L'$. Thus we can make sense of the cobordism $S([L])$ for an equivalence class $[L]$. The idea for constructing the universal surface bundle over $\rad_h(n,m)$, is to make the construction of $S([L])$ continuous in $[L]$. The result is a space over $\rad_h(n,m)$, and we check it is a universal bundle by comparing it to the definition of the universal bundle in the conformal construction of moduli space.

We first make sense of the radial sectors $\overline{\Sigma}(L)$ as a space over $\bprad_h(n,m)$. This seems obvious; we think of the sectors as a subspace of a disjoint union of annuli for each $L$, so one is tempted to just state that $\tilde{\Sigma}(L)$ is the relevant subspace of 
\[\bprad_h(n,m) \times \left(\bigsqcup_{j=1}^n \bA^{(j)}_{R}\right).\]
Two minor problems arise: (i) the full sectors are not actually subspaces of annuli and (ii) the number of entire sectors is not constant over $\bprad_h(n,m)$. 

Both problems are relatively harmless: problem (ii) is solved by noting that the number of entire sectors is locally constant, so one can work separately over each of the subspaces of components with a fixed number of entire sectors. Problem (i) is solved by considering a version of $\bprad_h(n,m)$ where the preconfigurations $L$ are endowed with lifts of the slits to elements of $\bigsqcup_{i=1}^n \tilde{\bA}_R$, the disjoint union of the universal covers of the annuli, under the condition that the distances between them are still equal to the angular distances. Over this version one has a space with fibers given by $\bigsqcup_{i=1}^n \tilde{\bA}_R$, which does contain the full sectors. One then notes that there is a canonical homeomorphism between the sectors over the same configurations with different choices of lifts. In the end, we conclude there exists a space $\tilde{\bA}$ over $\bprad_h(n,m)$ whose fibers consist of a disjoint union of annuli, and there is a subspace $\overline{\mr{PS}}_h(n,m) \subset \tilde{\bA}$ whose fiber over $L$ can be canonically identified with the sector space $\tilde{\Sigma}(L)$.

Recall that $\approx_L$ is the equivalence relation on $\overline{\Sigma}(L)$ used when glueing the sectors together to obtain a surface. Using it fiberwise defines an equivalence relation $\sim$:

\begin{definition}Let $\sim$ be the equivalence relation on $\overline{\mr{PS}}_h(n,m)$ generated by $(L,z) \sim (L',z')$, where $L,L' \in \bprad_h(n,m)$, $z \in \overline{\Sigma}(L) \subset \overline{\mr{PS}}_h(n,m)$ and $z' \in \overline{\Sigma}(L') \subset \overline{\mr{PS}}_h(n,m)$, if $L = L'$ and $z \approx_L z'$.\end{definition}

As mentioned before, there is a canonical isomorphism $\phi_{L,L'}$ between $\Sigma(L)$ and $\Sigma(L')$ if $L \equiv L'$. Using this we can define a version of $\equiv$ for $\overline{\mr{PS}}_h(n,m)$.

\begin{definition}Let $\cong$ be the equivalence relation on $\overline{\mr{PS}}_h(n,m)$ generated by $\sim$ and by saying that $(L,z)$ and $(L',z')$ are equivalent if $L \equiv L'$ and $z' = \phi_{L,L'}(z)$.\end{definition}

We can now define the surface bundle.

\begin{definition}We define $\mr{PS}_h(n,m)$ to be the restriction of $\overline{\mr{PS}}_h(n,m)$ to $\rad_h(n,m)$. We then define $\mr{S}_h(n,m)$ as $\mr{PS}_h(n,m)/{\cong}$, which is a space over $\rad_h(n,m)$.\end{definition}

A priori this is a space over $\rad_h(n,m)$ with fibers having the structure of cobordisms, but it is in fact a universal surface bundle. This is implicit in \cite{bodigheimer} but not explicitly stated there. We explain the reasoning below:

\begin{proposition}The space $\mr{S}_h(n,m)$ over $\rad_h(n,m)$ is a universal surface bundle.\end{proposition}

\begin{proof}[Sketch of proof] Varying radii allows one to extend $\mr{S}_h(n,m)$ to $\RAD_h(n,m)$. Theorem \ref{thm_bodig} tells us that the assignment $[L] \mapsto [S([L])]$ gives a homeomorphism $\RAD_h(n,m) \to \cM_g(n,m)$. Pulling back the universal bundle over $\cM_ g(n,m)$ defined at the end of Subsection \ref{subsec_confmodel} exactly gives $\mr{S}_h(n,m)$.\end{proof}

There is a universal $\mr{Mod}(S_{g,n+m})$-bundle over $\rad_h(n,m)$ given by the bundle with fiber over $[L]$ the isotopy classes diffeomorphisms of $\Sigma(L)$ fixing the boundary. We give an alternative explicit construction of this bundle in Definition \ref{deferad}.

%% file: fatgraphs.tex
\subsection{The definition}
%\hspace{20 mm}
Following the ideas of Strebel \cite{strebel}, Penner, Bowditch and Epstein  gave a triangulation of Teichm\"{u}ller space of surfaces with decorations, which is equivariant under the action of its corresponding mapping class group \cite{penner,bowditch_epstein}.  In this triangulation, simplices correspond to equivalence classes of marked fat graphs and the quotient of this triangulation gives a combinatorial model of the moduli space of surfaces with decorations. These ideas were studied by Harer for surfaces with punctures and boundary components \cite{Harer_arc} and used by Igusa to construct a category of fat graphs that models the mapping class groups of punctured surfaces \cite{igusahrt}. Godin extended Igusa's construction to surfaces with boundary and open-closed cobordisms \cite{godinunstable, godin}. 

In this section we define a category of fat graphs, as well as specific subcategories of it, in the spirit of Godin.  We also define the space of metric fat graphs in the spirit of Harer and Penner, as well as specific subspaces of these spaces, and show that these are the classifying spaces of these categories. Finally, we define the space of Sullivan diagrams as a quotient of a certain subspace of the space of metric fat graphs. It plays the role of a compactification.

\subsubsection{Fat graphs}We start with precise definitions of graphs and fat graphs.

\begin{definition}
A \emph{combinatorial graph} $G$ is a tuple $G=(V,H,s,i)$, with a finite set of \emph{vertices} $V$, a finite set of \emph{half edges} $H$, a \emph{source map} $s:H\to V$ and an \emph{edge pairing} involution $i:H\to H$ without fixed points.
\end{definition}

The source map $s$ ties each half edge to its source vertex, and the edge pairing involution $i$ attaches half edges together.  The set $E$ of \emph{edges} of the graph is the set of orbits of $i$. The \emph{valence} of a vertex $v\in V$ is the cardinality of the set $s^{-1}(v)$. A \emph{leaf} of a graph is a univalent vertex and an \emph{inner vertex} is a vertex that is not a leaf.  The \emph{geometric realization} of a combinatorial graph $G$ is the CW-complex $|G|$ with one 0-cell for each vertex, one 1-cell for each edge and attaching maps given by $s$ and $s \circ i$. A \emph{tree} is a graph whose geometric realization is a contractible space and a \emph{forest} is a disjoint union of trees.

\begin{definition}
A \emph{fat graph} $\Gamma=(G,\sigma)$ is a combinatorial graph together with a cyclic ordering $\sigma_v$ of the half edges incident at each vertex $v$. The \emph{fat structure} of the graph is given by the data $\sigma=(\sigma_v)$ which is a permutation of the half edges. 
\end{definition}

\begin{figure}[h!]
  \centering
    \begin{tikzpicture}
    \node at (0,1) {$\circlearrowleft$};
   	\draw (-2,0) circle (1cm);
   	\draw (-2,-1) -- (-2,1);
   	
   	\draw (2,1) arc (90:270:1);
   	\draw (2,1) to[out=0,in=160,looseness=2.5] (2,-1);
   	\draw [line width=1mm,white] (2,-.95) -- (2,.95);
   	\draw (2,-1) -- (2,1);
    \end{tikzpicture}
  \caption{Two different fat graphs  -- where the fat structure is given by the orientation of the plane, here denoted by the circular arrow -- with the same underlying combinatorial graph.}
  \label{Fat_example}
\end{figure}

From a fat graph $\Gamma=(G,\sigma)$ one can construct a surface with boundary $\Sigma_{\Gamma}$ by thickening the edges and the vertices. More explicitly, one can construct this surface by replacing each edge with a strip and glueing these strips to a disk at each vertex according to the fat structure. The cyclic ordering exactly gives the data required to do this. Notice that there is a strong deformation retraction of $\Sigma_{\Gamma}$ onto $|G|$ so one can think of $|G|$ as the skeleton of the surface.  

\begin{definition}
The \emph{boundary cycles} of a fat graph are the cycles of the permutation of half edges given by $\omega=\sigma\circ i$. Each cycle $\tau$ of $\omega$ gives a list of edges of the graph $\Gamma$ and thus determines a subgraph $\Gamma_{\tau}\subset \Gamma$, which we call the \emph{boundary graph} corresponding to $\tau$.
\end{definition}

\begin{remark}
Note that the fat structure of $\Gamma$ is completely determined by $\omega$. Moreover, one can show that the boundary cycles of a fat graph $\Gamma=(G,\omega)$ correspond to the boundary components of $\Sigma_{\Gamma}$ (cf. \cite{godinunstable}).  Therefore, the surface $\Sigma_{\Gamma}$ is completely determined up to topological type by the combinatorial graph and its fat structure.
\end{remark}

A fat graph gives one a surface, but not yet a cobordism. The difference is that it does not distinguish between incoming and outgoing boundary components, nor do these come with canonical parametrizations. Note that after deciding whether a boundary component is incoming or outgoing, a parametrization is uniquely determined once we pick a marked point and edge-lengths. Thus it suffices to add to each boundary component a leaf labeled either ``incoming'' or ``outgoing.''

\begin{definition}
A \emph{closed fat graph} $\Gamma=(\Gamma,L_\mr{in},L_\mr{out})$ is a fat graph with an ordered set of leaves and a partition of this set of leaves into two sets $L_\mr{in}$ and $L_\mr{out}$, such that:
\begin{enumerate}[(i)]
\item all inner vertices are at least trivalent,
\item there is exactly one leaf on each boundary cycle.  Given a leaf $l_i$ we denote its corresponding boundary graph by $\Gamma_{l_i}\subset \Gamma$. 
\end{enumerate}
Leafs in $L_\mr{in}$ or in $L_\mr{out}$, are called \emph{incoming} or \emph{outgoing} respectively.
\end{definition}

Note that the previous definition also removed unnecessary bivalent and univalent vertices. It turns out that one can consider an even more restricted type of fat graph, which reflects that (like in radial slits) we can decide to arrange the incoming boundary in a special way.
 
\begin{figure}
  \centering
    \begin{tikzpicture}
    	\clip (-1.5,-1.35) rectangle (8,1.35);
    	\draw (0,0) circle (1 cm);
    	\draw (2,0) circle (1 cm);
    	\draw [decoration={markings, mark=at position 0.5 with {\arrow{>}}},	postaction={decorate},blue] (135:.75) -- (135:1);
    	\node [blue] at (135:.5) {$3$};
    	\draw [decoration={markings, mark=at position 0.5 with {\arrow{<}}},	postaction={decorate},xshift=2cm] (-135:.75) -- (-135:1);
    	\node [xshift=2cm] at (-135:.5) {$4$};
    	\node at (1,0) {$\bullet$};
    	\draw (3,0) to[out=65,in=-115,looseness=2.5] (5,0);
    	\draw [line width = 2mm,white] (3.5,0) -- (4.5,0);
    	\draw (3,0) -- (6,0);
    	\node at (5,0) {$\bullet$};
    	\draw [xshift=7cm] (0,0) circle (1cm);
      	\draw [decoration={markings, mark=at position 0.5 with {\arrow{>}}},	postaction={decorate},xshift=7cm,blue] (180:.75) -- (180:1);  
      	\node [xshift=7cm,blue] at (180:.5) {$5$};	
    	\draw [decoration={markings, mark=at position 0.5 with {\arrow{>}}},	postaction={decorate},blue] (180:1.25) -- (180:1);
    	\node [blue] at (180:1.5) {$1$};
    	\draw [decoration={markings, mark=at position 0.75 with {\arrow{>}}},	postaction={decorate}] (5,0) -- (5,.3);
    	\node at (5,.5) {$2$};
    \end{tikzpicture}
  \caption{An example of a closed fat graph which is not admissible.  The incoming and outgoing leaves are marked by incoming or outgoing arrows.}
  \label{openclosed_example}
\end{figure}

\begin{definition}
Let $\Gamma$ be a closed fat graph.  Let $l_i$ denote a leaf of $\Gamma$ and $\Gamma_{l_i}\subset \Gamma$ be its corresponding boundary graph.  $\Gamma$ is called \emph{admissible} if the subgraphs $\Gamma_{l_i}-l_i$ for all incoming leaves $l_i$ are disjoint embedded circles in $\Gamma$. We refer to these boundary cycles as \emph{admissible cycles} (see Figure \ref{Admissible_example}).
\end{definition}

\begin{figure}[h!]
  \centering
    \begin{tikzpicture}[scale=0.95]
    	\clip (-1,-3.25) rectangle (11,3);
    	\draw (0,0) circle (1cm);
    	\draw (0,-1) -- (0,1);
    	\draw [decoration={markings, mark=at position 0.75 with {\arrow{>}}},	postaction={decorate}] (0:1) -- (0:.7);
    	\node at (0:.5) {$2$};
    	\draw [decoration={markings, mark=at position 0.75 with {\arrow{>}}},	postaction={decorate},blue] (0:1.3) -- (0:1);
		\node at (0:1.5) [blue] {$1$};    	
		\draw [decoration={markings, mark=at position 0.75 with {\arrow{>}}},	postaction={decorate}] (0,0) -- (-.3,0);
		\node at (-.5,0) {$3$};
		
		\begin{scope}[yshift=-1cm,xshift=4.5cm]
		
		\draw (0,0) circle (1cm);
		\draw (2.5,0) circle (1cm);
		\draw (5,0) circle (1cm);
	 	\draw [decoration={markings, mark=at position 0.75 with {\arrow{<}}},	postaction={decorate},blue] (0:1) -- (0:.7);
		\node at (0:.5) [blue] {$1$};
	 	\draw [decoration={markings, mark=at position 0.75 with {\arrow{<}}},	postaction={decorate},blue,xshift=2.5cm] (0:1) -- (0:.7);
		\node at (0:.5) [blue,xshift=2.5cm] {$3$};
		\draw [decoration={markings, mark=at position 0.75 with {\arrow{<}}},	postaction={decorate},blue,xshift=5cm] (0:1) -- (0:.7);
		\node at (0:.5) [blue,xshift=4.7cm] {$2$};
		\draw [decoration={markings, mark=at position 0.75 with {\arrow{>}}},	postaction={decorate},xshift=2.5cm] (-75:1) -- (-75:1.3);
		\node at (-75:1.5) [xshift=2.5cm] {$3$};
		
		\draw (1,0) to[out=-45,in=180] (2.5,-2) to[out=0,in=-135] ($(5,0) + (-135:1)$);
		\draw (-1,0) to[out=135,in=180] (-.2,1.5) to[out=0,in=165] ($(2.5,0) + (135:1)$);
		\draw ($(2.5,0) + (135:1)$) to[out=125,in=180] ($(2.5,0) + (100:2)$) to[out=0,in=45] ($(2.5,0) + (45:1)$);
		\draw [decoration={markings, mark=at position 0.75 with {\arrow{>}}},	postaction={decorate}] ($(2.5,0) + (100:2)$) -- ($(2.5,0) + (100:1.7)$);
		\node at ($(2.5,0) + (100:1.5)$) {$5$};
		\draw [xshift=2cm,yshift=3cm] (100:.5) to[out=70,in=45,looseness=3] (-135:.5);
		\draw [line width = 2mm,white] (2,3) ++ (90:.5) arc (90:45:.5);
		\draw (2,3) circle (.5cm);
		\draw ($(2.5,0) + (100:2)$) -- (2,2.5);

		\end{scope}
    \end{tikzpicture}
  \caption{Two examples of admissible fat graphs.  The first graph has the topological type of the pair of pants, and second graph that of a surface of genus $1$ with $5$ boundary components.}
  \label{Admissible_example}
\end{figure}

We organize fat graphs into a category. The idea is that when we use fat graphs to construct surfaces, we should be able to pick different lengths for the edges to obtain different conformal classes. Furthermore, if the length of an edge goes to zero, we expect the two disks corresponding to the vertices to be glued together. This makes sense as long as the edge is not a loop. The morphisms in the category of fat graphs encode this relationship between graphs. Recall that a tree is a graph whose geometric realization is contractible and a forest is a disjoint union of trees.

\begin{definition}\label{def_fatfatad}We define two categories:\begin{itemize}
\item The \emph{category of closed fat graphs $\Fat$} is the category with objects isomorphism classes of closed fat graphs and morphisms $[\Gamma]\to[\Gamma/F]$ given by collapsing to a point in each tree in a subforest of $\Gamma$ that does not contain any leaves.
\item The \emph{category of admissible fat graphs $\Fatad$} is the full subcategory of $\Fat$ with objects isomorphism classes of admissible fat graphs.
\end{itemize}
\end{definition}

The composition in $\Fat$ and $\Fatad$ and hence the categories themselves, are well defined. The category $\Fat$ was introduced by Godin in \cite{godinunstable} and $\Fatad$ is a slight variation of it introduced by the same author in \cite{godin}.

Note that the collapse of a subforest which does not contain any leaves induces a surjective homotopy equivalence upon geometric realizations and does not change the number of boundary components. Therefore, if there is a morphism $\varphi \colon [\Gamma]\to[\tilde{\Gamma}]$ between isomorphism classes of fat graphs, then the surfaces $\Sigma_{[\Gamma]}$ and $\Sigma_{[\tilde{\Gamma}]}$ are homeomorphic.

From a closed fat graph we can construct a two-dimensional cobordism.  The underlying surface of the cobordism is the oriented surface $\Sigma_{\Gamma}$.  This gives an orientation of the incoming and outgoing boundary component, so its enough to give a labeled marked point in each boundary component.  Note that each of the boundary components corresponds to exactly one leaf in the graph, which gives a marked point in the boundary component. We label this according to the labeling of its leaf. This gives a cobordism, well-defined up to isomorphism.

\subsubsection{Metric fat graphs}
We motivated the morphisms in the category of fat graphs by thinking about lengths of edges. This is made more concrete in the space of metric fat graphs, which we describe now. This space has a deformation retraction onto the classifying space of the category of fat graphs, but we feel metric fat graphs are more intuitive and hence discuss them first. Several equivalent versions of this space and its dual concept (using weighted arc systems instead of fat graphs) have been studied by Harer, Penner, Igusa and Godin \cite{Harer_cohomology_moduli, penner, igusahrt, godin_thesis}. 

The idea is simple: a metric fat graph is a fat graph with lengths assigned to its edges. We need a bit more care to make this interact well with the additional data and properties of admissible fat graphs.

\begin{definition}
A \emph{metric admissible fat graph} is a pair $(\Gamma,\lambda)$ where $\Gamma$ is an admissible fat graph and $\lambda$ is a \emph{length function}, i.e. a function $\lambda \colon E_{\Gamma}\to [0,1]$ where $E_{\Gamma}$ is the set of edges of $\Gamma$ and $\lambda$ satisfies:
\begin{enumerate}[(i)]
\item $\lambda(e)=1$ if $e$ is a leaf,
\item $\lambda^{-1}(0)$ is a forest in $\Gamma$ and $\Gamma/\lambda^{-1}(0)$ is admissible,
\item for any admissible cycle $C$ in $\Gamma$ we have $\sum_{e\in C}\lambda(e) = 1$.
\end{enumerate}
\end{definition}

We will call the value of $\lambda$ on $e$ the \emph{length} of the edge $e$ in $\Gamma$.

\begin{definition}
Suppose $\Gamma$ is an admissible fat graph with $p$ admissible cycles. Let $(n_1,n_2,\ldots,n_p)$ be the number of edges on each admissible cycle and set $n \coloneqq \sum_i n_i$.  The \emph{space of length functions} on $\Gamma$ is given as a set by
\[\mathpzc{M}(\Gamma) \coloneqq \lbrace\lambda \colon E_{\Gamma}\to[0,1]\,\vert\,\lambda \text{ is a length function}\rbrace\]
There is a natural inclusion
\[\mathpzc{M}(\Gamma)\cof \Delta^{n_1-1}\times\Delta^{n_2-1}\times\cdots \times \Delta^{n_p-1}\times([0,1])^{\# E_{\Gamma}-n}\]
we give $\mathpzc{M}(\Gamma)$ the subspace topology via this inclusion.
\end{definition}

\begin{definition}
Two metric admissible fat graphs $(\Gamma,\lambda)$ and $(\tilde{\Gamma}, \tilde{\lambda})$ are called \emph{isomorphic} if there is an isomorphism of admissible fat graphs $\varphi:\Gamma\to\tilde{\Gamma}$ such that $\lambda=\tilde{\lambda}\circ\varphi_*$, where $\varphi_*$ is the map induced by $\varphi$ on $E_{\Gamma}$.
\end{definition}

\begin{definition}\label{def_mfatad}
The space of \emph{metric admissible fat graphs} is defined as
\[\MFatad \coloneqq \frac{\bigsqcup_{\Gamma}\mathpzc{M}(\Gamma)}{\sim}\]
where $\Gamma$ runs over all admissible fat graphs and the equivalence relation $\sim$ is given by
\[(\Gamma,\lambda)\sim(\tilde{\Gamma},\tilde{\lambda})\Longleftrightarrow
(\Gamma/{\lambda^{-1}(0)},\lambda\vert_{E_{\Gamma}-\lambda^{-1}(0)})
\cong
(\tilde{\Gamma}/{\tilde{\lambda}^{-1}(0)},\tilde{\lambda}\vert_{E_{\tilde{\Gamma}}-\tilde{\lambda}^{-1}(0)})\]
\end{definition}

In other words, (i) we identify isomorphic admissible fat graphs with the same metric and (ii) we identify a metric admissible fat graph with some edges of length $0$ with the metric fat graph in which these edges are collapsed and all other edge lengths remain unchanged.

\begin{lemma}
\label{metric_and_nerve}
There is a deformation retraction of the space of metric admissible fat graphs $\MFatad$ onto the geometric realization of the nerve of $\Fatad$.
\end{lemma}

\begin{proof}
We will first give a continuous map $\iota \colon \vert\Fatad\vert\to\MFatad$. A point $x\in \vert\Fatad\vert$ is represented by $x=([\Gamma_0]\to[\Gamma_1]\to\ldots\to[\Gamma_k],s_0,s_1,\ldots s_k)\in N_k\Fatad\times\Delta^k$, where $N_k$ denotes the set of $k$-simplices of the nerve.  Choose representatives $\Gamma_i$ for $0\leq i\leq k$ and for each $i$, let $C^i_j$ denote the $j$th admissible cycle of $\Gamma_i$, $n^i_j$ denote the number of edges in $C^i_j$ and $m^i$ denote the number of edges that do not belong to the admissible cycles. Each graph $\Gamma_i$ naturally defines a metric admissible fat graph $(\Gamma_0,\lambda_i)$ where $\lambda_i$ is given as follows:
\begin{align*}
	\lambda_i \colon E_{\Gamma_0} & \longrightarrow [0,1]\\
	e&\longmapsto 
	\begin{cases}
		0 & \text{if } e \text{ is collapsed in }\Gamma_i,\\
		{1}/{n^i_j} & \text{if } e\in C^i_j, \\
		{1}/{m^i}& \text{otherwise.}
	\end{cases}\end{align*}
Then define $\iota(x) \coloneqq (\Gamma_0,\sum_{i=0}^k s_i\lambda_i)$.  It is easy to show that this assignment is well defined and respects the simplicial relations of the geometric realization and thus defines a continuous map.  Moreover, it is injective map between Hausdorff spaces with compact image, so is a homeomorphism onto its image. Note that the image of $\iota$ is the subspace of metric graphs where the sum of the lengths of the edges that do not belong to the admissible cycles is 1. 

We now construct a continuous map $r \colon \MFatad\times [0,1]\to \MFatad$ which is a strong deformation retraction of $\MFatad$ onto the image of $\iota$, by rescaling.  Since all the graphs we are considering are finite, we can define a continuous function $g$ as follows:
\begin{align*}
	g \colon \MFatad & \longrightarrow \R^{> 0}\\
	(\Gamma,\lambda )&\longmapsto 
	\sum_{e\in \tilde{E}_{\Gamma}}
	\lambda (e),
\end{align*}
%\[
%\begin{array}{ccl}
%g \colon \MFatad & \longrightarrow & \hspace{7pt} \R^{> 0}\\
%\hspace{7pt} (\Gamma,\lambda )&\longmapsto & 
%\sum_{e\in \tilde{E}_{\Gamma}}
%\lambda (e),
%\end{array}
%\]
where $\tilde{E}_{\Gamma}$ is the set of edges that do not belong to the admissible cycles.  We then define $r$ by linear interpolation as $r((\Gamma,\lambda),t)\coloneqq\left(\Gamma,(1-t)\lambda+t\lambda_g\right)$, where $\lambda_g$ is the rescaled length function given by:
\begin{align*}
	\lambda_g \colon E_{\Gamma} & \longrightarrow \R^{\geq 0}\\
	e&\longmapsto 
	\begin{cases}
		\lambda(e) & \text{if } e \text{ belongs to an admissible cycle,}\\
		\frac{\lambda(e)}{g(\Gamma,\lambda)} & \text{if } e\text{ does not belong to an admissible cycle.}\end{cases}\qedhere\end{align*}
\end{proof}

\begin{remark}
The space $\MFatad$  and the category $\Fatad$ split into components indexed by the topological type of the graphs as two-dimensional cobordisms. That is, we have
\[\MFatad \cong \bigsqcup_{g,n,m} \MFatadg, \qquad \text{and} \qquad \Fatad \cong \bigsqcup_{g,n,m} \Fatadg,\]
where $\MFatadg$ and $\Fatadg$ are the connected components corresponding to admissible fat graphs with $n$ admissible cycles which are homotopy equivalent to a surface of total genus $g$ and $n+m$ boundary components.
\end{remark}

\subsubsection{Sullivan diagrams} We now define a quotient space $\SD$ of $\MFatad$, which we will see in section \ref{sec_sdharmonic} is the analogue of the harmonic compactification for admissible fat graphs. To define it, we first describe an equivalence relation $\sim_{\SD}$ on metric admissible fat graphs.  

\begin{definition}
We say $\Gamma_1 \sim_{\SD} \Gamma_2$ if $\Gamma_2$ can be obtained from $\Gamma_1$ by:
\begin{description}
\item[Slides] Sliding vertices along edges that do not belong to the admissible cycles.
\item[Forgetting lengths of non-admissible edge] Changing the lengths of the edges that do not belong to the admissible cycles. 
\end{description}
\end{definition}

%Figure \ref{Sullivan_example} shows some examples of equivalent admissible fat graphs.
\begin{figure}
  \centering
    \begin{tikzpicture}[scale=0.95]
    	\clip (-2,-7) rectangle (8,2.5);
    	\draw (0,0) circle (1cm);
    	\draw [decoration={markings, mark=at position 0.75 with {\arrow{>}}},	postaction={decorate},blue] (.7,0) -- (1,0);
    	\node at (.5,0) [blue] {\footnotesize $1$};
    	\draw [decoration={markings, mark=at position 0.75 with {\arrow{>}}},	postaction={decorate}] (1,0) -- (1.3,0);
    	\node at (1.5,0)  {\footnotesize $3$};
   		\draw [decoration={markings, mark=at position 0.75 with {\arrow{>}}},	postaction={decorate}] (-135:1) -- (-135:1.3);
    	\node at (-135:1.5)  {\footnotesize $2$};
     	\draw ({.15*360}:1) to[out=100,in=140,looseness=3] node[pos=.5] (nodemid) {} (180:1);	
     	\draw [decoration={markings, mark=at position 0.75 with {\arrow{>}}},	postaction={decorate}] (nodemid.center) -- ($(nodemid.center)+(-45:.3cm)$);
     	\node at ($(nodemid.center)+(-45:.5cm)$) {$1$};
     	\node at (-135:1.5)  {\footnotesize $2$};
     	\draw (-{.2*360}:1) to[out=-100,in=-160,looseness=3] (180:1);
     	\node at ({.15*360/2}:.9) [text opacity=1,opacity=.7,fill=white] {\tiny $0.15$};      	
     	\node at ({.35*360/2+.15*360}:.9) [text opacity=1,opacity=.7,fill=white] {\tiny $0.35$};      	
      	\node at ({.15*360/2+.5*360}:.9) [text opacity=1,opacity=.7,fill=white] {\tiny $0.15$};           	
      	\node at ({.15*360/2+.65*360}:.9) [text opacity=1,opacity=.7,fill=white] {\tiny $0.15$};    
      	\node at ({.2*360/2+.8*360}:.9) [text opacity=1,opacity=.7,fill=white] {\tiny $0.2$}; 
      	
      	\draw[<->] (2,0) -- node[pos=.5,above] {$\sim_{SD}$} (3,0);
      	
      	\begin{scope}[xshift=6cm]
      	\draw (0,0) circle (1cm);
      	\draw [decoration={markings, mark=at position 0.75 with {\arrow{>}}},	postaction={decorate},blue] (.7,0) -- (1,0);
      	\node at (.5,0) [blue] {\footnotesize $1$};
      	\draw [decoration={markings, mark=at position 0.75 with {\arrow{>}}},	postaction={decorate}] (1,0) -- (1.3,0);
      	\node at (1.5,0)  {\footnotesize $3$};
      	\draw [decoration={markings, mark=at position 0.75 with {\arrow{>}}},	postaction={decorate}] (-135:1) -- (-135:1.3);
      	\node at (-135:1.5)  {\footnotesize $2$};
      	\draw ({.15*360}:1) to[out=100,in=140,looseness=3] node[pos=.5] (nodemid) {} (180:1);	
      	\draw [decoration={markings, mark=at position 0.85 with {\arrow{>}}},	postaction={decorate}] ({-.2*360}:1) -- ($({-.2*360}:1)+(-135:.4)$);
      	\node at ($({-.2*360}:1)+(-135:.6)$) {$1$};
      	\node at (-135:1.5)  {\footnotesize $2$};
      	\draw (-{.2*360}:1) to[out=-100,in=-135,looseness=1.5] (-2,2) to[out=45,in=80] ({.15*360}:1);
      	\end{scope}
      	
      	\draw[<->] (1,-2) -- node[pos=.5,below left] {$\sim_{SD}$} (2,-2.5);
      	\draw[<->] (5,-2) -- node[pos=.5,below right] {$\sim_{SD}$} (4,-2.5);
      	
      	 \begin{scope}[xshift=4cm,yshift=-5cm]
      	\draw (0,0) circle (1cm);
      	\draw [decoration={markings, mark=at position 0.75 with {\arrow{>}}},	postaction={decorate},blue] (.7,0) -- (1,0);
      	\node at (.5,0) [blue] {\footnotesize $1$};
      	\draw [decoration={markings, mark=at position 0.75 with {\arrow{>}}},	postaction={decorate}] (1,0) -- (1.3,0);
      	\node at (1.5,0)  {\footnotesize $3$};
      	\draw [decoration={markings, mark=at position 0.75 with {\arrow{>}}},	postaction={decorate}] (-135:1) -- (-135:1.3);
      	\node at (-135:1.5)  {\footnotesize $2$};
      	\draw ({-.2*360}:1) to[out=-100,in=-160,looseness=2] node[pos=.5] (nodemid) {} (180:1);	
      	\draw [decoration={markings, mark=at position 0.85 with {\arrow{>}}},	postaction={decorate}] ({.15*360}:1) -- ($({.15*360}:1)+(115:.4)$);
      	\node at ($({.15*360}:1)+(115:.6)$) {$1$};
      	\node at (-135:1.5) [fill=white,opacity=.7,text opacity=1] {\footnotesize $2$};
      	\draw (-{.2*360}:1) to[out=-90,in=-135,looseness=1.7] (-2,2) to[out=45,in=80] ({.15*360}:1);
      	\end{scope}
    \end{tikzpicture}
  \caption{Three equivalent metric admissible fat graphs. On the last two graphs the lengths of the edges of the admissible cycle have been left out; they equal those of the first graph.}
  \label{Sullivan_example}
\end{figure}

\begin{definition}
A \emph{metric Sullivan diagram} is an equivalence class of metric admissible fat graphs under the relation $\sim_{\SD}$.  
\end{definition}

We can informally think of a Sullivan diagram as an admissible fat graph where the edges not belonging to the admissible cycles are of length zero.

\begin{definition}\label{def_sd}
The space of \emph{Sullivan diagrams} $\SD$ is the quotient space $\SD=\MFatad/{\sim_{\SD}}$. 
\end{definition}

\begin{remark}
A path in $\SD$ is given by continuously moving the vertices on the admissible cycles. This space splits into connected components given by topological type.
\end{remark}

\begin{remark}In Section \ref{sec_sdharmonic} we show $\SD$ has canonical CW-complex structure.  Its cellular chain complex is the complex of (cyclic) Sullivan chord diagrams introduced by Tradler and Zeinalian. It was used by them and later by Wahl and Westerland, to construct operations on the Hochschild chains of symmetric Frobenius algebras \cite{TradlerZeinalian, wahlwesterland}.
\end{remark}

\subsection{The universal mapping class group bundle}

In this section we describe the universal mapping class group bundles over $\Fatad$ and $\MFatad$. Recall that from an admissible fat graph we can construct a cobordism which contains the graph as a deformation retract, though this depends on some choices. The idea for the construction of the universal mapping class group bundle, is that its fiber over an admissible fat graph $\Gamma$ consists all ways that $\Gamma$ can sit in a fixed standard cobordism.

For each topological type of cobordism fix a representative surface $\Sg$ of total genus $g$ with $n$ incoming boundary components and $m$ outgoing boundary components. Fix a marked point $x_k$ in the $k$th incoming boundary for $1\leq k\leq n$ and a marked point $x_{k+n}$ in the $k$th outgoing boundary $1\leq k\leq m$. 

\begin{definition}
Suppose $\Gamma$ is an admissible fat graph of topological type $\Sg$. Let $v_{\mr{in},k}$ denote the $k$th incoming leaf and $v_{\mr{out},k}$ denote the $k$th outgoing leaf.  A \emph{marking} of $\Gamma$ is an isotopy class of embeddings $H \colon |\Gamma|\cof \Sg$ such that $H(v_{\mr{in},k})=x_k$, $H(v_{\mr{out},k})=x_{k+n}$ and the fat structure of $\Gamma$ coincides with the one induced by the orientation of the surface. We will call a pair $(\Gamma,[H])$ a \emph{marked fat graph} and denote by $\Mark(\Gamma)$ the \emph{set of markings of $\Gamma$}.
\end{definition}

\begin{lemma} \label{marking_complement} Any marking $H \colon |\Gamma|\cof \Sg$ is a homotopy equivalence, and the map on $\pi_1$ induced by $H$ sends the $i$th boundary cycle of $\Gamma$ to the $i$th boundary component of $\Sg$.\end{lemma}

\begin{proof} Since the fat structure of $\Gamma$ coincides with the one induced by the orientation of the surface we can thicken $\Gamma$ inside $\Sg$ to a subsurface $S_{\Gamma}$ of the same topological type as $\Sg$.  Moreover, by the definition of a marking each boundary component of $S_{\Gamma}$ meets a boundary component of $\Sg$.  Thus, there is a deformation retraction of $\Sg$ onto this subsurface and onto $\Gamma$.\end{proof}

\begin{lemma}
\label{markings} Let $\Gamma$ be an admissible fat graph, $F$ be a forest in $\Gamma$,  which does not contain any leaves of $\Gamma$. Then there is a bijection $\Mark(\Gamma)\to \Mark(\Gamma/ F)$ denoted by $[H] \mapsto [H_F]$. 

This identification depends on the map connecting both graphs i.e. given $[H]$ a marking of $\Gamma$, if $\tilde{\Gamma}=\Gamma/ F_1=\Gamma/ F_2$ then  $[H_{F_1}]$ and $[H_{F_2}]$ can be different markings of $\tilde{\Gamma}$.  Figure \ref{Dehn_twist_example} gives an example of this in the case of the cylinder.
\end{lemma}

\begin{proof}
Let $H$ be a representative of a marking $[H]$ of $\Gamma$.  The image of $H|_F$ (the restriction of $H$ to $|F|$) is contained in a disjoint union of disks away from the boundary.  Therefore, the marking $H$ induces a marking $H_F \colon |\Gamma/ F|\cof \Sg$ given by collapsing each of the trees of $F$ to a point of the disk in which their image is contained.  Note that $H_F$ is well defined up to isotopy and it makes the following diagram commute up to homotopy
\begin{comment}\begin{equation*}
\begin{tikzpicture}[scale=0.5]
\node (a) at (0,3){$\vert\Gamma\vert$};
\node (b) at (5,3) {$\vert\Gamma/ F\vert$};
\node (c) at (5,0){$\Sg$};
\path[auto,arrow,->] (a) edge  [left hook->] node[swap]{$$ \mbox{\footnotesize $H$ } $$}  (c)
    		                      (b) edge  [right hook->] node{$$ \mbox{\footnotesize $H_F$ } $$} (c);
\path[auto,arrow,->>] (a) edge (b);
\end{tikzpicture}
\end{equation*}\end{comment}
%\[\xymatrix{\vert\Gamma\vert \ar[rd]_{H} \ar@{->>}[r] & \vert\Gamma/ F\vert \ar@{^{(}->}[d]^{H_F} \\
% & \Sg}\]
\[\begin{tikzcd}\vert\Gamma\vert \arrow{rd}[swap]{H} \rar[two heads] & \vert\Gamma/ F\vert \dar[hookrightarrow]{H_F} \\
	& \Sg. \end{tikzcd}\]
 
In fact, up to isotopy, there is a unique embedding of a tree with a fat structure into a disk, in which the fat structure of the tree coincides with the one induced by the orientation of the disk and the endpoints are fixed points on the boundary. This can proven by induction.  Start with the case where $F$ is a single edge.  Up to homotopy, there is a unique embedding of an arc in a disk where the endpoints of the arc are fixed points on the boundary.  Then by \cite{ feustel}, there is also a unique embedding up to isotopy.  For the induction step, let $\alpha$ be an arc embedded in the disk with its endpoints at the boundary and let $a$ and $b$ be fixed points in the boundary of a connected component of $D\setminus \alpha$. Then we have a map 
\[\mathrm{Emb}^{a,b}(I, D\setminus \alpha)\longrightarrow\mathrm{Emb}^{a,b}(I, D),\]
where $\mathrm{Emb}^{a,b}(I, D \setminus \alpha)$ is the space of embeddings of a path in $D \setminus \alpha$ which start at $a$ and end at $b$, with the $\C^\infty$-topology, and similarly for $\mathrm{Emb}^{a,b}(I, D)$.  By \cite{gramain}, this map induces injective maps in all homotopy groups, in particular in $\pi_0$, which gives the induction step.

It then follows that, given $[H_F]$ a marking of $\Gamma/ F$ there is a unique marking $[H]$ of $\Gamma$  such that the above diagram commutes up to homotopy.  \end{proof}

\begin{figure}[h!]
\centering
\begin{tikzpicture}
	\draw (0,0) circle (.5cm);
	\draw [decoration={markings, mark=at position 0.75 with {\arrow{>}}},	postaction={decorate},blue] (-.5,0) -- (-.2,0);
	\node at (-.1,0) [blue] {\footnotesize $2$};
	\draw [decoration={markings, mark=at position 0.75 with {\arrow{>}}},	postaction={decorate}] (.5,0) -- (.8,0);
	\node at (.9,0)  {\footnotesize $1$};
	\node at (90:.7) {$e_1$};
	\node at (-90:.7) {$e_2$};
	\node at (135:1) [left] {$\Gamma$};
	\node at (.5,0) {\tiny $\bullet$};
	\node at (-.5,0) {\tiny $\bullet$};
	
	\draw [->>] (1.2,.3) -- node[pos=.5,above] {\footnotesize collapse $e_1$} (3.2,.3);
	\draw [->>] (1.2,-.3) -- node[pos=.5,below] {\footnotesize collapse $e_2$} (3.2,-.3);
	
	\begin{scope}[xshift=4cm]
	\draw (0,0) circle (.5cm);
	\node at (.5,0) {\tiny $\bullet$};
	\draw [decoration={markings, mark=at position 0.75 with {\arrow{>}}},	postaction={decorate},blue] (.5,0) -- (.2,0);
	\node at (.1,0) [blue] {\footnotesize $2$};
	\draw [decoration={markings, mark=at position 0.75 with {\arrow{>}}},	postaction={decorate}] (.5,0) -- (.8,0);
	\node at (.9,0)  {\footnotesize $1$};
	\node at (45:1) [right] {$\tilde{\Gamma} = \Gamma/e_1 = \Gamma/e_2$};
	\end{scope}
	
	\begin{scope}[yshift=-3cm,xshift=2cm]
	\draw (-2,.5) -- (2,.5);
	\draw (-2,-.5) -- (2,-.5);
	\draw [xscale=.7] ({-2/.7},0) circle (.5cm);
	\draw [xscale=.7] ({2/.7},0) circle (.5cm);
	\draw [xscale=.7,thick] (0,.5) arc (90:-90:.5);	
	\draw [xscale=.7,densely dotted,thick] (0,.5) arc (90:270:.5);
	\draw [thick] (.35,0) node {\tiny $\bullet$} -- (2.35,0) node  {\tiny $\bullet$} node[right] {\footnotesize $1$};
	\draw [densely dotted,thick,blue] (-.35,0) node {\tiny $\bullet$} to[out=180,in=0] (-1,.5);
	\draw [thick,blue] (-1,.5) to[out=180,in=0] (-1.65,0) node {\tiny $\bullet$} node [left] {\footnotesize $2$};	
	\node at (0,1) {$H$};
	
	\draw [->] (-.5,-1) -- (-2,-2);
	\draw [->] (.5,-1) -- (2,-2);
	
	\begin{scope}[xshift=-3cm,yshift=-3cm]
	\node at (0,1) {$H_{e_1}$};
	\draw (-2,.5) -- (2,.5);
	\draw (-2,-.5) -- (2,-.5);
	\draw [xscale=.7] ({-2/.7},0) circle (.5cm);
	\draw [xscale=.7] ({2/.7},0) circle (.5cm);
	\draw [xscale=.7,thick] (0,.5) arc (90:-90:.5);	
	\draw [xscale=.7,densely dotted,thick] (0,.5) arc (90:270:.5);
	\draw [thick] (.35,0) node {\tiny $\bullet$} -- (2.35,0) node  {\tiny $\bullet$} node[right] {\footnotesize $1$};
	\draw [thick,blue] (.35,0) node {\tiny $\bullet$} -- (-1.65,0) node {\tiny $\bullet$} node [left] {\footnotesize $2$};	
	\end{scope}
	
	\begin{scope}[xshift=3cm,yshift=-3cm]
	\node at (0,1) {$H_{e_2}$};
	\draw (-2,.5) -- (2,.5);
	\draw (-2,-.5) -- (2,-.5);
	\draw [xscale=.7] ({-2/.7},0) circle (.5cm);
	\draw [xscale=.7] ({2/.7},0) circle (.5cm);
	\draw [xscale=.7,thick] (0,.5) arc (90:-90:.5);	
	\draw [xscale=.7,densely dotted,thick] (0,.5) arc (90:270:.5);
	\draw [thick] (.35,0) node {\tiny $\bullet$} -- (2.35,0) node  {\tiny $\bullet$} node[right] {\footnotesize $1$};
	\draw [thick,blue] (.35,0) node {\tiny $\bullet$} to[out=180,in=0] (-.6,-.5);
	\draw [thick,densely dotted, blue] (-.6,-.5) to[out=180,in=0] (-1.3,.5);	
	\draw [thick,blue] (-1.3,.5) to[out=180,in=0] (-1.65,0) node {\tiny $\bullet$} node [left] {\footnotesize $2$};		
	\end{scope}
	\end{scope}
\end{tikzpicture}
\caption{Two different embeddings of $\tilde{\Gamma}$ in the cylinder differing by a Dehn twist and corresponding to the same marking of $\Gamma$.}
\label{Dehn_twist_example}
\end{figure}

%From now on, we will denote by $[H_F]$ the marking of $\Gamma/ F$ corresponding to the marking $[H]$ of $\Gamma$ under this identification.

\begin{definition}
Define the category $\Ead$ to be the category with objects isomorphism classes of marked admissible fat graphs $([\Gamma],[H])$ (where two marked admissible fat graphs are isomorphic if their underlying fat graphs are isomorphic and they have the same marking) and morphisms given by morphisms in $\Fatad$ where the map acts on the marking as stated in the previous lemma.  We denote by $\Eadg$, the full subcategory with objects marked admissible fat graphs whose thickening give a cobordism of topological type $\Sg$. 
\end{definition}

\begin{definition}\label{deffatuniversal}
The space of \emph{marked metric admissible fat graphs} $\EMad$ is defined to be
\[\EMad \coloneqq \frac{\bigsqcup_{\Gamma}\mathpzc{M}(\Gamma)\times \Mark(\Gamma)}{\sim_E}\]
where $\Gamma$ runs over all admissible fat graphs and the equivalence relation is given by 
\[(\Gamma,\lambda,[H])\sim_E(\tilde{\Gamma},\tilde{\lambda},[\tilde{H}])\Longleftrightarrow
(\Gamma,\lambda)\cong(\tilde{\Gamma},\tilde{\lambda})
\text{ and }
[H_{\lambda}]=[\tilde{H}_{\tilde{\lambda}}],\]
where $\cong$ denotes isomorphism of metric fat graphs,  $H_\lambda$ is the induced marking $H_\lambda:\vert \Gamma/ F_\lambda\vert \cof S_{g,n+m}$ where $F_\lambda$ is the subforest of $\Gamma$ of edges of length zero i.e., $F_\lambda=\lambda^{-1}(0)$ and $H_{\tilde{\lambda}}$ is defined analogously.
\end{definition}

The following result is proven in \cite{egas}, in fact in more generality for a category modeling open closed cobordism and not only closed cobordisms.

\begin{theorem}
\label{ad_universal}
The projection $\vert\Eadg\vert\to\vert\Fatadg\vert$ is a universal $\Modgpq$-bundle.
\end{theorem}

The proof follows the original ideas of Igusa \cite{igusahrt} and Godin \cite{godinunstable}. Since all spaces involved are CW-complexes, one firstly shows that $\vert\Eadg\vert$ is contractible, which follows from contractibility of the arc complex \cite{hatcher_arc}. Secondly, one proves that the action of the mapping class group $\Modgpq$ on $\Eadg$ is free and transitive. That is, for any two markings $[H_1]$ and $[H_2]$, there is a unique $[\varphi]\in\Modgpq$ such that $[\varphi\circ H_1]=[H_2]$. This proof in particular gives rise to an abstract homotopy equivalence $\cM \simeq \Fatad$.

By Lemma \ref{markings}, as a set $\EMad$ is given by $\lbrace ([\Gamma,\lambda],[H])\vert [\Gamma,\lambda]\in\MFatad, [H]\in \Mark ([\Gamma])\rbrace$. As before, let $\EMadg$ denote the subspace of marked metric admissible fat graphs whose thickening give an open closed cobordism of topological type $\Sg$.  Then $\Modgpq$ acts on $\EMadg$ by composition with the marking and it follows that:

\begin{corollary}
\label{metricad_universal}
The projection $\EMadg\to\MFatadg$ is a universal $\Modgpq$-bundle.
\end{corollary}
\begin{proof}
This is clear since we have a pullback diagram
\[\begin{tikzcd}\EMadg \arrow{rr}{\simeq}[swap]{r(-,1) \times \mr{id}} \dar[two heads] & & \vert\Eadg\vert \dar[two heads]  \\
	\MFatadg  \arrow{rr}{\simeq}[swap]{r(-,1)}  & & \vert\Fatadg\vert,\end{tikzcd}\]
where the horizontal maps are the homotopy equivalences given by $r$, the map constructed in Lemma \ref{metric_and_nerve}.\end{proof}

%% file: projection.tex
In this section we construct the space $\Radt$ as well as the maps \eqref{cor_pi1homeq} and \eqref{cor_pi2homeq}, and prove these are homotopy equivalences.

\subsection{Lacher's theorem} The idea for proving that certain maps $f \colon X \to Y$ are homotopy equivalences, will be it is a nice enough map between nice enough spaces with contractible fibers. This is made precise by the Theorem on page 510 of \cite{lacher77}.

\begin{definition}	\
	\begin{enumerate}[(i)]
		\item A subspace $X$ of a space $Y$ is a \emph{neighborhood retract} if there exists an open subset $U$ of $Y$ containing $X$ and a retraction $r \colon U \to X$. 
		\item A space $X$ is an \emph{ANR} if, whenever $X$ is a closed subspace of a metric space $Y$, $X$ is a neighborhood retract of $Y$.
\end{enumerate}\end{definition}

\begin{definition}\
	\begin{enumerate}[(i)]
		\item A subset $A$ of a manifold $M$ is \emph{cellular} if it is the intersection $\bigcap_n E_n$ of a nested sequence $E_1 \supset E_2 \supset \ldots$ of $n$-cells $E_i$ in $M$, i.e., subsets homeomorphic to $D^n$.
		\item A space $X$ is \emph{cell-like} if there is an embedding (i.e. continuous map that is an homomorphism onto its image) $\phi \colon X \to M$ of $X$ into a manifold $M$, such that $\phi(X)$ is cellular.
		\item A map $f \colon X \to Y$ is \emph{cell-like} if for all $y \in Y$ the point inverse $f^{-1}(\{y\})$ is cell-like.
\end{enumerate}\end{definition}

\begin{theorem}[Lacher] \label{thm_lacher} A proper map $f \colon X \to Y$ between locally compact ANR's is cell-like if and only if for all opens $U \subset Y$ the restriction $f|_{f^{-1}(U)} \colon f^{-1}(U) \to U$ is a proper homotopy equivalence.\end{theorem}

The conditions in the above definitions are difficult to verify, so we will provide criteria which imply them. Our main reference for ANR's is \cite{vanmill}, for polyhedra is Chapter 3 of \cite{fritschp}, and for cell-like spaces is \cite{lacher77}. 

\begin{proposition}\label{prop_anrprops}The following are properties of ANR's:
	\begin{enumerate}[(i)]
		\item For all $n\geq 0$, the closed $n$-disc $D^n$ is an ANR.
		\item An open subset of an ANR is an ANR.
		\item If $X$ is a space with an open cover by ANR's, then $X$ is an ANR.
		\item If $X$ and $Y$ are compact ANR's, $A \subset X$ is a compact ANR and $f \colon A \to Y$ is continuous, then $X \cup_f Y$ is an ANR.
		\item Any locally finite CW-complex is an ANR.
		\item Any locally finite polyhedron is an ANR.
		\item A product of finitely many ANR's is an ANR.
		\item A compact ANR is cell-like if and only if it is contractible.
	\end{enumerate}
\end{proposition}

\begin{proof}Property (i) follows from Corollary 5.4.6 of \cite{vanmill}, property (ii) is Theorem 5.4.1, property (iii) is Theorem 5.4.5, property (iv) is Theorem 5.6.1. Together these can combined to prove property (v), by noting that by (ii) and (iii) one can reduce to the case of finite CW-complex and since by definition these can be obtained by glueing closed $n$-disks together, (i) and (iv) prove that finite CW-complexes are ANR's. Property (vi) follows from property (v), but is also Theorem 3.6.11 of \cite{vanmill}. Property (vii) is Proposition 1.5.7. Finally, property (viii) follows from Theorem \ref{thm_lacher} by considering the map to a point.
\end{proof}

\subsection{The fattening of the radial slit configurations and the critical graph map} \label{subsec_blowup}

There is a natural admissible metric fat graph associated to a radial slit configuration; the unstable critical graph obtained by taking the inner boundaries of the annuli and the complements of the slit segments and gluing these together according to the combinatorial data. The inner boundaries of the annuli give the admissible cycles of the graph and the incoming leaves are placed at the positive real line of each annulus. The outgoing leaves are obtained from marked points on the outgoing boundary components. This graph gets a canonical fat graph structure as a subspace of the surface $S(L)$.

We now make this definition precise. Because we fixed the outer radii of the annuli, we shorten $\smash{\bA^{(i)}_{R_i}}$ to $\bA_i$. Recall the subsets $\alpha^\pm_i$ and $\beta^\pm_i$ in the sector $F_i$, defined in Definition \ref{def_alphabeta}. These lie in a pair of distinct radial segments of $F_i$, unless it is a thin sector in which case they lie in a single radial segment. To a radial slit configuration $L\in\QRad$ we associate a space $E_L$ defined as follows:

\begin{definition}The space $\overline{E}_L$ is given by
	\[\overline{E}_L \coloneqq \left(\bigsqcup_ {1\leq j\leq n}\partial_\mr{in}\bA_j\right) \sqcup \left(\bigsqcup_{1\leq j\leq 2h+m}E_j\right) \sqcup \left(\bigsqcup_ {1\leq j\leq n}I_j\right)\]
	where each of the terms is defined as follows:
	\begin{description}
		\item[Admissible boundaries] For each annulus $\bA_j$ we take the inner boundary $\partial_\mr{in} \bA_j$.
		\item[Radial segments for slits and outgoing leaves] For $1\leq j\leq 2h+m$ with $\xi_j\in\bA_k$ we take $E_j=\lbrace z\in\bA_k\vert \arg(z)=\arg(\xi_j) \text{ or } \arg(z) = \arg(\xi_{\overline{\omega}(j)})\rbrace$.
		\item[Incoming leaves] For each annulus $\bA_j$ we take $I_j=\lbrace z\in\bC_j\vert \arg(z)=0, 0\leq |z| \leq \frac{1}{2\pi}\rbrace$.
\end{description}\end{definition}
\noindent The equivalence relation $\sim_L$ on $\overline{E}_L$ is that generated by:
\begin{description}
	\item[Attaching incoming leaves] We set $(\frac{1}{2\pi}\in I_j)\sim_L (\frac{1}{2\pi}\in\partial_\mr{in}\bA_j)$ for $j=1,2,\ldots, n$.
	\item[Attaching radial segments] For $r\in\partial_\mr{in}\bA_k$ and $e\in E_j$, we set $r\sim_L e$ if $r=e$.
	\item[Identifying coinciding segments] Defining subsets $\alpha^\pm_i$ and $\beta^\pm_i$ of $E_i$ as in Definition \ref{def_alphabeta}, we let $\sim_L$ identify $z \in \alpha^+_i$ with $z \in \alpha^-_{\overline{\omega}(i)}$ and $z \in \beta^+_i$ with $z \in \beta^-_{\overline{\lambda}(i)}$.	
\end{description}

Note that each of the terms in $\overline{E}_L$ can be considered as subspace of $\overline{\Sigma}(L)$; recalling Definition \ref{def_approxl}, one observes that $\sim_L$ simply identifies those points on $\overline{E}_L$ that are identified by $\approx_L$ on $\overline{\Sigma}(L)$. As a consequence, the quotient space $\overline{E}_L/{\sim_L}$ is invariant under the slit jump relation. Thus for a configuration $[L]\in\Rad$ we obtain a well-defined graph $\Gamma_{[L]}$ if we demand it has no bivalent vertices. Some of its leaves are labeled by the incoming or outgoing boundary components; the remaining ones we will remove.

\begin{definition}
	For $L\in\QRad$ the corresponding \emph{critical graph} $\Gamma_{L}$ is the graph obtained from $\overline{E}_L/{\sim_L}$ by removing those leaves that do not correspond to incoming or outgoing boundary cycles (see Figure \ref{criticalgraph_example}).
\end{definition}

\begin{figure}[h!]
	\centering
	\begin{tikzpicture}
		\node at (0,0) {\footnotesize $1$};
		\draw (90:1.2cm) -- (90:1.5cm);
		\node at (90:1.6cm) {\footnotesize $1$};    	
		\draw (45:1.2cm) -- (45:1.5cm);
		\node at (45:1.6cm) {\footnotesize $2$};
		
		\draw [dashed] (.5,0) -- (1.2,0);
		\draw [decorate,decoration={brace,amplitude=2pt,mirror,raise=0pt},yshift=0pt] (90:.5cm) -- (90:.9cm);
		\node at (90:.7cm) [right] {\tiny $t$};
		\draw [red,thick] (90:.9cm) -- (90:1.2cm);
		\draw [red,thick] (-90:.9cm) -- (-90:1.2cm);
		
		\draw (0,0) circle (1.2cm);
		\draw[decoration={markings, mark=at position 0.5 with {\arrow{<}}},	postaction={decorate}]  (0,0) circle (.5cm);
		
		\draw [dashed,->] (0,-1.5) -- (0,-2.2);
		
		\begin{scope}[yshift=-3.7cm,xshift=.5cm]
			\draw (0,0) circle (.5cm);
			\draw [decoration={markings, mark=at position 0.75 with {\arrow{<}}},	postaction={decorate},blue] (.5,0) -- (.2,0) node [left] {\tiny $1$};
			\draw [decoration={markings, mark=at position 0.75 with {\arrow{>}}},	postaction={decorate}] (45:.5) -- (45:.7);
			\node at (45:.8) {\tiny $2$};
			\draw (0,.5) to[out=90,in=90,looseness=1.5] node[pos=.5,text opacity=1,opacity=.7,fill=white] {$t$} (-1.5,0) node {$\cdot$} to[out=-90,in=-90,looseness=1.5] node[pos=.5,text opacity=1,opacity=.7,fill=white] {$t$}   (0,-.5);
			\draw [decoration={markings, mark=at position 0.75 with {\arrow{>}}},	postaction={decorate}] (-1.5,0) -- (-1.2,0);
			\node at (-1.1,0) {\tiny $1$};
		\end{scope}
		
		\begin{scope}[xshift=4cm]
			\node at (0,0) {\footnotesize $1$};
			\draw (90:1.2cm) -- (90:1.5cm);
			\node at (90:1.6cm) {\footnotesize $1$};    	
			
			\draw [dashed] (.5,0) -- (1.2,0);
			\draw [decorate,decoration={brace,amplitude=2pt,raise=0pt},yshift=0pt] (-90:.5cm) -- (-90:.8cm);
			\node at (-90:.65cm) [right] {\tiny $t$};
			\draw [red,thick] (0:.8cm) -- (0:1.2cm);
			\draw [red,thick] (-90:.8cm) -- (-90:1.2cm);
			\draw [blue,thick] (180:1cm) -- (180:1.2cm);
			\draw [blue,thick,yshift=-.7mm] (0:1cm) -- (0:1.2cm);
			\draw [decorate,decoration={brace,amplitude=2pt,mirror,raise=0pt},yshift=0pt] (180:.5cm) -- (180:1cm);
			\node at (180:.75cm) [above] {\tiny $s$};
			
			\draw (0,0) circle (1.2cm);
			\draw[decoration={markings, mark=at position 0.25 with {\arrow{<}}},	postaction={decorate}]  (0,0) circle (.5cm);
			
			\draw [dashed,->] (0,-1.5) -- (0,-2.2);
			
			\begin{scope}[yshift=-3.45cm,xshift=0cm]
				\draw (0,0) circle (.5cm);
				\draw [decoration={markings, mark=at position 0.75 with {\arrow{<}}},	postaction={decorate},blue] (.5,0) -- (.2,0) node [left] {\tiny $1$};
				\draw [decoration={markings, mark=at position 0.75 with {\arrow{>}}},	postaction={decorate}] (90:.5) -- (90:.7);
				\node at (90:.8) {\tiny $1$};	
				\draw (1,-1) to[out=135,in=0] (-.25,-.8) to[out=180,in=180] node[pos=.3,text opacity=1,opacity=.7,fill=white,left] {$2s-t$} (-.5,0);
				\fill [white] (-.05,-.9) rectangle (.1,-.6);
				\draw (0,-.5) to[out=-90,in=-135,looseness=1.5] node[pos=.5,text opacity=1,opacity=.7,fill=white] {$t$} (1,-1) node {$\cdot$} to[out=55,in=0,looseness=1.5] node[pos=.5,text opacity=1,opacity=.7,fill=white] {$t$} (.5,0);
				
			\end{scope}
		\end{scope}
	\end{tikzpicture}
	\caption{Critical graphs for different configurations. Edge lengths of the critical graphs are \emph{not} to scale.}
	\label{criticalgraph_example}
\end{figure}

By construction, this graph comes embedded in the surface $\Sigma_{[L]}$ and thus inherits a fat structure. Moreover, it inherits a metric $\lambda_{[L]}$ from the standard metric in $\bC$. In it, the incoming leaves have fixed length $\frac{1}{2\pi}$ and the outgoing leaves have strictly positive length. Because for our purposes the lengths of the outgoing leaves are superfluous information, we set $\lambda_{[L]}(e)$ to be given by the standard metric in $\bC$ if $e$ is not a leaf and $\lambda_{[L]}(e)=1$ if $e$ is a leaf. This makes $(\Gamma_{[L]},\lambda_{[L]})$ a metric admissible fat graph. 

\begin{notation}We will just write $\Gamma_{L}$, when it is clear from that context that we consider it as a metric admissible fat graph.\end{notation}

The construction of the critical graph gives a function
\begin{align*} \Rad&\lra \MFatad \\
	[L]&\longmapsto(\Gamma_{[L]},\lambda_{[L]}).\end{align*}
However, this function is \emph{not} continuous at non-generic configurations. For an example, consider the path in $\Rad$ given by continuously varying the argument of a slit as in Figure \ref{badpath}; when the moving slit reaches a neighboring one the associated metric graph jumps.  

\begin{figure}[h!]
	\centering
	\begin{tikzpicture}
		\draw (90:.9cm) -- (90:1.1cm);
		\node at (90:1.2cm) {\footnotesize $1$};    	
		\draw (-135:.9cm) -- (-135:1.1cm);
		\node at (-135:1.2cm) {\footnotesize $2$};
		\draw (-80:.9cm) -- (-80:1.1cm);
		\node at (-80:1.2cm) {\footnotesize $3$};
		
		\draw [dashed] (.2,0) -- (.9,0);
		\draw [red,thick] (0:.5cm) -- (0:.9cm);
		\draw [red,thick] (180:.5cm) -- (180:.9cm);
		\draw [blue,thick] (-110:.7cm) -- (-110:.9cm);
		\draw [blue,thick] (-70:.7cm) -- (-70:.9cm);
		
		\draw (0,0) circle (.9cm);
		\draw[decoration={markings, mark=at position 0.5 with {\arrow{<}}},	postaction={decorate}]  (0,0) circle (.2cm);
		
		\draw [->] (1.05,0) -- (1.45,0);
		\draw [->] (0.75,-3) -- (1.75,-3);
		\draw [->] (3.25,-3) -- (4.75,-3);
		\draw [->] (7.25,-3) -- (8.75,-3);
		
		\draw [->,dashed] (0,-1.4) -- (0,-2.1);
		
		\begin{scope}[yshift=-3cm]
			\draw (0,0) circle (.2cm);
			\draw (0,.2) -- (0,.3);    	
			\draw (0,-.2) -- (0,-.3);	
			\node at (0,-.4) {\tiny $3$};
			\draw [blue] (.05,0) -- (.2,0);
			\node at (0,.4) {\tiny $1$};    	
			\draw (-135:.2) -- (-135:.3);	
			\node at (-135:.4) {\tiny $2$};
			\draw (.2,0) to[out=45,in=0,looseness=1.5] (0,.6) to[out=180,in=135,looseness=1.5] (-.2,0);
			\draw (-70:.2) to[out=-45,in=0,looseness=1.5] (-90:1) to[out=180,in=-135,looseness=1.5] (-110:.2);
		\end{scope}
		
		\begin{scope}[xshift=2.5cm]
			\draw (90:.9cm) -- (90:1.1cm);
			\node at (90:1.2cm) {\footnotesize $1$};    	
			\draw (-135:.9cm) -- (-135:1.1cm);
			\node at (-135:1.2cm) {\footnotesize $2$};
			\draw (-80:.9cm) -- (-80:1.1cm);
			\node at (-80:1.2cm) {\footnotesize $3$};
			
			\draw [dashed] (.2,0) -- (.9,0);
			\draw [red,thick] (0:.5cm) -- (0:.9cm);
			\draw [red,thick] (180:.5cm) -- (180:.9cm);
			\draw [blue,thick] (-110:.7cm) -- (-110:.9cm);
			\draw [blue,thick] (-20:.7cm) -- (-20:.9cm);
			
			\draw (0,0) circle (.9cm);
			\draw[decoration={markings, mark=at position 0.5 with {\arrow{<}}},	postaction={decorate}]  (0,0) circle (.2cm);
			
			\draw [->] (1.05,0) -- (1.45,0);
			
			\draw [->,dashed] (0,-1.4) -- (0,-2.1);
			
			\begin{scope}[yshift=-3cm]
				\draw (0,0) circle (.2cm);
				\draw (0,.2) -- (0,.3);    	
				\draw (0,-.2) -- (0,-.3);	
				\node at (0,-.4) {\tiny $3$};
				\draw [blue] (.05,0) -- (.2,0);
				\node at (0,.4) {\tiny $1$};    	
				\draw (-135:.2) -- (-135:.3);	
				\node at (-135:.4) {\tiny $2$};
				\draw (.2,0) to[out=45,in=0,looseness=1.5] (0,.6) to[out=180,in=135,looseness=1.5] (-.2,0);
				\draw (-20:.2) to[out=-25,in=30,looseness=1.5] (-70:1) to[out=-150,in=-135,looseness=1.5] (-110:.2);
			\end{scope}
		\end{scope}
		
		\begin{scope}[xshift=5cm]
			\draw (90:.9cm) -- (90:1.1cm);
			\node at (90:1.2cm) {\footnotesize $1$};    	
			\draw (-135:.9cm) -- (-135:1.1cm);
			\node at (-135:1.2cm) {\footnotesize $2$};
			\draw (-80:.9cm) -- (-80:1.1cm);
			\node at (-80:1.2cm) {\footnotesize $3$};
			
			\draw [dashed] (.2,0) -- (.9,0);
			\draw [red,thick] (0:.5cm) -- (0:.9cm);
			\draw [red,thick] (180:.5cm) -- (180:.9cm);
			\draw [blue,thick] (-110:.7cm) -- (-110:.9cm);
			\draw [blue,thick,yshift=-.7mm] (0:.7cm) -- (0:.9cm);
			
			\draw (0,0) circle (.9cm);
			\draw[decoration={markings, mark=at position 0.5 with {\arrow{<}}},	postaction={decorate}]  (0,0) circle (.2cm);
			
			\draw [<->] (1.05,0) --node[pos=.5,above] {$=$} (1.45,0);
		\end{scope}
		
		\draw [->,dashed] (5,-1.4) -- (5.5,-2.1);    	
		\draw [->,dashed] (7.5,-1.4) -- (7,-2.1);
		
		\begin{scope}[xshift=6.25cm,yshift=-3cm]
			\draw (0,0) circle (.2cm);
			\draw (0,.2) -- (0,.3);    	
			\draw (0,-.2) -- (0,-.3);	
			\node at (0,-.4) {\tiny $3$};
			\draw [blue] (.05,0) -- (.2,0);
			\node at (0,.4) {\tiny $1$};    	
			\draw (-135:.2) -- (-135:.3);	
			\node at (-135:.4) {\tiny $2$};
			\draw (.2,0) to[out=45,in=0,looseness=1.5] (0,.6) to[out=180,in=135,looseness=1.5] (-.2,0);
			\draw (-110:.2) to[out=-110,in=-110,looseness=1.5] (160:1) to[out=70,in=90,looseness=1.5] (0,.6);
		\end{scope}
		
		\begin{scope}[xshift=7.5cm]
			\draw (90:.9cm) -- (90:1.1cm);
			\node at (90:1.2cm) {\footnotesize $1$};    	
			\draw (-135:.9cm) -- (-135:1.1cm);
			\node at (-135:1.2cm) {\footnotesize $2$};
			\draw (-80:.9cm) -- (-80:1.1cm);
			\node at (-80:1.2cm) {\footnotesize $3$};
			
			\draw [dashed] (.2,0) -- (.9,0);
			\draw [red,thick] (0:.5cm) -- (0:.9cm);
			\draw [red,thick] (180:.5cm) -- (180:.9cm);
			\draw [blue,thick] (-110:.7cm) -- (-110:.9cm);
			\draw [blue,thick,yshift=-.7mm] (180:.7cm) -- (180:.9cm);
			
			\draw (0,0) circle (.9cm);
			\draw[decoration={markings, mark=at position 0.5 with {\arrow{<}}},	postaction={decorate}]  (0,0) circle (.2cm);
			
			\draw [->] (1.05,0) -- (1.45,0);
		\end{scope}
		
		\begin{scope}[xshift=10cm]
			\draw (90:.9cm) -- (90:1.1cm);
			\node at (90:1.2cm) {\footnotesize $1$};    	
			\draw (-135:.9cm) -- (-135:1.1cm);
			\node at (-135:1.2cm) {\footnotesize $2$};
			\draw (-80:.9cm) -- (-80:1.1cm);
			\node at (-80:1.2cm) {\footnotesize $3$};
			
			\draw [dashed] (.2,0) -- (.9,0);
			\draw [red,thick] (0:.5cm) -- (0:.9cm);
			\draw [red,thick] (180:.5cm) -- (180:.9cm);
			\draw [blue,thick] (-110:.7cm) -- (-110:.9cm);
			\draw [blue,thick] (-150:.7cm) -- (-150:.9cm);
			
			\draw (0,0) circle (.9cm);
			\draw[decoration={markings, mark=at position 0.5 with {\arrow{<}}},	postaction={decorate}]  (0,0) circle (.2cm);
			
			\draw [->,dashed] (0,-1.4) -- (0,-2.1);
			
			\begin{scope}[yshift=-3cm]
				\draw (0,0) circle (.2cm);
				\draw (0,.2) -- (0,.3);    	
				\draw (0,-.2) -- (0,-.3);	
				\node at (0,-.4) {\tiny $3$};
				\draw [blue] (.05,0) -- (.2,0);
				\node at (0,.4) {\tiny $1$};    	
				\draw (-135:.2) -- (-135:.3);	
				\node at (-135:.4) {\tiny $2$};
				\draw (.2,0) to[out=45,in=0,looseness=1.5] (0,.6) to[out=180,in=135,looseness=1.5] (-.2,0);
				\draw (-110:.2) to[out=-90,in=-50,looseness=1.5] (-140:1) to[out=140,in=180,looseness=1.5] (-150:.2);
			\end{scope}
		\end{scope}
	\end{tikzpicture}
	\caption{An example of a path in $\Rad$ which leads to a path in $\MFatad$ that is not continuous. Labelings have been left out for the sake of clarity.}
	\label{badpath}
\end{figure}

To solve this problem, we enlarge $\Rad$ at non-generic configurations by a contractible space, by ``opening up'' the edges $E_L$. To do this, we first need to introduce some notation. We can think of the thin sector
\[F_i = \{z \in \bA_j \,|\,\arg(\xi_i) = \arg(z)\}\]
as being obtained by identifying two copies of $F_i$, which we will denote $E_i^+,E_i^-$, along the equivalence relation that identifies $z \in E_i^+$ with $z \in E_i^-$. Let us extend this notation to ordinary and full sectors: if $F_i$ is ordinary then
\[E_i^+  \coloneqq \{z \in F_i \mid \arg(z) = \arg(\xi_{\overline{\omega}(i)})\}, \quad \text{and} \quad 
E_i^-  \coloneqq \{z \in F_i \mid \arg(z) = \arg(\xi_{\omega(i)})\},\]
and if $F_i$ is full then $E_i^+ = S^+_i$ and $E_i^- = S_i^-$. Let us also generalize Definition \ref{def_alphabeta} to this section by taking $\alpha^+_i,\beta^+_i \subset E_i^+$ and $\alpha^-_i,\beta^-_i \subset E_i^+$. Then we can also write $\overline{E}_L/\sim_L$ as $\overline{E}'_L/\sim'_L$ with
\[\overline{E}'_L \coloneqq \left(\bigsqcup_ {1\leq j\leq n}\partial_\mr{in}\bA_j\right) \sqcup \left(\bigsqcup_{1\leq j\leq 2h+m}E_j^+ \sqcup E_j^-\right) \sqcup \left(\bigsqcup_ {1\leq j\leq n}I_j\right)\]
and $\sim'_L$ the equivalence relation on $\overline{E}'_L$ generated by replacing $E_j$ by $E_j^\pm$ in the three operations generating $\sim_L$ and adding a fourth one:
\begin{description}
	\item[Identifying thin sectors] If $F_i$ is thin, we let $\sim'_L$ identify $z \in E^+_i$ with $z \in E^-_i$.
\end{description}

The idea is now to vary the extent to which we identify $E^+_i$ with $E^-_i$ in the last of these:

\begin{definition}\label{defdil2} Let $\mr{thin}(L)$ be the set of thin sectors of $L$ and let $t \colon \mr{thin}(L) \to [0,1]$ be a function. The equivalence relation $\sim'_{t}$ on the space \[\overline{E}'_L=\left(\bigsqcup_{1\leq j\leq n}\partial_\mr{in}\bA_j\right) \sqcup \left(\bigsqcup_{1\leq j\leq 2h+m}E^+_j \sqcup E^-_j\right) \sqcup \left(\bigsqcup_ {1\leq j\leq n}I_j\right)\]
	is the one generated by:
	\begin{description}
		\item[Attaching incoming leaves] We set $(\frac{1}{2\pi}\in I_j)\sim'_t (\frac{1}{2\pi}\in\partial_\mr{in}\bA_j)$ for $j=1,2,\ldots, n$.
		\item[Attaching radial segments] For $r\in\partial_\mr{in}\bA_k$ and $e\in E^\pm_j$, we set $r\sim'_t e$ if $r=e$.
		\item[Identifying coinciding segments] With $\alpha^\pm_i$ and $\beta^\pm_i$ of the $E^\pm_j$ as above, we let $\sim'_t$ identify $z \in \alpha^+_i$ with $\alpha^-_{\overline{\omega}(i)}$ and $z \in \beta^+_i$ with $z \in \beta^-_{\overline{\lambda}(i)}$.	
		\item[Partially identifying thin sectors] If $F_i$ is thin, we let $\sim'_t$ identify $z \in E^+_i$ with $z \in E^-_i$ as along as $|z| \leq t(F_i)+\tfrac{1}{2\pi}$.
	\end{description}
\end{definition}

\begin{definition}
	We define $\Gamma_{L,t}$ to be obtained from $\overline{E}'_L/{\sim'_{t}}$ by removing those leaves that do not correspond to incoming or outgoing boundary cycles.\end{definition}

\begin{example}When $t$ is a constant function equal to $1$, $\Gamma_{L,t}$ is the critical graph $\Gamma_{L}$, which is invariant under slit and parametrization points jumps.  However, for most other $t$, the graph $\Gamma_{L,t}$ is \emph{not} invariant under slit jumps.\end{example}

\begin{notation}If $t$ is constant equal to $0$, we will call this the \emph{unfolded graph of $L$} and denote it $\Gamma_{L,0}$ (see Figure \ref{partiallyunfolded}).
\end{notation}

Just like the critical graph, the graph $\Gamma_{L,t}$ has a natural metric making $(\Gamma_{L,t},\lambda_{L,t})$ an admissible metric fat graph. Figure \ref{partiallyunfolded} shows examples of unfolded and partially unfolded metric admissible fat graphs. 

\begin{figure}[h!]
	\centering
	\begin{tikzpicture}
		\clip (-5,-9.7) rectangle (5.5,1.5);
		\node at (0,0) {\footnotesize $1$};
		\draw (45:1.2cm) -- (45:1.5cm);
		\node at (45:1.6cm) {\footnotesize $1$};    
		\draw (-45:1.2cm) -- (-45:1.5cm);
		\node at (-45:1.6cm) {\footnotesize $2$};  
		\draw (-135:1.2cm) -- (-135:1.5cm);
		\node at (-135:1.6cm) {\footnotesize $3$};   	
		\draw (135:1.2cm) -- (135:1.5cm);
		\node at (135:1.6cm) {\footnotesize $4$};
		
		\draw [dashed] (.5,0) -- (1.2,0);
		\draw [decorate,decoration={brace,amplitude=2pt,mirror,raise=0pt},yshift=0pt] (90:.5cm) -- (90:.7cm);
		\node at (90:.6cm) [right] {\tiny $l$};
		\draw [orange,thick,yshift=.6mm] (0:.7cm) -- (0:1.2cm);
		\draw [orange,thick] (90:.7cm) -- (90:1.2cm);
		\draw [blue,thick] (180:.9cm) -- (180:1.2cm);
		\draw [blue,thick] (0:.9cm) -- (0:1.2cm);
		\draw [PineGreen,thick] (-90:1.1cm) -- (-90:1.2cm);
		\draw [PineGreen,thick,yshift=-.6mm] (0:1.1cm) -- (0:1.2cm);
		\draw [decorate,decoration={brace,amplitude=2pt,mirror,raise=0pt},yshift=0pt] (180:.5cm) -- (180:.9cm);
		\node at (180:.7cm) [above] {\tiny $s$};    	
		\draw [decorate,decoration={brace,amplitude=2pt,mirror,raise=0pt},yshift=0pt] (-90:.5cm) -- (-90:1.1cm);
		\node at (-90:.8cm) [left] {\tiny $r$};
		
		\draw (0,0) circle (1.2cm);
		\draw  (0,0) circle (.5cm);
		
		\begin{scope}[xshift=-3cm,yshift=-3.3cm,scale=.8]
			\draw  (0,0) circle (1cm);
			\draw (.7,0) -- (1,0);
			\draw (90:1) edge[bend left=90,looseness=2] node[pos=.5,text opacity=1,opacity=.7,fill=white] {\footnotesize $2l$} (0:1);
			\draw (180:1) edge[bend left=140,looseness=5] node[pos=.5,above] {\footnotesize $2s$} (0:1);
			\draw (-90:1) edge[bend right=90,looseness=2] node[pos=.5,text opacity=1,opacity=.7,fill=white] {\footnotesize $2r$}(0:1);   	
			\node at (0,-2) {$t=(0,0)$};
		\end{scope}
		
		\begin{scope}[xshift=3cm,yshift=-3.3cm,scale=.8]
			\draw  (0,0) circle (1cm);
			\draw (.7,0) -- (1,0);
			\draw (90:1) edge[bend left=90,looseness=2] node (firstmid) [pos=.5] {} node [pos=.25,text opacity=1,opacity=.7,fill=white] {\footnotesize $l$} node [pos=.75,text opacity=1,opacity=.7,fill=white] {\footnotesize $l$} (0:1);
			\draw (180:1) edge[bend left=140,looseness=3] node[pos=.5,text opacity=1,opacity=.7,fill=white,above] {\footnotesize $s$} node (secondmid) [pos=.75] {} node [pos=.9,text opacity=1,opacity=.7,fill=white] {\footnotesize $s-l$} (firstmid.center);
			\draw (-90:1) edge[bend right=110,looseness=1.5] node[pos=.3,text opacity=1,opacity=.7,fill=white,right] {\footnotesize $2r-s$} (secondmid.center);   	
			\node at (0,-2) {\parbox{3cm}{\centering $t=(a,b)$ \\ $a \geq l$, $b \geq s$}};
		\end{scope}
		
		\begin{scope}[xshift=-3cm,yshift=-7.3cm,scale=.8]
			\draw  (0,0) circle (1cm);
			\draw (.7,0) -- (1,0);
			\draw (90:1) edge[bend left=90,looseness=2] node (firstmid) [pos=.6] {} node (secondmid) [pos=.8] {} node [pos=.25,text opacity=1,opacity=.7,fill=white] {\footnotesize $2l-a$} node [pos=.9,text opacity=1,opacity=.7,fill=white] {\footnotesize $b$} node [pos=.7,text opacity=1,opacity=.7,fill=white] {\footnotesize $a-b$} (0:1);
			\draw (180:1) edge[bend left=140,looseness=3] node[pos=.5,text opacity=1,opacity=.7,fill=white,above] {\footnotesize $2s-a$}   (firstmid.center);
			\draw (-90:1) edge[bend right=110,looseness=1.5] node[pos=.3,text opacity=1,opacity=.7,fill=white,right] {\footnotesize $2r-b$} (secondmid.center);   	
			\node at (0,-2.2) {\parbox{3cm}{\centering $t=(a,b)$ \\ $a \geq b$ \\  $a \leq l$, $b \leq s$}};
		\end{scope}
		
		\begin{scope}[xshift=3cm,yshift=-7.3cm,scale=.8]
			\draw  (0,0) circle (1cm);
			\draw (.7,0) -- (1,0);
			\draw (90:1) edge[bend left=90,looseness=2] node (firstmid) [pos=.7] {} node [pos=.25,text opacity=1,opacity=.7,fill=white] {\footnotesize $2l-a$} node [pos=.85,text opacity=1,opacity=.7,fill=white] {\footnotesize $a$} (0:1);
			\draw (180:1) edge[bend left=140,looseness=3] node[pos=.5,text opacity=1,opacity=.7,fill=white,above] {\footnotesize $2s-b$} node (secondmid) [pos=.75] {} node [pos=.9,text opacity=1,opacity=.7,fill=white] {\footnotesize $b-a$} (firstmid.center);
			\draw (-90:1) edge[bend right=110,looseness=1.5] node[pos=.3,text opacity=1,opacity=.7,fill=white,right] {\footnotesize $2r-b$} (secondmid.center);   	
			\node at (0,-2.2) {\parbox{3cm}{\centering $t=(a,b)$ \\ $b \geq a$ \\ $a \leq l$, $b \leq s$}};
		\end{scope}
	\end{tikzpicture}
	\caption{A configuration $[L]$ on the top, and several graphs obtained from it using different functions $t \colon \mr{thin}([L]) \to [0,1]$, here written as a pair of real numbers. The leaves have been omitted to make the graphs more readable, but they are all located along the admissible cycles according to the positions of the marked points in $[L]$. The edges are not to scale.}
	\label{partiallyunfolded}
\end{figure}

\begin{remark}Two preconfigurations with the same combinatorial type have the same underlying admissible fat graphs but with different metric. Thus it makes sense to talk about $\Gamma_{\cL,t}$ which is an admissible fat graph. Similarly, it makes sense to talk about the critical graph of a combinatorial type, which we denote $\Gamma_{[\cL]}$.\end{remark}

\begin{definition}
	\label{graph_blowup}
	Let $[L]\in\Rad$, we define a subspace of $\MFatad$
	\[\cG([L]) \coloneqq \lbrace [\Gamma_{L_i,t},\lambda_{L_i,t}]\,\vert\, [L]=[L_i], t \colon \mr{thin}(L_i) \to [0,1]\rbrace.\]
	We define the \emph{fattening of $\Rad$} to be the space
	\[\Radt=\lbrace ([L],[\Gamma,\lambda])\in \Rad\times\MFatad\,\vert\, [\Gamma,\lambda]\in\cG([L])\rbrace.\]
	For simplicity, we will just write $\Gamma_{L_i,t}$ or $\Gamma$ when it is clear from the context that we are talking about metric graphs.
\end{definition}

We will see that $\Radt$ is constructed by replacing the point $[L] \in \Rad$ by a contractible space $\cG([L])$, which is a space of graphs which interpolate between the critical graph of $[L]$ and the unfolded graphs of the different representatives  $L_1,L_2, \ldots, L_k$ of $[L]$ in $\QRad$. 

The fattening of $\Rad$ splits into connected components given by the topological type of the cobordism they describe:
\[\Radt \coloneqq \bigsqcup_{h,n,m}\Radt_h(n,m).\]
Moreover, it comes with two natural maps 
\[\begin{tikzcd}\Rad & \lar[two heads,swap]{\pi_1} \Radt \rar{\pi_2} & \MFatad.\end{tikzcd}\]
We call $\pi_1$ the \emph{projection map} and  $\pi_2$ the \emph{critical graph map}. The goal of the remaining subsections is to prove that these are homotopy equivalences. The next section is the main input for proving $\pi_1$ is a homotopy equivalence.

\subsection{The space $\cG([L])$ is contractible}
\begin{proposition}
	\label{GL_contractible}
	$\cG([L])$ is contractible for any radial slit configuration $[L]$. 
\end{proposition}

We prove this inductively, by removing parametrization points or slits. In particular, we allow radial slit configuration \emph{without parametrization points}; all relevant definitions may be extended to this case in a straightforward manner.

\begin{notation}
	For a radial slit configuration $L^1$, we denote by $L$ the radial slit configuration obtained from $L^1$ by removing all parametrization points.
	
	If $L$ is not empty, then it has $m\geq 1$ shortest pairs of slits of $L$. That is, $L$ has pairs of slits $(\zeta_{i_j}$, $\zeta_{\lambda(i_j)})$ for $1\leq j \leq m$, which are all of the same length and are the shortest in the following sense:
	\begin{itemize}
		\item $\vert\zeta_{i_j}\vert=\vert \zeta_{\lambda(i_j)}\vert=\vert\zeta_{i_l}\vert=\vert \zeta_{\lambda(i_l)}\vert$, for all  $1\leq j,l\leq m $ and,
		\item $\vert\zeta_{i_j}\vert > \vert \zeta_s\vert$, for any $s\notin \{i_j,\lambda(i_j)| 1\leq j \leq m\}$.
	\end{itemize}
	We denote by $\overline{L}$ the configuration obtained from $L$ by forgetting the shortest slit pair(s).
\end{notation}

Note that if $L^1$ is not degenerate, then $L$ and $\overline{L}$ are also not degenerate. The induction step in the proof of Proposition \ref{GL_contractible} is provided by:

\begin{lemma}
	\label{forget_slits_homeq}
	There are homotopy equivalences
	\[\pi^1_L \colon \cG([L^1])\lra \cG([{L}]) \qquad \text{and} \qquad \pi_L \colon \cG([L])\lra\cG([\overline{L}]).\]   
\end{lemma}

%\begin{remark}As our induction ends at it, it is helpful $\cG[\overline{L}]$ is a point if $[\overline{L}]$ is the empty configuration.\end{remark}

Informally, the map $\pi^1_L$ removes the leaves of $\Gamma^1 \in \cG([L^1])$ corresponding to the outgoing boundary components. Similarly, the map $\pi_L$ removes the edges of $\Gamma \in \cG([L])$ corresponding to the shortest pair(s) of slits in $[L]$. Assuming Lemma \ref{forget_slits_homeq}, we now prove Proposition \ref{GL_contractible}.

\begin{proof}[Proof of Proposition \ref{GL_contractible}]
	By the first part of Lemma \ref{forget_slits_homeq}, it is enough to show that $\cG([L])$ is contractible, where $[L]$ a radial slit configuration without parametrization points. We will prove this by induction on $h$, the number of pairs of slits of $[L]$. When $h=0$, then $\cG([L])$ is a point and therefore contractible. Assume that $\cG([L])$ is contractible when $h<k$ for some fixed $k$.  Now, let $h=k$ and consider the map
	\[\pi_L \colon \cG([L])\lra \cG([\overline{L}]).\]  
	Given that $[\overline{L}]$ has $\overline{h} < k$ pairs of slits, it is contractible by the induction hypothesis. Thus by the second part of Lemma \ref{forget_slits_homeq}, $\cG([L])$ is also contractible.
\end{proof}

\subsubsection{Proof of Lemma \ref{forget_slits_homeq}}

To prove Lemma \ref{forget_slits_homeq} we will show that the spaces involved are compact ANRs and the maps involved are cell-like, and invoke Theorem \ref{thm_lacher}. We start by considering the domain and target of the maps.

\begin{lemma}\label{lemcglpolyhedron} For all configurations $[L]$, with or without parametrization points, the space $\cG([L])$ is a compact polyhedron and thus a compact ANR.
\end{lemma}

\begin{proof}We give the proof only when $[L]$ has parametrization points; the other case is similar.
	
	The space $\cG([L])$ is a subspace of $\MFatadg$. The latter is contained in the larger compact polyhedron given by
	\[P_{g,n+m} \coloneqq \frac{\bigsqcup_{\Gamma} \Delta^{n_1-1}\times\Delta^{n_2-1}\times\cdots \times \Delta^{n_p-1}\times([0,1])^{\# E_{\Gamma}-n}}{\sim}\]
	with $\Gamma$ indexed by the objects of $\Fatadg$ and the equivalence relation $\sim$ given by Definition \ref{def_fatfatad}. This is compact because $\smash{\Fatadg}$ has finitely many objects.
	
	The subspace $\cG([L])$ can be characterized as the union of the images of maps from the cubes $[0,1]^{\mr{thin}(L_i)}$ to $\MFatadg$ for all representatives $L_i$ of $[L]$. Each of these map is a piecewise linear map between polyhedra, which implies that their image is a subpolyhedron. This is true because a piecewise linear map by definition can be made simplicial with respect to some triangulation and the images of simplicial maps are clearly polyhedra. Note that there are only finitely many representatives for $[L]$, so that $\cG([L])$ is a union of finitely many compact polyhedra, which implies it is a polyhedron by Corollary 3.1.27 of \cite{fritschp}. The last claim then follows from property (vi) of Proposition \ref{prop_anrprops}.
\end{proof}

We now define the maps $\pi^1_L$ and $\pi_L$. We start with the former, which ``removes leaves corresponding to the parametrization points.''

\begin{definition}
	Let $[L^1]$ be a radial slit configuration and let $[L]$ be the configuration obtained from $[L^1]$ by removing the parametrization points.  We define the function
	\[\pi^1_L \colon \cG([L^1]) \lra \cG([L])\]
	by sending $\Gamma$ to the metric fat graph obtained from $\Gamma$ by the following procedure: 
	\begin{enumerate}
		\item Removing all leaves corresponding to outgoing boundary components.
		\item Removing all bivalent vertices, i.e., if there is a bivalent vertex we replace the two edges attached to it by a single edge whose length is the sum of the lengths of both.
	\end{enumerate}
\end{definition}

Let $[L]$ be a radial slit configuration without parametrization points and assume it is non-empty, i.e.~$[L]$ has at least one pair of slits. We now define the function $\pi_L$, which ``removes the edges corresponding to the longest slit pair(s) of $[L]$.''

\begin{definition}
	\label{def:piL}
	For any $\Gamma\in\cG([L])$ the continuous function $d_\mathrm{ad} \colon \Gamma \to \R_{\geq 0}$ is defined by sending a point $x$ in a leaf of $\Gamma$ to $0$ and any other point $x\in\Gamma$ to its path distance to the admissible cycles. By the extreme value theorem it attains a maximum $d_\mathrm{max}$. We denote by $\Gamma'$ the fat graph with unlabelled leaves obtained by removing from $\Gamma$ the preimage of $d_\mathrm{max}$. That is, we set $\Gamma' \coloneqq \Gamma-d_\mathrm{ad}^{-1}(d_\mathrm{max})\subset \Gamma$. We define the function
	\[\pi_L \colon \cG([L])\lra \cG([\overline{L}])\]
	by sending $\Gamma$ to the metric fat graph $\overline{\Gamma}$ obtained from $\Gamma'$ by the following recursive procedure:
	\begin{enumerate}
		\item Remove all unlabelled leaves of $\Gamma'$.
		\item Remove all bivalent vertices from to obtain a fat graph $\Gamma''$. 
		\item If $\Gamma''$ has unlabelled leaves repeat the procedure.
	\end{enumerate}
	Note that the only leaves of $\pi_L(\Gamma)$ are the ones corresponding to the admissible cycles.
\end{definition}

We will focus on $\pi_L$ first, leaving $\pi^1_L$ to the end of this subsection. We start with some properties of $\pi_L$:

\begin{lemma}\label{lem:pil-props} \
	\begin{enumerate}[(i)]
		\item $\pi_L$ is well-defined.
		\item $\pi_L$ is continuous.
		\item The fibers of $\pi_L$ are compact ANR's.
\end{enumerate}\end{lemma}

\begin{proof}Let $\Gamma\in\cG([L])$, so that there is a representative $L$  and function $t \colon \mathrm{thin}(L)\to [0,1]$ such that $\Gamma=\Gamma_{L,t}$. Let $\overline{L}$ be the configuration obtained from $L$ by removing the shortest pair(s) of slits. To prove that $\pi_L(\Gamma)$ is well-defined, we exhibit a function $\overline{t} \colon \overline{\mr{thin}}(\overline{L})\to [0,1]$ such that $\pi_L(\Gamma)=\pi_L(\Gamma_{L,t})=\Gamma_{\overline{L},\overline{t}}$. Note any thin sector $F$ of $\overline{L}$ is of one of two kinds:
	\begin{enumerate}
		\item The sector $F$ corresponds uniquely to a sector in $L$. In this case we define $\tilde{t}(F) \coloneqq t(F)$. 
		\item The sector $F$ corresponds to several thin sectors $F_1, F_2, \ldots F_s$ in $L$.  This happens when in between the slits defining the sector $F$ in $\overline{L}$ there are one or more slits in $L$ which have been removed.  In this case, we define 
		\[\overline{t}(F)\coloneqq \min\{t(F_1), t(F_2), \ldots, t(F_s)\}.\]
	\end{enumerate}
	Then we have that $\pi_L(\Gamma)=\Gamma_{\overline{L},\overline{t}}$. This completes the proof of (i).
	
	\vspace{.25em}
	
	For (ii), it suffices to prove $\pi_L$ is continuous on each of the finitely many closed subsets of the form $\{[\Gamma_{L_i,t},\lambda_{L_i,t}] \mid t \colon \mr{thin}(L_i) \to [0,1]\}$, that is, fixing the representative $L_i$ of $[L]$. This is clear from the construction of $\overline{t}$ and hence of $\Gamma_{\overline{L},\overline{t}}$.
	
	\vspace{.25em}
	
	As in the proof of Lemma \ref{lemcglpolyhedron}, for (iii) it suffices to prove the fibers are compact polyhedra, by proving each fiber is the union of the images of finitely many piecewise linear maps with compact domain. But this follows once more from the construction of $\overline{t}$ and hence of $\Gamma_{\overline{L},\overline{t}}$.	
\end{proof}

We now state the main ingredient for the proof of Lemma \ref{forget_slits_homeq}. 

\begin{lemma}\label{lem:preimage_contractible}
	For $\overline{\Gamma}\in\cG([\overline{L}])$ the preimage $\pi_L^{-1}(\overline{\Gamma})\subseteq \cG([L])$ is contractible. 
\end{lemma}

By construction, any $\Gamma \in \pi_L^{-1}(\overline{\Gamma})$ can be built from $\overline{\Gamma}$ by attaching to it a graph. We will show that the space of graphs that can be attached to $\Gamma$ is contractible, and that there is a contractible space of ways to attach each of these. Before doing so, we give two illustrative examples. 

\begin{example}[Single pair of shortest slits]
	Consider the configurations $L$ and $\overline{L}$ obtained by deleting the shortest pair of slits shown in on Figure \ref{fig:example_single_slit} (A) .
	The other representatives $L'$ of $[L]$ are given by letting the purple or green slit on the right jump; 
	for any such representative deleting its shortest pairs of slits also yields a representative of $\overline{L}$. 
	\begin{figure}[!h]
		\centering
		\includegraphics[width=5in]{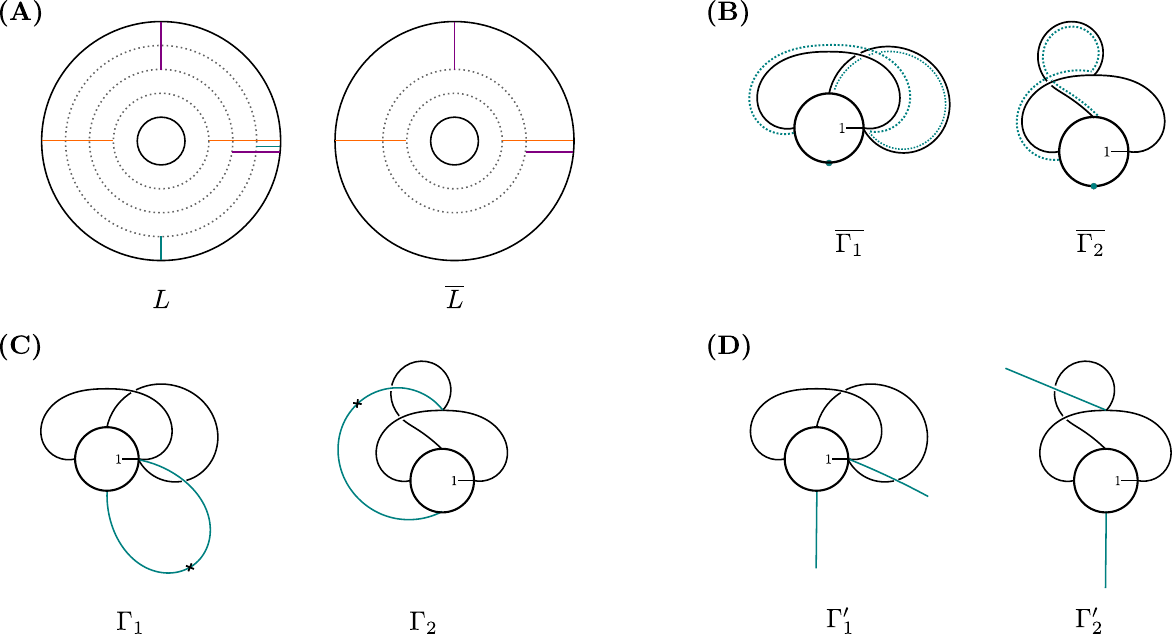}
		% \vspace*{-0.5cm}
		\caption{(A) A configuration $L$ and the configuration $\overline{L}$ obtained from it by deleting the shortest pair of slits (that is, those where $d_{\max}$ is attained). (B) Graphs in $\cG([\overline{L}])$; $\overline{\Gamma_1}$ is the unfolded graph of $\overline{L}$ and $\overline{\Gamma_2}$ is the critical graph of $\overline{L}$.  The green dotted lines trace the boundary interval defined by the open chord corresponding to the deleted green slit and thus describe the places where one endpoint of the new chord can be attached. (C) Graphs in $\cG([L])$ such that $\pi_L(\Gamma_i)=\overline{\Gamma_i}$. In both cases $\Gamma_i$ is the maximally unfolded graph of $L$ relative to $\overline{\Gamma_i}$. The points marked with an $\times$ denote the points in $\Gamma$ as which the maximum of $d_{\mr{ad}}$ is attained. (D) The open graphs of the maximally unfolded graphs relative to $\overline{\Gamma}$ given in part (C).}
		\label{fig:example_single_slit}
	\end{figure}
	
	Panel (C) in Figure \ref{fig:example_single_slit} shows two different graphs in $\cG([L])$: $\Gamma_1$ the unfolded graph of $L$ and $\Gamma_2$ a partially folded graphs of $L$. The map $\pi_L \colon \cG([L]) \to \cG([\overline{L}])$ is given by removing the point marked by an $\times$ in the green arc---which in the case of $\Gamma_1$ is the midpoint of the green arc---and deleting the resulting leaves. In particular, we have that $\pi_L(\Gamma_i)=\overline{\Gamma_i}$ for $i=1,2$, where the graphs $\overline{\Gamma_i}$ are shown in panel (B).  Note that $\overline{\Gamma_1}$ is the unfolded graph of $\overline{L}$ and $\overline{\Gamma_2}$ is the critical graph of $\overline{L}$.
	Therefore, we know that in either case $\pi_L^{-1}(\overline{\Gamma}_i)$ is not empty. 
	
	The entire preimage $\pi_L^{-1}(\overline{\Gamma}_i)$ is given by the locations for attaching a chord to $\overline{\Gamma_i}$. This may be done along the dashed green segments, for one end of the chord and the fixed point marked in green for the other, as marked in panel (B). Thus, the preimages are homeomorphic to intervals. In the either case, the endpoints of the interval correspond to the unfolded graphs of $L$ and the radial slit configuration obtained from $L$ by letting the shortest segment jump. In fact, the preimages $\pi_L^{-1}(\overline{\Gamma})$ are homeomorphic to an interval for all $\overline{\Gamma} \in \cG([\overline{L}])$.
	
	The reason why the second end of the chord could only be attached to a single point is because its corresponding slit is isolated, i.e., it is the only slit on its radial segment. If this were not the case, then the other end of this chord could also be attached to an interval.  The intervals at which both end points of the chord can be attached must be disjoint, otherwise  there would be a sequence of jumps that would bring both slits together and thus $L$ would be degenerate. So in this more generic case  $\pi_L^{-1}(\overline{\Gamma})$ is homeomorphic to a square. 
	Finally, there is another simple generalization of this case: if there are several pairs of shortest slits in $L$, but the intervals describing the endpoints where their corresponding chords can be attached are all disjoint.  In this case, the preimage is homeomorphic to a higher-dimensional cube.
\end{example}

In the previous example we considered the case where there is exactly one pair of slits which is the shortest pair, as well as some simple generalizations of this. On the other end of the spectrum there is the case where all slits are of equal size.

\begin{example}[All slits of equal size]\label{exam:cone} In the following radial slit configuration $L$, the configuration $\overline{L}$ obtained by deleting all shortest slit pairs is empty.
	\[\begin{tikzpicture}
		\clip (-2,-1.5) rectangle (2,1.5);
		
		\node at (-1.6,0) {$L=$};
		\node at (0,0) {\footnotesize $1$};
		\draw [dashed] (.5,0) -- (1.2,0);
		\draw [orange,thick,yshift=.6mm] (0:.8cm) -- (0:1.2cm);
		\draw [orange,thick] (90:.8cm) -- (90:1.2cm);
		\draw [blue,thick] (180:.8cm) -- (180:1.2cm);
		\draw [blue,thick] (0:.8cm) -- (0:1.2cm);
		\draw (0,0) circle (1.2cm);
		\draw (0,0) circle (.5cm);
		
	\end{tikzpicture}\]
	The configuration $[L]$ has three representatives, and $\cG([L])$ (which is the preimage over the unique point in $\cG[\overline{L}]$) is homeomorphic to the cone on three points. These three points are represented by the unfolded graphs of the three representatives and the cone point by the critical graph.
\end{example}

The general case is an amalgamation of these two cases.  More precisely, in the first case---where there is exactly one pair of slits which is the shortest---the preimage is homemorphic to an interval or to a cube arising from the choices of where to attach the endpoints of the attached chord.  In the second case---where all slits are of the same length---we have that the preimage is a cone on three points corresponding to the unfolded representatives.  In general, the preimage is homeomorphic to a product of ``cones on cubes.''  We will show this by going through an intermediary subspace of metric fat graphs corresponding to attaching trees on chords. 

\begin{definition}\label{def:maxim_unfolded}
	Let $\overline{\Gamma}\in \cG([\overline{L}])$. By definition, there is a representative $\overline{L}$ and a function $\overline{t}$ such that $\overline{\Gamma}=\Gamma_{\overline{L},\overline{t}}$. 
	
	Let $L_1, L_2, \ldots, L_r$ be all the radial slit configurations that can obtained from $\overline{L}$ by adding slits such that each $L_i$ is equivalent to $L$ by slit jumps. 
	For any $i$, there is at least one function $t \colon \mr{thin}(L_i) \to [0,1]$ such that $\pi_L(\Gamma_{L_i,t_i})=\overline{\Gamma}$. Let $t_i$ be the minimal one among such functions, i.e.,~the one that takes the smallest possible values for every element of $\mr{thin}(L_i)$. 
	\begin{itemize}
		\item The \emph{maximally unfolded graph of $L_i$ relative to $\overline{\Gamma}$}, is the fat graph $\Gamma_i \coloneqq \Gamma_{L_i,t_i}$. 
		\item The \emph{open graph of $L_i$ relative to $\overline{\Gamma}$}  is the fat graph with unlabeled leaves given by $\Gamma'_i \coloneqq \Gamma_i - d^{-1}_\mathrm{ad}(d_{\max})$ where $d_\mathrm{max}$ is the maximum of the distance from any point in $\Gamma_i$ to the admissible cycles.
	\end{itemize}
\end{definition}

Examples of maximally unfolded graphs relative to some graph can be seen in Figure \ref{fig:example_single_slit} panel (C). Their corresponding open graphs are given in panel (D).

\begin{remark}
	\label{rmk:attaching_forest}
	Any maximally unfolded graph relative to $\overline{\Gamma}$, say $\Gamma_i$, is obtained from $\overline{\Gamma}$ by attaching a chord for each pair of slits deleted in $L_i$. In particular, if $\overline{\Gamma}$ is an unfolded graph then each $\Gamma_i$ is an unfolded graph as well. 
	Furthermore, the preimage $d^{-1}_\mathrm{ad}(d_{\max})$ consists of exactly one point in each of these chords: that point at which the half edges corresponding to each slit pair are glued to each other. Therefore, each leaf in the open graph of $L_i$ relative to $\overline{\Gamma}$ corresponds precisely to a slit deleted from $L_i$.
	
	Moreover, for any graph $\Gamma\in \cG([L])$ there is at least one $L_i$ and a function $t \colon \mr{thin}(L_i) \to [0,1]$ such that $\Gamma=\Gamma_{L_i, t}$ and $t\geq t_i$.  Thus, any graph in $\cG([L])$ can be thought of as a ``folding'' of a maximally unfolded graph relative to $\overline{\Gamma}$, say $\Gamma_i$, where we only ``fold'' the chords that have been attached to $\overline{\Gamma}$ in the construction of $\Gamma_i$. In particular, this shows that any such $\Gamma$ can be obtained from $\overline{\Gamma}$ by attaching to it a forest along its leaves.
	
	A special example of this is the case of the critical graph $\Gamma_{\mr{crit}}\in \cG([L])$. It can be constructed from $\overline{\Gamma}$ by attaching corollas to $\overline{\Gamma}$.  This graph can be obtained by ``completely folding" any of the maximally unfolded graphs relative to $\overline{\Gamma}$. Furthermore, the preimage $d^{-1}_\mathrm{ad}(d_{\max})$ consist exactly of the central vertices of the corollas attached.
\end{remark}

Informally, one can think of $\pi^{-1}_L(\overline{\Gamma})$ as a space of graphs that interpolates the maximally unfolded graphs relative to $\overline{\Gamma}$ with the critical graph. At one extreme we attach chords, at the other we attach corollas, and in between we attach forests that arise as all possible foldings of these chords on their way to the corollas.  

We now show that these forests can be attached to boundary intervals (possibly of length $0$, so points) in the outgoing boundary of the metric fat graph $\overline{\Gamma}$. Those boundary intervals that are not points are described combinatorially as follows:

\begin{definition}
	Let $\Gamma$ be a metric (admissible) fat graph and let $\tau$ be a boundary cycle of $\Gamma$.
	We can think of $\tau$ as a set of half-edges of $\Gamma$ together with a cyclic order.
	A \emph{boundary interval in $\tau$}, denoted $\cB$, is a proper subset of the half-edges of $\tau$ which can be written as
	\[\cB=\{ h_1, h_2=\tau(h_1),h_3=\tau^2(h_1), \ldots,h_n=\tau^{n-1}(h_1)\}\]
	for some half-edge $h_1$  in  $\tau$. In particular, $\cB$ is an ordered set. 
\end{definition}

A boundary interval determines an ordered list of edges in $\Gamma$, in which an edge can appear at most twice. Consecutive edges in this list share a vertex and thus define a path in $\Gamma$ between 
$s(h_1)$ and $s(\iota(h_n))$, where $s$ and $\iota$ are the source and involution maps in the definition of the graph $\Gamma$. Up to scaling there is a canonical map from the unit interval to $\Gamma$ which traces this path and sends $0$ to $s(h_1)$ and $1$ to $s(\iota(h_n))$. By scaling the unit interval, we can construct a canonical map which is an isometry when restricted to the edges of the path. We do this below.

\begin{definition}
	Let $\cB$ be a boundary interval in a boundary cycle $\tau$. We denote by $I_\cB$ an oriented interval whose length is the length of the path in $\Gamma$ determined by $\cB$. More precisely, $I_\cB$ can be subdivided into consecutive subintervals $I_i$ for 
	$1\leq i \leq \vert\cB\vert$.  The length of the $i$-th subinterval $I_i$ is the length of the $i$-th edge 
	$e_i=\{h_i, \iota(h_i)\}$
	on the path determined by $\cB$.  
	We denote by $x_i^-, x_i^+$, the boundary points of $I_i$ using its orientation.

	The \emph{parametrization map of $\cB$} is the unique map 
	\[f_\cB \colon I_\cB \lra \Gamma\]
	which sends $x_1^-\mapsto s(h_1)$, $x_n^+\mapsto s(\iota(h_n))$ and such that for all $i$ it restricts to an isometry $f_\cB\vert \colon I_i \to e_i \coloneqq \{h_i, \iota(h_i)\}$ that sends $x_i^-$ to $s(h_i)$. 
\end{definition}

The map $f_\cB$ is a parametrization of an interval in the boundary component corresponding to $\tau$. Thus a point in $x\in I_\cB$ uniquely determines a way in which a leaf can be attached to $\Gamma$ such that the leaf is in the boundary interval defined by $\cB$. 

We now describe the boundary intervals that will arise given $\overline{\Gamma}\in \cG([\overline{L}])$.

\begin{definition}
	Let $\Gamma'_i$ denote the open graph of $L_i$ relative to $\overline{\Gamma}$ for $1\leq i \leq r$.
	Let $l$ be an unlabeled leaf of $\Gamma'_i$. This leaf defines a boundary cycle $\tau_l$ in $\Gamma_i'$.  
	We define $\cB_l$ to be the subset of the half-edges of $\tau_l$ given by:
	\[\cB_l \coloneqq \{ \tau_l^j (l)\vert 
	j\in\Z,\, j\neq 0,
	\text{ and }
	\tau_l^j (l) \text{ is not part of an edge in an admissible cycle} \}.\]
	Note in particular that $\cB_l$ could be empty and this indeed happens when $l$ is attached to a vertex $v$ which is essentially trivalent in the sense that it has valence four if it is also attached to an admissible leaf but trivalent otherwise.
\end{definition} 

An example of this construction can be seen in Figure \ref{fig:example_single_slit} panel (B), where the dotted lines in $\overline{\Gamma}_i$ for $i=1,2$ correspond precisely to the boundary intervals defined by the leaves of the open graph. The sets $\cB_l$ have the following properties:

\begin{lemma}\label{lem:boundary_intervals}
	For  $1\leq i \leq r$, let $\Gamma_i'$ denote the open graphs of $L_i$ relative to $\overline{\Gamma}\in \mr{Im}(\pi_L)$ as in Definition \ref{def:maxim_unfolded}.  Recall that each unlabelled leaf of $\Gamma_i$, say $l$, corresponds precisely to a shortest slit of $L_i$, and thus it has a ``pair" leaf which we denote by $\lambda(l)$. 
	Then the following hold:
	\begin{enumerate}[(i)]
		\item For any unlabelled leaf $l$ of $\Gamma_i'$, the set $\cB_l$ is either empty or it is a boundary interval in $\overline{\Gamma}$.
		\item For any unlabelled leaf $l$ of $\Gamma_i'$, the sets $\cB_l$ and $\cB_{\lambda(l)}$ are disjoint.
		\item For any pair of unlabeled leaves $l_1$ and $l_2$ in  $\Gamma'_i$ such that $\cB_{l_1}\neq \emptyset\neq \cB_{l_2}$ then either 
		\[\cB_{l_1}\cap\cB_{l_2} = \emptyset \quad\quad\text{or}\quad\quad \cB_{l_1}=\cB_{l_2}.\]
		\item For any open graphs relative to $\overline{\Gamma}$, say $\Gamma'_i$ and $\Gamma'_j$, the set of boundary intervals defined by their unlabelled leaves coincide. 
	\end{enumerate}
\end{lemma}

\begin{proof}
	We first show (i) holds. Let $\zeta_l$ denote the slit corresponding to the unlabelled leaf $l$ in $\Gamma_i'$. Then $\cB_l$ is the section of the outgoing boundary along which the leaf $l$ can move around, given by slit jumps of $\zeta_l$. In particular, if $\zeta_l$ is isolated, that is, it is the only slit on its radial segment, then this is a single point and $\cB_l$ is empty. If $\cB_l$ is not empty it is enough to show that $\cB_l$ is not the entire boundary cycle that corresponds to $l$. Assume by contradiction that $\cB_l$ is the entire boundary cycle.  Then there must be a set of slits in $L_i$, $\{ \zeta_1, \lambda(\zeta_1), \zeta_2, \lambda(\zeta_2), \ldots, \zeta_s, \lambda(\zeta_s)\}$ for some $s\geq 1$ such that the following hold:
	\begin{enumerate}
		\item The slit $\zeta_l$ lies between $\zeta_1$ and $\lambda(\zeta_s)$.  More precisely, $\lambda(\zeta_s), \zeta_l, \zeta_1$ all lie in the same radial segment and $\omega(\zeta_1)=\zeta_l, \omega(\zeta_l)=\lambda(\zeta_s)$.
		\item For each $1\leq i < s$, the slits $\lambda(\zeta_i)$ and $\zeta_{i+1}$ lie in the same radial segment and $\omega(\zeta_{i+1})=\lambda(\zeta_i)$.
	\end{enumerate}
	Let $\zeta_*$ be a slit in $\{\zeta_1, \ldots \zeta_s\}$ of largest modulus, i.e., a shortest slit in that set.  Then $\zeta_*$ and $\lambda(\zeta_*)$ can jump along the other slits.  In particular, $L_i$ is equivalent via slit jumps to a configuration $L_*$ where $\lambda(\zeta_*), \zeta_*$ and $\lambda_l$ lie in the same radial segment and 
	\[\omega(\zeta_*) = \zeta_l,\quad \omega(\zeta_l)=\lambda(\zeta_*)\quad \text{and}\quad \vert\zeta_l\vert\leq \vert\zeta_*\vert.\]
	So $L_*$ and also $L_i$ are degenerate configurations, which is not possible.
	
	Statement (ii) follows in a similar way.  More precisely, if $\cB_l$ and $\cB_{\lambda(l)}$ are not disjoint, then $L_i$ is equivalent via slit jumps to a configuration where $\zeta_l$ and $\lambda(\zeta_l)$ lie next to each other and thus $L_i$ is degenerate. 
	
	Statements (iii) and (iv) follow by construction.
\end{proof}

\begin{definition}[Attaching intervals]
	\label{def:attaching_int}
	Let $\overline{\Gamma}\in \cG(\overline{L})$. Let $I_{L,\overline{\Gamma}}$ be the set of oriented metric intervals (possibly of length zero) corresponding to the parametrization of the boundary intervals and isolated points in $\overline{\Gamma}$ along which a graph can be attached to obtain an element in its preimage.
	
	That is, $\smash{I_{L,\overline{\Gamma}}}$ is given by those $I_{\cB_l}$ such that $l$ is an unlabelled leaf of $\Gamma'$, an open graph relative to $\overline{\Gamma}$ as in Definition \ref{def:maxim_unfolded}.
	This interval is of length zero if its corresponding boundary interval is empty.  Recall that this happens precisely when there is a leaf in $\overline{\Gamma}$ corresponding to an isolated slit, i.e., a slit that is the only one in its radial segment. Note in particular that by Lemma \ref{lem:boundary_intervals} (iv) this definition does not depend on the choice of $\Gamma'$ but only on the class $[L]$ and the metric fat graph $\overline{\Gamma}$.
\end{definition}

Any point in the preimage can be obtained by attaching a forest to $\overline{\Gamma}$ along the parametrization intervals in $I_{L, \overline{\Gamma}}$. To make this precise we define certain spaces of forests attached to intervals, which will use the following combinatorial definition.

\begin{definition}
	Let $\cI \coloneqq I_1\sqcup I_2\sqcup \ldots  \sqcup I_{k}$ 
	denote a disjoint union of $k$ compact intervals of a given length.  We allow intervals to have length zero. Let $\cD$ denote a family of piecewise linear functions $\cD \coloneqq \{d_i \colon I_i\to \R_{>0}\vert 1\leq i\leq k\}$, whose derivative is $\pm 1$ outside a finite set and we define $\max \cD \coloneqq \max_{1 \leq i \leq k} \{\max_{x_i \in I_i} d_i(x_i)\}$.
\end{definition}

\begin{notation}[Configurations of chords]
	We will consider the set of all possible configurations of $k-1$ chords attached by their endpoints to the intervals in $\cI$ such that the resulting graph is: (i) connected, (ii) planar, and (iii) has no loops; we denote this set by $\Conf_{\cI}$.
	See Figure \ref{fig:chords} for examples of configuration of chords. 
	We will construct a space of metric planar forests attached to these intervals and we will use the configurations above to restrict which metrics are allowed.  For this, we will use the path distance function in a metric graph which we denote by $d_\mr{path}$. 
\end{notation}
\begin{figure}
	\begin{tikzpicture}[scale=1.2]
		\draw [thick,PineGreen] (-.6,-1) -- (-1,-.6);
		\draw [thick,PineGreen]  (1,-.6) -- (.6,-1);
		\draw [thick,PineGreen]  (.6,1) -- (1,.6);
		\draw [thick,PineGreen]  (-1,.6) -- (-.6,1);
		\node [PineGreen]  at (.8,.8) {$\bullet$};
		\node [PineGreen] at (-.8,.8) {$\bullet$};
		\node [PineGreen] at (.8,-.8) {$\bullet$};
		\node [PineGreen] at (-.8,-.8) {$\bullet$};
		\draw [Mahogany,thick] (.8,.8) -- (-.8,.8) -- (-.8,-.8) -- (.8,-.8);
		
		\begin{scope}[xshift=3cm]
			\draw [thick,PineGreen]  (-.6,-1) -- (-1,-.6);
			\draw [thick,PineGreen]  (1,-.6) -- (.6,-1);
			\draw [thick,PineGreen]  (.6,1) -- (1,.6);
			\draw [thick,PineGreen]  (-1,.6) -- (-.6,1);
			\node [PineGreen] at (.8,.8) {$\bullet$};
			\node [PineGreen] at (-.8,.8) {$\bullet$};
			\node [PineGreen] at (.8,-.8) {$\bullet$};
			\node [PineGreen] at (-.8,-.8) {$\bullet$};
			\draw [Mahogany,thick] (.8,.8) -- (-.8,.8) -- (.8,-.8) -- (-.8,-.8);
		\end{scope}
	\end{tikzpicture}
	\caption{Two of the 8 configurations of chords for $k=4$. The green line segments are the intervals, the vertices are the marked points in these intervals, and the red arcs are the chords.}
	\label{fig:chords}
\end{figure}

\begin{definition}Let $\cI$ and $\cD$ be as in the previous definition and $d \in \R_{>0}$ such that $2d>\max \cD$. Denote by $\cF_{\cI, \cD,d}$ those metric graphs obtained by attaching a metric forest $F$  with at most $2(k-1)$ leaves to the intervals $\cI$ such that: 
	\begin{itemize}
		\item The graph obtained, denoted by $G$, is planar, connected and has no loops.
		\item There is a configuration $C\in\Conf_{\cI}$ such that for any pair of intervals $I_i, I_j$ connected by a chord in $C$ the path distance in $G$ from $x_i$ to $x_j$ two attaching points of leaves of the forest $F$ is 
		\[d_\mathrm{path}(x_i, x_j)=2d-d_i(x_i)-d_j(x_j).\]
	\end{itemize}
	Note that $\cF_{\cI, \cD,d}$ is a subset of the space of metric fat graphs.  We consider it as a space using the subspace topology.
\end{definition}

\begin{lemma}\label{lem:f-contr} The topological space $\cF_{\cI, \cD,d}$ is contractible.\end{lemma}

\begin{proof}Fix a marked point $\ast_i \in I_i$ for all $1 \leq i \leq k$ such that $\ast_i$ is a local maximum for $d_i$. Let $\cF_{\cI, \cD,d,\ast} \subset \cF_{\cI,\cD,d}$ be the subspace where the forest is attached to the marked points in the intervals $\cI$. We will construct a deformation retraction onto a point in two steps.
	
	\vspace{.5em}
	
	\noindent \emph{Step 1: Deformation retraction onto $\cF_{\cI,\cD,d,\ast}$.} We will construct a deformation retraction of $\cF_{\cI,\cD,d}$ onto $\cF_{\cI, \cD,d,\ast}$. Intuitively, we slide the endpoints along $\cI$ towards the marked points but some care is require to make sure the conditions on the metric remain satisfied. By definition, each $I_i$ can be subdivided into finitely many intervals on which $d_i$ is linear. Let $N_i$ be the number of these in a uniquely minimal such subdivision. Our argument will be by induction over $N = N_1+\ldots+N_k$.
	
	In the initial case $N=0$ there is nothing to prove. For the induction step, let $I' \subset I_j$ be an interval in the aforementioned minimal subdivision such that $I_j = I' \cup I'_j$ with $I' \cap I'_j$ is a point and $\ast_j \in I'_j$. Let $\cI'$ be obtained from $\cI$ by replacing $I_j$ with $I'_j$ and let $\cD'$ be obtained by replacing $d_j$ by $d'_j \coloneqq d_j|_{I'_j}$. We will show that $\cF_{\cI,\cD,d}$ deformation retracts onto a space homeomorphic to $\cF_{\cI',\cD',d}$. There are two cases: 
	\begin{enumerate}[(A)]
		\item \emph{The point $I' \cap I'_j$ is a local minimum of $d_j$.} In this case we ``open'' along the edge $I'$ towards $I'_j$:
		\[\begin{tikzpicture}[scale=1.1]
			\draw [Mahogany,thick] (-1,1) -- (0,0);
			\draw [thick,PineGreen] (0,0) -- (.5,.5);
			\draw [dotted,thick,PineGreen] (.5,.5) -- (.7,.3);
			\node at (-1,.7) [Mahogany] {$I'$};
			\node at (.5,.1) [PineGreen] {$I'_j$};
			\draw (-1.2,2) -- (-1,1);
			\draw (-.8,2) -- (-1,1);
			\draw (-.6,2) -- (-.5,.8);
			\draw (-.4,2) -- (-.5,.8);
			\draw (-.5,.8) -- (-.5,.5);
			
			\begin{scope}[xshift=3cm]
				\draw [Mahogany,thick] (-1,1) -- (0,0);
				\draw [thick,PineGreen] (0,0) -- (.5,.5);
				\draw [dotted,thick,PineGreen] (.5,.5) -- (.7,.3);
				\node at (-1,.7) [Mahogany] {$I'$};
				\node at (.5,.1) [PineGreen] {$I'_j$};
				\draw (-1.2,2.2) -- (-1,1.2);
				\draw (-.8,2.2) -- (-1,1.2);
				\draw (-1,1.2) -- (-.5,.5);
				\draw (-.6,2) -- (-.5,.8);
				\draw (-.4,2) -- (-.5,.8);
				\draw (-.5,.8) -- (-.5,.5);
				
			\end{scope}
			
			\begin{scope}[xshift=6cm]
				\draw [Mahogany,thick] (-1,1) -- (0,0);
				\draw [thick,PineGreen] (0,0) -- (.5,.5);
				\draw [dotted,thick,PineGreen] (.5,.5) -- (.7,.3);
				\node at (-1,.7) [Mahogany] {$I'$};
				\node at (.5,.1) [PineGreen] {$I'_j$};
				\draw (-1.2,2.2) -- (-1,1.2);
				\draw (-.8,2.2) -- (-1,1.2);
				\draw (-.6,2.1) -- (-.5,.9);
				\draw (-.4,2.1) -- (-.5,.9);
				\draw (-1,1.2) -- (0,0);
				\draw (-.5,.9) -- (-.5,.6);
			\end{scope}
			
			\draw [->] (-1,-1) -- (6.5,-1);
			\node at (-1,-.7) {$t=0$};
			\node at (6.5,-.7) {$t=1$};
		\end{tikzpicture}\]
		The precise construction is as follows. If $I'$ has length $\ell$ we linearly identify the interval $I'$ by $[0,\ell]$, with $I' \cap I'_j$ corresponding to $\ell$. Suppose that $s \in [0,\ell]$ is the unique smallest value at which an edge is attached to $I' \cong [0,\ell]$. Then on a metric graphs $G$ the deformation retraction at time $t \in [0,1]$ is the identity for $t\ell<s$ and for $t\ell \geq s$ replaces $I' \cong [0,\ell]$ with $[0,\ell] \cup_{t \ell} [s \ell,t \ell]$; note we may identify $[t\ell,\ell] \cup_{t\ell} [s\ell,t\ell] \subset [0,\ell] \cup_{t\ell} [s\ell,t\ell]$ with $[s\ell,t\ell]$. We attach the edges originally attached to $[s\ell,t\ell] \subset I'$ to this new interval. The result has a canonical the metric.
		\item \emph{The point $I' \cap I'_j$ is a local maximum on $d_j$.} In this case we ``fold'' along the edge $I'$ towards $I'_j$:
		\[\begin{tikzpicture}[scale=1.1]
			\draw [Mahogany,thick] (-1,0) -- (0,1);
			\draw [thick,PineGreen] (0,1) -- (.5,.5);
			\draw [dotted,thick,PineGreen] (.5,.5) -- (.7,.3);
			\node at (-.8,-.1) [Mahogany] {$I'$};
			\node at (.5,.1) [PineGreen] {$I'_j$};
			\draw (-1.1,2) -- (-1,0);
			\draw (-.9,2) -- (-1,0);
			\draw (-.6,2) -- (-.5,1.2);
			\draw (-.4,2) -- (-.5,1.2);
			\draw (-.5,1.2) -- (-.5,.5);
			
			\begin{scope}[xshift=3cm]
				\draw [Mahogany,thick] (-1,0) -- (0,1);
				\draw [thick,PineGreen] (0,1) -- (.5,.5);
				\draw [dotted,thick,PineGreen] (.5,.5) -- (.7,.3);
				\node at (-.8,-.1) [Mahogany] {$I'$};
				\node at (.5,.1) [PineGreen] {$I'_j$};
				\draw (-1.1,2) -- (-.5,.5);
				\draw (-.9,2) -- (-.5,.5);
				\draw (-.6,2) -- (-.5,1.2);
				\draw (-.4,2) -- (-.5,1.2);
				\draw (-.5,1.2) -- (-.5,.5);
			\end{scope}
			
			\begin{scope}[xshift=6cm]
				\draw [Mahogany,thick] (-1,0) -- (0,1);
				\draw [thick,PineGreen] (0,1) -- (.5,.5);
				\draw [dotted,thick,PineGreen] (.5,.5) -- (.7,.3);
				\node at (-.8,-.1) [Mahogany] {$I'$};
				\node at (.5,.1) [PineGreen] {$I'_j$};
				\draw (-1.1,2) -- (0,1);
				\draw (-.9,2) -- (0,1);
				\draw (-.6,2) -- (-0.05,1.2);
				\draw (-.4,2) -- (-0.05,1.2);
				\draw (-0.05,1.2) -- (0,1);
			\end{scope}
			
			\draw [->] (-1,-1) -- (6.5,-1);
			\node at (-1,-.7) {$t=0$};
			\node at (6.5,-.7) {$t=1$};
		\end{tikzpicture}\]
		The precise construction is as follows. Let us linearly identify the interval $I'$ by $[0,\ell]$ as in case (A). Then the subtree of $G$ given by points that are distance $t\ell$ from $0 \in I' \in [0,\ell]$. We identify this subtree with the interval $[0,t\ell]$ by identifying all points with distance $s$ to $s \in [0,t\ell]$. The result has a canonical metric.
	\end{enumerate}
	
	\vspace{.5em}
	
	\noindent \emph{Step 2: $\cF_{\cI,\cD,d,\ast}$ is contractible.} We will prove that $\cF_{\cI, \cD,d,\ast}$ is contractible by a variation of the Alexander trick. To do so, we replace the metric tree $(T,d_T)$ attached to the marked points by $(T,(1-t)d_T)$ and add edges of length $t(d-d_i(\ast_i))$ connecting $\ast_i$ to the endpoint in this scaled tree originally attached to $\ast_i$. (The circles contain the rescaled graphs.)
	\[\begin{tikzpicture}[scale=.6]
		\draw [thick,PineGreen] (-1.2,-2) -- (-2,-1.2);
		\draw [thick,PineGreen] (2,-1.2) -- (1.2,-2);
		\draw [thick,PineGreen] (1.2,2) -- (2,1.2);
		\draw [thick,PineGreen] (-2,1.2) -- (-1.2,2);
		\node [PineGreen] at (1.6,1.6) {$\bullet$};
		\node [PineGreen] at (-1.6,1.6) {$\bullet$};
		\node [PineGreen] at (1.6,-1.6) {$\bullet$};
		\node [PineGreen] at (-1.6,-1.6) {$\bullet$};
		\draw (1.6,1.6) -- (1,0) -- (-1.6,1.6);
		\draw (1,0) -- (-1.6,-1.6) -- (1.6,-1.6);
		\draw [PineGreen,dotted] (0,0) circle (2.26 cm);
		
		\begin{scope}[xshift=6cm]
			\draw [thick,PineGreen] (-1.2,-2) -- (-2,-1.2);
			\draw [thick,PineGreen] (2,-1.2) -- (1.2,-2);
			\draw [thick,PineGreen] (1.2,2) -- (2,1.2);
			\draw [thick,PineGreen] (-2,1.2) -- (-1.2,2);
			\node [PineGreen] at (1.6,1.6) {$\bullet$};
			\node [PineGreen] at (-1.6,1.6) {$\bullet$};
			\node [PineGreen] at (1.6,-1.6) {$\bullet$};
			\node [PineGreen] at (-1.6,-1.6) {$\bullet$};
			\draw (1.6,1.6) -- (.8,.8) -- (.5,0) -- (-0.8,0.8) -- (-1.6,1.6);
			\draw (.5,0) -- (-.8,-.8) -- (.8,-.8) -- (1.6,-1.6);
			\draw (-1.6,-1.6) -- (-.8,-.8);
			\draw [PineGreen,dotted] (0,0) circle (1.13 cm);
		\end{scope}
		
		\begin{scope}[xshift=12cm]
			\draw [thick,PineGreen] (-1.2,-2) -- (-2,-1.2);
			\draw [thick,PineGreen] (2,-1.2) -- (1.2,-2);
			\draw [thick,PineGreen] (1.2,2) -- (2,1.2);
			\draw [thick,PineGreen] (-2,1.2) -- (-1.2,2);
			\node [PineGreen] at (1.6,1.6) {$\bullet$};
			\node [PineGreen] at (-1.6,1.6) {$\bullet$};
			\node [PineGreen] at (1.6,-1.6) {$\bullet$};
			\node [PineGreen] at (-1.6,-1.6) {$\bullet$};
			\draw (1.6,1.6) -- (0,0) -- (-1.6,1.6);
			\draw (-1.6,-1.6) -- (0,0) -- (1.6,-1.6);
		\end{scope}
		
		\draw [->] (-2,-3.5) -- (14,-3.5);
		\node at (-2,-3) {$t=0$};
		\node at (14,-3) {$t=1$};
	\end{tikzpicture}\]
	The resulting metric graphs are still planar, connected, without loops, and satisfy the metric condition. At $t=1$ we obtain the $k$-valent corolla attached to all intervals, with edge between the vertex of the corolla and $\ast_i$ given by $d-d_i(\ast_i)$. 
\end{proof}

\begin{lemma}\label{lem:f-is-preimage}
	Let $\overline{\Gamma}\in\cG([\overline{L}])$.  There is a positive real number $d\in \R_{>0}$ and a finite collection of sets of intervals $\cI$ and sets of functions $\cD$, such that there is a homeomorphism
	\begin{equation}
		\label{eq:homeo_forest_preim}
		\pi^{-1}_L(\overline{\Gamma})\cong \cF_{\cI, \cD,d} \times \cdots \times \cF_{\cI',\cD',d}.
	\end{equation}
\end{lemma}

The intuition behind this homeomorphism is as follows. In the simplest scenario, there is only one term in the product of the right hand side of (\ref{eq:homeo_forest_preim}). On the one hand, the critical graph corresponds to the unique point in $\cF_{\cI,\cD,d}$ given by a single corolla. On the other hand, the maximally unfolded graphs relative to $\overline{\Gamma}$ correspond to elements in $\Conf_{\cI}$, that is, to arrangements of $k-1$ cords attached to the intervals (where $k-1$ is the number of pairs of shortest slits of $L$). Finally, an arbitrary point in $\cF_{\cI, \cD,d}$ is a ``folding" of a configuration in  $\Conf_{\cI}$, and an arbitrary point in $\pi^{-1}_L(\overline{\Gamma})$ is a ``folding'' of a maximally unfolded graph relative to $\overline{\Gamma}$.

\begin{proof}
	Given $[L]$ and $\overline{\Gamma}$, the set of intervals will be $\cI=I_{L, \overline{\Gamma}}$; see Definition \ref{def:attaching_int}.  
	Recall that there is a map 
	\[f \colon I_{L,\overline{\Gamma}}\lra \overline{\Gamma},\]
	which is an isometry when restricted to edges of $\overline{\Gamma}$ that are in the image. Moreover, we have a canonical embedding $\overline{\Gamma}\cof \Gamma$ for which $\Gamma-\overline{\Gamma}=F$ is a forest
	and such that $\Gamma$ is obtained from $\overline{\Gamma}$ by attaching the leaves of $F$ to $I_{L,\overline{\Gamma}}$, see Remark \ref{rmk:attaching_forest}. 
	
	For a choice of $\Gamma$ in the preimage, we denote by $G_{\Gamma}$ the subgraph of $\Gamma$ that is given by the union of the forest $F$ and the boundary intervals in $I_{L,\overline{\Gamma}}$ along which $F$ is attached.  
	The number of components of $G_{\Gamma}$ is independent from the choice of $\Gamma$ in the preimage of $\overline{\Gamma}$ and it corresponds to the number of elements in the product of the right hand side of (\ref{eq:homeo_forest_preim}).  
	An intuitive way to think about this, is that the slits which are deleted from $L$ to obtain $\overline{L}$ come in \emph{clusters}, collections of slits which map to the same point in the glued surface $\Sigma([L])$, and each of these clusters contributes a single term in the product.
	
	We will assume for the sake of simplicity there is a single component in $G_{\Gamma}$ or a single \emph{cluster} of slits, thought the argument easily generalizes to the case of several components. 
	The functions $d_i \in \cD$ are induced by the modulus in $\bC$.  That is, they are determined by the path distance to the admissible cycles of $\overline{\Gamma}$. 
	More precisely, for any $x\in I_i\in I_{L,\Gamma}$ we set $d_i(x) = d_{\mr{ad}}(x)$.  This yields a well-defined piecewise-linear function on each $I_i$.
	The real number $d$ is the common modulus of all slits which are deleted from $L$ to obtain $\overline{L}$. Then there is a continuous map $\cF_{\cI,\cD,d} \to \pi^{-1}_L(\overline{L})$ given by gluing the forest $F$ into $\overline{\Gamma}$ according to the intervals $\cI_i$. This has an inverse given by the continuous map that sends  $\Gamma$ to $G_{\Gamma}$.\end{proof}

\begin{comment}	
	The slits which are deleted from $L$ to obtain $\overline{L}$ come in \emph{clusters}, collections of slits which map to the same point in the glued surface $\Sigma([L])$. Each of these clusters contributes a single term in the product, and we will assume for the sake of simplicity there is a single cluster.
	
	We then let $\cI$ be the disjoint union of the boundary intervals $\cB_i\in \cB_{L,\overline{\Gamma}}$ and an isolated point given by the point $z_i \in \partial_\mr{in} \bA_j$ with $\arg(z_i) = \arg(\zeta_i)$ for each slit $\zeta_i \in \bA_j$ deleted from $L$ to obtain $\overline{L}$ which is the unique slit on its radial segment. The functions $d_i \in \cD$ are induced by the modulus in $\bC$, which yields a well-defined piecewise-linear function on $\cB_i$ and $z_i$. The real number $d$ is the common modulus of all slits which are deleted from $L$ to obtain $\overline{L}$. 
	
	Then there is a map $\cF_{\cI,\cD,d} \to \pi^{-1}_L(\overline{L})$ given by gluing the forest $F$ into $\overline{\Gamma}$ according to the intervals $\cI_i$. This has an inverse given by recording the forest that was deleted from $\Gamma \in \cG([L])$ to obtain $\pi_L(\Gamma) = \overline{\Gamma}$ and the attaching point of its leaves.
\end{comment}

Putting together these results we prove that the preimages of $\pi_L$ are contractible. 
\begin{proof}[Proof of Lemma \ref{lem:preimage_contractible}]
	Let $\overline{\Gamma}\in \cG([\overline{L}])$.  By Lemma \ref{lem:f-is-preimage}, $\pi^{-1}_L(\overline{\Gamma})$ is homeomorphic to a product of spaces of forests attached at intervals.  These are contractible by Lemma \ref{lem:f-contr}.
\end{proof}

The proofs given above for $\pi_L$ can be adapted to the simpler case of $\pi^1_L$, and we will spare the reader the technical details. The result is:

\begin{lemma}\label{lem:pil1-props} \
	\begin{enumerate}[(i)]
		\item $\pi_L$ is well-defined.
		\item $\pi_L$ is continuous.
		\item The fibers of $\pi_L$ are compact, contractible ANR's.
\end{enumerate}\end{lemma}

We now finish the proof of Lemma \ref{forget_slits_homeq}, which said $\pi_L$ and $\pi^1_L$ are homotopy equivalences:

\begin{proof}[Proof of Lemma \ref{forget_slits_homeq}]
	We apply Theorem \ref{thm_lacher}. By Lemma \ref{lemcglpolyhedron} the domain and targets of the maps $\pi_L$ and $\pi^1_L$ are compact ANR's, so it suffices to prove the fibers of both maps are cell-like. This follows by combining Proposition \ref{prop_anrprops} (viii) with Lemma's \ref{lem:pil-props}, \ref{lem:preimage_contractible} and \ref{lem:pil1-props}.
\end{proof}

\subsection{The projection map is a homotopy equivalence}
\label{subsec_projrad}

Our next goal is to check that the spaces $\Rad$ and $\Radt$ are ANR's and that the map $\pi_1 \colon \Radt \to \Rad$ is proper and cell-like. For the remainder of this section we fix $g$, $n$ and $m$.

\begin{proposition}\label{prop_radanr} The space $\Rad$ is a locally compact ANR.\end{proposition}

\begin{proof}The space $\Rad$ is a smooth manifold, so it is locally compact and has an open cover by $\bR^n$'s. The latter are ANR's by property (v) of Proposition \ref{prop_anrprops}, so $\Rad$ is an ANR by property (iii) of Proposition \ref{prop_anrprops}. (Alternatively one can argue that $\Rad$ is an open subspace of the finite CW-complex $\bRad$ and use properties (ii) and (v) of Proposition \ref{prop_anrprops}.)\end{proof}

To prove that $\Radt$ is an ANR and that $\pi_1$ is a proper cell-like map, we will write $\Radt$ as an open subspace of a space $(\bRad)^\sim$ obtained by glueing together finitely many compact ANR's. By Definition \ref{deffrrad}, $\bRad \backslash \Rad = \Rad'$ is a CW-complex, and in fact a subcomplex of $\bRad$. Then $(\bRad)^\sim$ is defined by adding a boundary to the blowup $\Radt$ in the most naive way. In the proof of Lemma \ref{lemcglpolyhedron}, we saw that $\MFatad_{g,n+m}$ is a subspace of a compact polyhedron $P_{g,n+m}$, which we abbreviate to $P$ here.

\begin{definition}The space $(\bRad)^\sim$ is the subspace of $\bRad \times P$ consisting of all pairs $([L],\Gamma,\lambda)$ such that either 
	\begin{enumerate}[(i)]
		\item $[L] \in \Rad$ and $(\Gamma,\lambda) \in \cG(L)$, or 
		\item $[L] \in \bRad \backslash \Rad$ and $(\Gamma,\lambda) \in P$.
	\end{enumerate} 
\end{definition}

\begin{lemma}\label{lem_radtar} The topological space $(\bRad)^\sim$ is a compact ANR.\end{lemma}

\begin{proof} Fix a representative $[L]$ for each combinatorial type $[\cL]$ and note that if $[L]$ and $[L']$ have the same combinatorial type, there is a canonical homeomorphism $\cG([L]) \cong \cG([L'])$. The space $\cG([\cL])$ is then by definition $\cG([L])$ for the representative $[L]$ of $[\cL]$. Remark that $(\bRad)^\sim$ is obtained by glueing together $\bRad \backslash \Rad \times P$ and $\bRad_{[\cL]} \times \cG([\cL])$ for all combinatorial types $[\cL]$ along $\partial\bRad_{[\cL]} \times \cG([\cL])$.
	
	Note that $\bRad \backslash \Rad \times P$ is the product of a subcomplex of the finite complex $\bRad$ with a compact polyhedron. Thus parts (v) and (vii) of Proposition \ref{prop_anrprops} say it is a compact ANR. Similarly, by Lemma \ref{lemcglpolyhedron} we have that $\bRad_{[\cL]} \times \cG([\cL])$ and $\partial \bRad_{[\cL]} \times \cG([\cL])$ are each a product of a finite CW-complex with a compact polyhedron, and thus compact ANR's by parts (v), (vi) and (vii) of Proposition \ref{prop_anrprops}. Attaching cells $\bRad_{[\cL]}$ one at a time in order of dimension and repeatedly applying property (iv) of Proposition \ref{prop_anrprops}, one proves inductively over $k$ that 
	\[\left(\bRad \backslash \Rad \times P\right) \cup \left(\bigcup_{\dim \bRad_{[\cL]} \leq k} \bRad_{[\cL]} \times \cG([\cL])\right)\]
	is a compact ANR. This uses that $\bRad$ has finitely many cells after fixing $g$, $n$ and $m$. In particular this process has to end at some $k \geq 0$ and hence $(\bRad)^\sim$ is also a compact ANR.
\end{proof}

\begin{proposition}\label{prop_radtanr} The topological space $\Radt$ is an ANR.\end{proposition}

\begin{proof}$\Radt$ is an open subspace of $(\bRad)^\sim$ and by property (ii) of Proposition \ref{prop_anrprops} we conclude it is an ANR.\end{proof}

\begin{proposition}\label{prop_pi1celllike} The map $\pi_1 \colon \Radt \to \Rad$ is proper and cell-like.\end{proposition}

\begin{proof}Observe $\pi_1$ extends to a continuous map $\bar{\pi}_1 \colon  (\bRad)^\sim \to \bRad$. If $K \subset \Rad$ is compact, then it is also compact considered as a subset of $\bRad$ and thus closed. By continuity $\bar{\pi}_1^{-1}(K)$ is closed in $(\bRad)^\sim$ and since the latter is a compact space it must be compact. But $\bar{\pi}_1^{-1}(K) \subset \Radt$ and $\bar{\pi}_1^{-1}(K) \cap \Radt = \pi_1^{-1}(K)$, so that $\pi_1$ is proper. 
	
	That $\pi_1$ is cell-like is a consequence of Lemma's \ref{GL_contractible} and \ref{lemcglpolyhedron}, which say that the point inverses of $\pi_1$ are contractible compact polyhedra, and property (viii) in Proposition \ref{prop_anrprops}, which implies that contractible compact polyhedra are cell-like.
\end{proof}

\begin{corollary}\label{cor_pi1homeq} The projection $\pi_1 \colon \Radt \to \Rad$ is a homotopy equivalence.\end{corollary}

\begin{proof}We may fix $g$, $n$ and $m$. Then we can simply apply Theorem \ref{thm_lacher} to Propositions \ref{prop_radanr}, \ref{prop_radtanr} and \ref{prop_pi1celllike}. The domain is locally compact by because it is an open subspace of a compact space by Lemma \ref{lem_radtar} and the target is locally compact by Proposition \ref{prop_radanr}.\end{proof}

\subsection{The critical graph map is a homotopy equivalence}
\label{subsec_critgraph}
We now show that the critical graph map $\Radt\to\MFatad$ is a homotopy equivalence using the relation between the universal bundles over $\Rad$ and $\MFatad$.  We start by recalling some well-known results regarding universal bundles.

\begin{proposition}
	Given a two-dimensional cobordism $\Sg$ and a paracompact base space $B$, there are bijections natural in $B$ between
	\begin{enumerate}[(i)]
		\item isomorphism classes of smooth $\Sg$-bundles over $B$, i.e. the transition functions lie in $\Diff(S_{g,n+m})$,
		\item isomorphism classes of principal $\Diff(S_{g,n+m})$-bundles over $B$, and
		\item isomorphism classes of principal $\Modgpq$-bundles over $B$.
	\end{enumerate}
\end{proposition}

\begin{proof}[Sketch of proof]For the one direction of the first bijection, for a principal $\Diff(S_{g,n+m})$-bundle $p \colon W\to B$, its corresponding $\Sg$-bundle is given by taking $\Sg\times_{\Diff(S_{g,n+m})} W$. 
	
	For the other direction of the first bijection, suppose that $\pi \colon E\to B$ is a smooth $\Sg$-bundle.  Each fiber $E_b \coloneqq \pi^{-1}(b)$ is a Riemann surface with boundary with a marked point in each boundary component.  These marked points are ordered and labeled as incoming or outgoing.  Let $x^b_k$ denote the marked point in the $k$th incoming boundary component for $1\leq k\leq n$ and $x^b_{k+n}$ denote the marked point in the $k$th outgoing boundary $1\leq k\leq m$.  Its corresponding $\Diff(\Sg)$-bundle is given by taking fiberwise orientation-preserving diffeomorphisms i.e. it is the bundle $p \colon W\to B$ whose fibers are given by 
	\[W_b \coloneqq p^{-1}(b)=\lbrace \varphi \colon \Sg\to E_b\,\vert\,\varphi \text{ is a diffeomorphism, }\varphi(x_i)=x^b_i\rbrace\]
	These constructions are mutually inverse.  
	
	Because each connected component of $\Diff(S_{g,n+m})$ is contractible, taking $\pi_0$ gives a homotopy equivalence $\Diff(S_{g,n+m}) \to \Modgpq$. Thus there is a bijection between principal $\Diff(S_{g,n+m})$-bundles and principal $\Modgpq$-bundles, where one can obtain the $\Modgpq$-bundle corresponding to $p \colon W\to B$ by taking $\pi_0$ fiberwise.\end{proof}

We now construct a space $\ERad$ that maps to $\Rad$ and use the previous proposition to show that $\ERad\to\Rad$ is a universal $\Modgpq$-bundle.  To construct this space we use the ideas of the construction of $\EMad$ in Definition \ref{deffatuniversal}. That is, as a set we define
\[\ERad \coloneqq \lbrace ([L],[H])\,\vert\, [L]\in\Rad, [H] \text{ is a marking of } \Gamma_{[L]}\rbrace.\]
We will topologize $\ERad$ so that the map $\ERad \to \Rad$ is a covering map. Then a path in $\ERad$ will be given by a path $\gamma \colon t\to [L(t)]$ in $\Rad$ together with a marking $H_0 \colon \Gamma_{[L(0)]}\cof\Sg$. Hence we must describe how $H_{0}$ and the path $\gamma$ uniquely determine a sequence of markings $H_t \colon \Gamma_{[L(t)]} \hookrightarrow \Sg$. To make this precise, we will give a procedure to obtain a well defined marking of $\Gamma_{\tilde{[\cL]}}$ from a combinatorial type ${[\cL]}$, a marking of $\Gamma_{[\cL]}$ and a configuration $[\tilde{L}]\in\partial\overline{\Rad}_{[\cL]}$, where $\tilde{[\cL]}$ is the combinatorial type of $[\tilde{L}]$. To describe this procedure, notice that if $[\cL]$ and $\tilde{[\cL]}$ are related in this manner, then $\tilde{[\cL]}$ must be obtained from $[\cL]$ by collapsing radial and annular chambers. Hence, we will start by analyzing these cases separately.

\begin{definition}[Annular chamber collapse map]
	Let $[\cL]$ and $[\cL']$ be two non degenerate combinatorial types such that $[\cL']$ can be obtained from $[\cL]$ by collapsing the annular chambers $A_{i_1},A_{i_2},\ldots, A_{i_k}$ and let  $A:=\cup_i A_i$.  We will define a map in $\Fatad$
	\[\rho \colon \Gamma_{[\cL]}\to\Gamma_{[\cL']}\]
	which we will call the \emph{annular chamber collapse map} (see Figure \ref{radial_collapse}).
	
	Choose a representative $[L]$ of $[\cL]$.  Then following the construction of $\Gamma_{[L]}$ we can define a subgraph $F_{A}$ which is given by the intersection of $E_{L}$ and $A$.  The subgraph $F_A$ must be a forest inside $\Gamma_{[L]}$.  To see this, assume there is a loop in $F_A$. Then there must be a loop in $\Gamma_{[L]}$ and hence there are two paired slits $\zeta_i$, $\zeta_{\lambda(i)}$ which lie on the same radial segment. Since $[L]$ is non-degenerate there must be slits $\zeta_{i_1}, \zeta_{i_2}, \ldots, \zeta_{i_j}$ such that $i_j\geq 1$ and $\vert\zeta_{i_l}\vert<\vert\zeta_i\vert$ for all $i_l$.  Finally, since the loop is in $F_A$, $A$ must contain the radial segment between $\zeta_i$ and $\zeta_{i_l}$ for some $i_l$, but then collapsing $A$ will give a degenerate configuration and we assumed $[\cL']$ is non-degenerate. Therefore $F_A$ is a forest in $\Gamma_{[L]}$ and since $\Gamma_{[L]}=\Gamma_{[\cL]}$ this description gives a well defined subforest of $\Gamma_{[\cL]}$ giving with a well defined map on $\Fatad$.
\end{definition}

\begin{figure}[h!]
	\centering
	\begin{tikzpicture}
		\fill [yellow!20!white] (0,0) circle (.9cm);
		\fill [white] (0,0) circle (.7cm);
		\node at (0,0) {\footnotesize $1$};
		\draw (45:1.2cm) -- (45:1.5cm);
		\node at (45:1.6cm) {\footnotesize $1$};    
		\draw (-45:1.2cm) -- (-45:1.5cm);
		\node at (-45:1.6cm) {\footnotesize $2$};  
		\draw (-135:1.2cm) -- (-135:1.5cm);
		\node at (-135:1.6cm) {\footnotesize $3$};   	
		\draw (135:1.2cm) -- (135:1.5cm);
		\node at (135:1.6cm) {\footnotesize $4$};
		
		\draw [dashed] (.5,0) -- (1.2,0);
		\draw [dotted] (0,0) circle (.7cm);
		\draw [dotted] (0,0) circle (.9cm);
		\draw [dotted] (0,0) circle (1.1cm);
		
		\draw [orange,thick,yshift=.6mm] (0:.7cm) -- (0:1.2cm);
		\draw [orange,thick] (90:.7cm) -- (90:1.2cm);
		\draw [blue,thick] (180:.9cm) -- (180:1.2cm);
		\draw [blue,thick] (0:.9cm) -- (0:1.2cm);
		\draw [PineGreen,thick] (-90:1.1cm) -- (-90:1.2cm);
		\draw [PineGreen,thick,yshift=-.6mm] (0:1.1cm) -- (0:1.2cm);
		
		\draw (0,0) circle (1.2cm);
		\draw  (0,0) circle (.5cm);
		\node at (0,1.75) {$L$};
		
		\draw [->] (1.5,0) -- node [pos=.5,above] {\tiny annular chamber collapse} (4.5,0);
		\draw [dashed,->] (0,-1.75) -- (0,-2.5);
		
		\begin{scope}[xshift=6cm]
			\node at (0,0) {\footnotesize $1$};
			\draw (45:1.2cm) -- (45:1.5cm);
			\node at (45:1.6cm) {\footnotesize $1$};    
			\draw (-45:1.2cm) -- (-45:1.5cm);
			\node at (-45:1.6cm) {\footnotesize $2$};  
			\draw (-135:1.2cm) -- (-135:1.5cm);
			\node at (-135:1.6cm) {\footnotesize $3$};   	
			\draw (135:1.2cm) -- (135:1.5cm);
			\node at (135:1.6cm) {\footnotesize $4$};
			
			\draw [dashed] (.5,0) -- (1.2,0);
			\draw [dotted] (0,0) circle (.7cm);
			\draw [dotted] (0,0) circle (1.1cm);
			
			\draw [orange,thick,yshift=.6mm] (0:.7cm) -- (0:1.2cm);
			\draw [orange,thick] (90:.7cm) -- (90:1.2cm);
			\draw [blue,thick] (180:.7cm) -- (180:1.2cm);
			\draw [blue,thick] (0:.7cm) -- (0:1.2cm);
			\draw [PineGreen,thick] (-90:1.1cm) -- (-90:1.2cm);
			\draw [PineGreen,thick,yshift=-.6mm] (0:1.1cm) -- (0:1.2cm);
			
			\draw (0,0) circle (1.2cm);
			\draw  (0,0) circle (.5cm);
			\node at (0,1.75) {$L'$};
			\draw [dashed,->] (0,-1.75) -- (0,-2.5);
		\end{scope}
		
		\begin{scope}[yshift=-4.5cm,scale=.7]
			\draw  (0,0) circle (1cm);
			\draw (.7,0) -- (1,0);
			\draw (90:1) edge[bend left=90,looseness=2] node (firstmid) [pos=.5] {} (0:1);
			\draw (180:1) edge[bend left=140,looseness=3] node (secondmid) [pos=.75] {} (firstmid.center);
			\draw (-90:1) edge[bend right=110,looseness=1.5] (secondmid.center);   	
			\begin{scope}
				\clip (secondmid.center) rectangle ($(firstmid.center)+(1,0)$);
				\draw [thick,red] (180:1) edge[bend left=140,looseness=3] (firstmid.center);
			\end{scope}
			\node at (0,-1.75) {$\Gamma_L$};	
			\draw [->] (2,0) -- node [pos=.5,above] {\tiny edge collapse} ({4.5/.7},0);
		\end{scope}
		
		\begin{scope}[yshift=-4.5cm,xshift=6cm,scale=.7]
			\draw  (0,0) circle (1cm);
			\draw (.7,0) -- (1,0);
			\draw (90:1) edge[bend left=90,looseness=2] node (firstmid2) [pos=.5] {} (0:1);
			\draw (180:1) edge[bend left=140,looseness=3] node (secondmid2) [pos=.75] {} (firstmid2.center);
			\draw (-90:1) edge[bend right=110,looseness=1.5] (firstmid2.center); 
			\node at (0,-1.75) {$\Gamma_{L'}$};  
		\end{scope}
	\end{tikzpicture}
	\caption{An example of the annular chamber collapse map. The leaves have been omitted from the graphs to make them more readable. The annular chambers are marked with dotted lines. The yellow radial sector is collapsed in $L$ and the annular chamber collapse map is given by contracting the edge shown in red.}
	\label{radial_collapse}
\end{figure}

\begin{definition}[Radial chamber collapse zigzag]
	Let $[\cL]$ and $[\cL'']$ be two non degenerate combinatorial types such that $[\cL'']$ can be obtained from $[\cL]$ by collapsing radial chambers. We will define an admissible fat graph $\Gamma([\cL],[\cL''])$ together with a zigzag in $\Fatad$
	\[\Gamma_{[\cL]}\stackrel{\tau_1}{\longrightarrow}
	\Gamma([\cL],[\cL''])\stackrel{\tau_2}{\longleftarrow}
	\Gamma_{[\cL']},\]
	which we will call the \emph{radial chamber collapse zigzag} (see Figure \ref{chamber_collapse}).
	
	Choose a representative $L\in\QRad$ of combinatorial type $[\cL]$ and let $L''\in\QRad$ be the preconfiguration of combinatorial type $[\cL'']$ obtained by collapsing radial chambers.  We will call the radial segments onto which the radial chambers have been collapsed the \emph{special radial segments}.  Notice that $L''$ is well defined up to a choice of $L$, and slit jumps and parametrization point jumps away from the special radial segments.  Thus the idea is to define $\Gamma([\cL],[\cL''])$ as a partially unfolded graph of $L''$ which is unfolded at the special radial slit segments and folded everywhere else.  This gives a well-defined isomorphism class of admissible fat graphs. 
	
	To make this precise, let $S_{k_1}, S_{k_2}, \ldots, S_{k_r}$ denote the special radial segments of $L''$.  We define $\Gamma([\cL],[\cL''])=\Gamma_{L'',t}$ where $t\in [0,1]^{d(L'')}$ is defined as follows:
	\begin{equation*}
		t_\alpha \coloneqq \begin{cases}
			0 & \text{if } \alpha=k_i+j \text{ for } 1\leq i \leq r \text{ and } 1\leq j\leq s_{k_i}-1,\\
			1 & \text{else.}\end{cases}
	\end{equation*}
	This is a well-defined isomorphism class of admissible fat graphs, since the graph is folded in all radial segments in which jumps are allowed.  Let $F_L$ be the subgraph of $\Gamma_L$ obtained by the intersection of $E_L$ with the collapsing chambers.  Then $\tau_1 \colon \Gamma_{[\cL]}=\Gamma_L\to \Gamma_L/ F_L=\Gamma([\cL],[\cL''])$ is a well defined map in $\Fatad$.  Similarly let $F_{L''}$ be the subgraph of $\Gamma_{L''}$ obtained from the intersection of $E_{L''}$ and the special radial segments.  Then $\tau_2 \colon \Gamma_{[\cL'']}=\Gamma_{L''}\to \Gamma_{L''}/ F_{L''}=\Gamma([\cL],[\cL''])$ is a well-defined map in $\Fatad$.
\end{definition}

\begin{figure}[h!]
	\centering
	\begin{tikzpicture}
		\fill [yellow!20!white] (0,0) ++ (0:.5) arc (0:-60:.5) -- (-60:1.2) arc (-60:0:1.2) -- cycle;
		\node at (0,0) {\footnotesize $1$};
		\draw (90:1.2cm) -- (90:1.5cm);
		\node at (90:1.6cm) {\footnotesize $1$};    
		\draw (-150:1.2cm) -- (-150:1.5cm);
		\node at (-150:1.6cm) {\footnotesize $2$};  
		\draw (-90:1.2cm) -- (-90:1.5cm);
		\node at (-90:1.6cm) {\footnotesize $3$};   	
		
		\draw [dashed] (.5,0) -- (1.2,0);
		\draw [dotted] (-150:.5) -- (-150:1.2);
		\draw [dotted] (90:.5) -- (90:1.2);
		\draw [dotted] (-90:.5) -- (-90:1.2);
		\draw [dotted] (140:.5) -- (140:1.2);
		\draw [dotted] (-135:.5) -- (-135:1.2);
		\draw [dotted] (-60:.5) -- (-60:1.2);
		
		\draw [red,thick] (0:.7cm) -- (0:1.2cm);
		\draw [red,thick] (140:.7cm) -- (140:1.2cm);
		\draw [blue,thick] (-135:.9cm) -- (-135:1.2cm);
		\draw [blue,thick] (-60:.9cm) -- (-60:1.2cm);
		
		\draw (0,0) circle (1.2cm);
		\draw  (0,0) circle (.5cm);
		\node at (-.5,1.75) {$L$};
		
		\draw [dashed,->] (1.75,0) -- (2.75,0);
		
		\draw (-1.2,-1.2) edge[bend right = 45,->] node[fill=white,pos=.5] {\tiny radial chamber collapse} (-1.2,-2.8);
		
		\begin{scope}[xshift=4cm]
			\draw  (0,0) circle (.5cm);
			\draw (.2,0) -- (.5,0);
			\draw (140:.5) edge[bend left=100,looseness=2.5] (0:.5);
			\draw (-135:.5) edge[bend right=100,looseness=2.5] (-60:.5);
			\node at (-.5,1.25) {$\Gamma_L$};
			\begin{scope}
				\clip  (-60:.5) rectangle (0:.5);
				\draw [thick,orange] (0,0) circle (.5cm);
			\end{scope}
		\end{scope}
		
		\begin{scope}[yshift=-4cm]
			\node at (0,0) {\footnotesize $1$};
			\draw (90:1.2cm) -- (90:1.5cm);
			\node at (90:1.6cm) {\footnotesize $1$};    
			\draw (-150:1.2cm) -- (-150:1.5cm);
			\node at (-150:1.6cm) {\footnotesize $2$};  
			\draw (-90:1.2cm) -- (-90:1.5cm);
			\node at (-90:1.6cm) {\footnotesize $3$};   	
			
			\draw [dashed] (.5,0) -- (1.2,0);
			\draw [dotted] (-150:.5) -- (-150:1.2);
			\draw [dotted] (90:.5) -- (90:1.2);
			\draw [dotted] (140:.5) -- (140:1.2);
			\draw [dotted] (-135:.5) -- (-135:1.2);
			\draw [dotted] (-90:.5) -- (-90:1.2);
			
			\draw [red,thick] (0:.7cm) -- (0:1.2cm);
			\draw [red,thick] (140:.7cm) -- (140:1.2cm);
			\draw [blue,thick] (-135:.9cm) -- (-135:1.2cm);
			\draw [blue,thick,yshift=-.7mm] (0:.9cm) -- (0:1.2cm);
			
			\draw (0,0) circle (1.2cm);
			\draw  (0,0) circle (.5cm);
			\node at (-.5,-1.75) {$L''$};
			
			\draw [dashed,->] (1.75,0) -- (2.75,0);
			
			\begin{scope}[xshift=4cm]
				\draw  (0,0) circle (.5cm);
				\draw (.2,0) -- (.5,0);
				\draw (140:.5) edge[bend left=100,looseness=2.5] node (edgemid) [pos=.5] {} (0:.5);
				\draw (-135:.5) edge[bend right=100,looseness=2.5] (edgemid.center);
				\node at (-.5,1.25) {$\Gamma_{L''}$};
				\begin{scope}
					\clip (edgemid.center) rectangle (1,-1);
					\draw [thick,orange] (140:.5) edge[bend left=100,looseness=2.5]  (0:.5);
				\end{scope}
			\end{scope}
		\end{scope}
		
		\draw [->] (5,-.25) -- node [pos=.5,fill=white] {\tiny edge collapse} (6.25,-1.25);
		\draw [->] (5,-3.25) -- node [pos=.5,fill=white] {\tiny edge collapse} (6.25,-2.25);
		
		\begin{scope}[xshift=7cm,yshift=-2cm]
			\draw  (0,0) circle (.5cm);
			\draw (.2,0) -- (.5,0);
			\draw (140:.5) edge[bend left=100,looseness=2.5] (0:.5);
			\draw (-135:.5) edge[bend right=100,looseness=2.5] (0:.5);
			\node at (.5,1.25) {$\Gamma(L,L'')$};
		\end{scope}
	\end{tikzpicture}
	\caption{An example of the radial chamber collapse zigzag.  The radial chambers are marked with dotted lines. The yellow radial chamber is collapsed in $L$ and the radial chamber collapse zigzag is given by collapsing the edges shown in orange.}
	\label{chamber_collapse}
\end{figure}

For the general case consider any $[\tilde{L}]\in\partial\overline{\Rad_{[\cL]}}\cap\Rad_{\tilde{[\cL]}}$. Then $\tilde{[\cL]}$ is obtained from $[\cL]$ by collapsing chambers. If we let $[\cL']$ be the configuration obtained from collapsing only the annular chambers, then the previous construction gives a well-defined zigzag in $\Fatad$.
\begin{equation}
	\begin{tikzcd} \Gamma_{[\cL]} \rar{\rho} & \Gamma_{[\cL']} \rar{\tau_1} & \Gamma([\cL'],[\cL]) & \Gamma_{[\cL']} \lar[swap]{\tau_2}. \end{tikzcd}
	\label{collapse_zigzag}
\end{equation} 
Note that if $\tilde{[\cL]}$ is obtained by only collapsing annular chambers then $\tau_1=\mr{id}=\tau_2$ and if $\tilde{[\cL]}$ is obtained by only collapsing radial chambers then $\rho=\mr{id}$.

\begin{definition}\label{deferad}
	We define the space $\ERad$ as follows
	\[\ERad \coloneqq \frac{\bigsqcup_{[\cL]}\Rad_{[\cL]}\times \Mark(\Gamma_{[\cL]})}{\sim},\]
	where the disjoint union runs over all non degenerate combinatorial types $[\cL]$ and the equivalence relation $\sim$ is generated by saying that $([\tilde{L}],[H])\sim([\tilde{L}],[\tilde{H}])$ if given  $[\tilde{L}]\in\partial\overline{\Rad_{[\cL]}}\cap\Rad_{\tilde{[\cL]}}$, $[H]\in\Mark(\Gamma_{[\cL]})$, $[\tilde{H}]\in\Mark(\Gamma_{\tilde{[\cL]}})$ we have that 
	$\tilde{[H]}=(\tau_{2\ast})^{-1}\circ(\tau_{1\ast})\circ\rho_{\ast}([H])$. Here $\rho$, $\tau_1$ and $\tau_2$ are given as in \eqref{collapse_zigzag} and the induced maps are the ones constructed in Remark \ref{markings}.
\end{definition}

\begin{proposition}
	The projection $\ERad\to \Rad$ is a universal $\Modgpq$-bundle over $\Rad$.
\end{proposition}

\begin{proof}
	It is enough to show that $\ERad\to \Rad$ is the $\Modgpq$-bundle corresponding to the universal surface bundle $p \colon \mr{S}_h(n,m)\to\rad \cong \Rad$.  Recall that the universal surface bundle has fibers $p_{[L]}=S([L])$, a surface with boundary with a marked point in each boundary component. These marked points are ordered and labeled as incoming or outgoing.  
	
	Let $x^L_k$ denote the marked point in the $k$th incoming boundary component for $1\leq k\leq n$ and $x^L_{k+n}$ denote the marked point in the $k$th outgoing boundary $1\leq k\leq m$.  Following the description in the beginning of this subsection, the $\Diff(S_{g,n+m})$-bundle $W\to\Rad$, corresponding to the universal surface bundle is given by taking fiberwise orientation preserving diffeomorphisms. That is, we have
	\[W_{[L]} \coloneqq \left\{ \varphi \colon \Sg\to S([L]) \,\middle|\, \text{\parbox{6cm}{$\varphi$ is an orientation-preserving diffeomorphism with $\varphi(x_i)=x^L_i$}}\right\}.\]
	Furthermore, its corresponding $\Modgpq$-bundle $Q\to\Rad$, has fibers $Q_{[L]} \coloneqq W_{[L]}/{\text{isotopy}}$. This amounts to passing to connected components of the group of diffeomorphisms.
	
	Note that $Q_{[L]}$ is discrete, and thus by the description of $\ERad$ it is enough to show that there is a bijection between $\Mark(\Gamma_{[L]})$ and  
	$Q_{[L]}$.  We define inverse maps
	\[\Phi \colon Q_{[L]}\adj\Mark(\Gamma_{[L]}) \colon \Psi\]
	By construction, there is a canonical embedding $H_{[L]} \colon \Gamma_{[L]}\cof S([L])$ and this embedding is a marking of $\Gamma_{[L]}$ in $S([L])$.  Given $[\varphi]\in Q_{[L]}$ we define $\Phi([\varphi]) \coloneqq [\varphi^{-1} \circ H_{[L]}]$, this is a well defined map.  
	
	To go back, let $[H]\in\Mark(\Gamma_{[L]})$ and choose a representative $H \colon \Gamma_{[L]}\cof \Sg$.  We will construct an orientation preserving homeomorphism $f \colon \Sg\to S([L])$ such that $[f \circ H] = [H_{[L]}]$, which we can approximate by a diffeomorphism $\varphi$ using Nielsen's approximation theorem \cite{nielsen}. To do so, we use that the complements of the markings are disks and construct the homeomorphism by first on markings and then extending it to the disks. 
	
	By Lemma \ref{marking_complement}, the complement $\Sg\backslash H(\Gamma\backslash \text{leaves of }\Gamma)$ is a disjoint union of $n+m$ cylinders. For all $1\leq i\leq n+m$, one of the boundary components of the $i$th cylinder consists of the $i$th boundary of $\Sg$. The other boundary component consists of the image of the $i$th boundary cycles of $\Gamma$ under $H$.  The leaf corresponding to the $i$th boundary component is embedded in the cylinder and connects both boundary components. We conclude that $\Sg\backslash H(\Gamma_{[L]}) \cong \bigsqcup_{i=1}^{n+m} D_i$ where each $D_i$ is a disk. 
	
	Let $x_i$ denote the marked point of the $i$th boundary component of $\Sg$.  The boundary of $D_i$ has two copies of $x_i$. Connecting these on one side is the $i$th boundary component of $\Sg$ and on the other side the embedded image of the $i$th boundary cycle of $\Gamma_{[L]}$. The orientation of the $i$th boundary component of $\Sg$ allows us to order the two copies of $x_i$ and label them as $x_{i,1}$ and $x_{i,2}$ respectively. Similarly, we have that $S([L])\backslash H_{[L]}(\Gamma_{[L]}) \cong \bigsqcup_{i=1}^{n+m} \tilde{D_i}$ where each $\tilde{D_i}$ is a disk. Let $x_{i,j}^L$ for $j=1,2$ denote the two copies of the marked point on the $i$th boundary component of $S([L])$, that lie on the boundary of $\tilde{D_i}$. Take $f_i\vert_{\partial D_i} \colon \partial D_i\to\partial\tilde{D_i}$ to be an orientation preserving homeomorphism satisfying $f(x_{i,j})=x_{i,j}^L$ for $j=1,2$. Let $f_i$ be an extension of $f_i\vert_{\partial D_i}$ to the entire disk. One can choose the maps $f_i\vert_{\partial D_i}$ consistently so that they glue together to a homeomorphism $f \colon \Sg\to S([L])$. Since the maps $f_i$ are unique up to homotopy, $f$ is also unique up to homotopy. 
	
	We define $\Psi([H])=[\varphi]$, where $\varphi$ is a diffeomorphism approximating $f$. The map $\Psi$ is well-defined and by construction it is inverse to $\Phi$.
\end{proof}

We now extend this to $\Radt$ by defining a fattening of $\ERad$ as follows:
\begin{definition} The \emph{fattening} $\ERad^\sim$ is defined as
	\[\ERad^{\sim} \coloneqq \lbrace(([L],[H]),[\Gamma,\lambda,\tilde{H}] )\,\vert\, [\Gamma,\lambda]\in\cG([L])\rbrace \subset \ERad\times\EMad\]
	where $\cG([L])$ is the space given in Definition \ref{graph_blowup}.
\end{definition}

Recall that $E\Rad$ consists of pairs $([L],[H])$ of a radial slit configuration and a marking, and that $\EMad$ consists of isomorphism classes of triples $[\Gamma,\lambda,H]$ of an admissible fat graph, a metric and a marking.

\begin{corollary}
	\label{ERad_universal}
	The projection $\ERad^{\sim}\to\Radt$ is a universal $\Modgpq$-bundle over $\Radt$
\end{corollary}
\begin{proof}
	Consider the diagram below, in which $\pi_1$ is a homotopy equivalence by Corollary \ref{cor_pi1homeq}:
	\[\begin{tikzcd} E\Radt \dar[two heads] \rar{\pi_1 \times \mr{id}} & E\Rad \dar[two heads] \\
		\Radt \rar{\simeq}[swap]{\pi_1} & \Rad.\end{tikzcd}\]
	It suffices to prove this is a pullback diagram. To do so, observe that the path from $[\Gamma,\lambda]\in\cG([L])$ to the critical graph $[\Gamma_{[L]}]$ described in Lemma \ref{GL_contractible} determines a zigzag in $|\Fatad|$ under the composite
	\[\cG([L])\stackrel{\iota}{\cof} \MFatad \xrightarrow{r(-,1)} {|\Fatad|}\]
	where $\iota$ is the inclusion and $r$ is the map give on Lemma \ref{metric_and_nerve}. Moreover, since $\cG([L])$ is contractible, $\iota$ is an inclusion and $r(-,1)$ is a homotopy equivalence there is a contractible choice of zig-zags representing paths from $[\Gamma,\lambda]$ to $[\Gamma_{[L]}]$ in $\cG([L])$.  Therefore, by Remark \ref{markings}, a marking of $[\Gamma_{[L]}]$, uniquely determines a marking of $[\Gamma]$ and vice versa.  Thus, for $[\Gamma,\lambda]\in \cG([L])$ giving a tuple $(([L],[H]),[\Gamma,\lambda,\tilde{H}] )\in \ERad\times\EMad$ is equivalent to giving either a triple $(([L],[H]),[\Gamma,\lambda] )$ or a triple $([L],[\Gamma,\lambda,\tilde{H}] )$.
\end{proof}

We now describe a general result on universal bundles, which we use to conclude that $\pi_2$ is a homotopy equivalence.

\begin{proposition}\label{prop:universal-bundles}
	Let $E\to B$ and $E'\to B'$ be universal principal $G$-bundles with $B$ and $B'$ paracompact spaces.  Let $f \colon B\to B'$ be a continuous map.  If $f^*(E')$ is isomorphic to $E$ as a bundle over $B$, then $f$ is a homotopy equivalence.
\end{proposition}
\begin{proof}
	For any paracompact space $X$ there is a diagram
	\[\begin{tikzcd}{[X,B]} \dar[swap]{f \circ -} \rar{\cong} & \lbrace\text{principal $G$-bundles over } X\rbrace \\[-2pt]
		{[X,B'],} \arrow{ru}[swap]{\cong} & \end{tikzcd}\]
	which commutes since $f^*(E')\cong E$.  For $X=B'$ one finds there is a $[g]\in [B',B]$ such that $[f\circ g]=[\mr{id}_{B'}]$.  Then, $g^*(E)\cong g^*(f^*(E')) = E' $, so we can repeat the argument and obtain that there is an $h\in [B,B']$ such that $[g \circ h]=[\mr{id}_B]$.  Finally, since $[h]=[f\circ g \circ h]=[f]$, $f$ and $g$ are mutually inverse homotopy equivalences.
\end{proof}

\begin{corollary}\label{cor_pi2homeq}
	The projection $\pi_2 \colon \Radt\to\MFatad$ is a homotopy equivalence.
\end{corollary}

\begin{proof}
	This follows from Proposition \ref{prop:universal-bundles}, as there is a pullback diagram
	\[\begin{tikzcd} E\Radt \dar[two heads] \rar{\pi_2 \times \mr{id}} & \EMad \dar[two heads] \\
		\Radt \rar{\pi_2} & \MFatad.\end{tikzcd} \]\end{proof}

%% file: sullivanandrad.tex
We now compare the harmonic compactification of radial slit configurations $\bRad$ and the space of Sullivan diagrams $\SD$, as in Definitions \ref{def_radcw} and \ref{def_sd} respectively. To do this, we observe that the $\UbRad$ is the subcomplex of $\bRad$ consisting of cells indexed by the subset $\Upsilon_{\textfrak{U}}$ of $\Upsilon$ consisting of all combinatorial types of unilevel radial slit configurations. As a consequence, the projection $p \colon \bRad \to \UbRad$ is cellular.

\begin{proposition}\label{prop_sdubradhomeo}
The space $\SD$ is homotopy equivalent to $\bRad$. In fact, there is a cellular homeomorphism between $\UbRad$ and $\SD$.
\end{proposition}

\begin{proof}It is enough to show this for connected cobordisms.  Recall that the harmonic compactification of the space of radial slit configurations $\bRad$ is homotopy equivalent to the space of unilevel radial slit configurations $\UbRad$ by Lemma \ref{lem_bradubrad}, so it suffices to prove the second stronger statement.

Since in $\UbRad$ all annuli have the same outer and inner radius and all slits sit in the outer boundary, the annular chambers are superfluous information. Thus, the combinatorial type of a unilevel configuration is determined only by its radial chamber configuration. More precisely, two univalent configurations $[L]$ and $[L']$ have the same combinatorial type if and only if they differ from each other only by the size of the radial chambers. Finally, the orientation of the complex plane and the positive real line, induce a total ordering of the radial chambers on each annulus. 

Similarly, on a Sullivan diagram, the leaves of the boundary cycles and the fat structure at the vertices where they are attached give a total ordering of the edges on the admissible cycles.  We say two Sullivan diagrams $[\Gamma]$ and $[\Gamma ']$ have the same combinatorial data if they differ from each other only on the lengths of the edges on the admissible cycles.  A \emph{(non-metric) Sullivan diagram} $G$ is an equivalence class of Sullivan diagrams under this relation.  We will first show that a radial slit configuration and a Sullivan diagram are given by the same combinatorial data.  That is, that there is a bijection
\[\begin{tikzcd}\Upsilon_{\textfrak{U}} \coloneqq \lbrace\text{combinatorial types of unilevel radial slit configurations}\rbrace \dar[leftrightarrow]\\[-7pt]
	\Lambda \coloneqq \lbrace\text{non-metric Sullivan diagrams}\rbrace. \end{tikzcd}\]

We define a map $f \colon \Upsilon_{\textfrak{U}}\to\Lambda$ by  $[\cL]\mapsto G_{[\cL],0}$ where $G_{[\cL],0}$ is the underlying (non-metric) Sullivan diagram of a unfolded graph of $[\cL]$.  This map is well defined, since a slit or a parametrization point jumping along another slit corresponds to a slide of a vertex along an edge not belonging to the admissible cycle. For example the configurations in Figure \ref{sUbRad_example} are mapped to the graphs in Figure \ref{ssd_example}.

\begin{figure}[h!]
  \centering
    \begin{tikzpicture}
    	\clip (-4.5,-5.1) rectangle (7,1.3);
    	\draw ($(3,0)+({3*360/5}:.5)$)  edge[bend right=150,looseness=7]  ($(3,0)+({4*360/5}:.5)$);
    	\draw [line width=1mm,white]($(3,0)+({2*360/5}:.58)$)  edge[bend right=150,looseness=7]  ($(3,0)+({3*360/5}:.58)$);
    	\draw (0,0) circle (.5cm);
    	\draw (3,0) circle (.5cm);
    	\draw [blue] (.2,0) node [left] {\footnotesize 1} -- (.5,0);
    	\draw [blue] (3.2,0) node [left] {\footnotesize 2} -- (3.5,0);
    	\draw (-.8,0) node [left] {\footnotesize 1} -- (-.5,0);
    	\draw (.5,0) edge[bend left=65] ($(3,0)+({360/5}:.5)$);
    	\draw ($(3,0)+({2*360/5}:.5)$)  edge[bend right=150,looseness=8]  ($(3,0)+({3*360/5}:.5)$);
    	\node [white!50!black] at ($(3,0)+({360/5-360/10}:.6)$) {\tiny 0};
    	\node [white!50!black] at ($(3,0)+({2*360/5-360/10}:.6)$) {\tiny 1};
    	\node [white!50!black] at ($(3,0)+({3*360/5-360/10}:.6)$) {\tiny 2};
    	\node [white!50!black] at ($(3,0)+({4*360/5-360/10}:.6)$) {\tiny 3};
    	\node [white!50!black] at ($(3,0)+({5*360/5-360/10}:.6)$) {\tiny 4};    	
    	\node [white!50!black] at (90:.6) {\tiny 0};
    	\node [white!50!black] at (-90:.6) {\tiny 1};
    	\node at (1.5,-1) {$[G]$};
    	
    	\begin{scope}[yshift=-3.7cm,xshift=-3cm];
    	\draw ($(3,0)+({2*360/4}:.5)$)  edge[bend right=150,looseness=7]  ($(3,0)+({3*360/4}:.5)$);
    	\draw [line width=1mm,white]($(3,0)+({1*360/4}:.58)$)  edge[bend right=150,looseness=7]  ($(3,0)+({2*360/4}:.58)$);
    	\draw (0,0) circle (.5cm);
    	\draw (3,0) circle (.5cm);
    	\draw [blue] (.2,0) node [left] {\footnotesize 1} -- (.5,0);
    	\draw [blue] (3.2,0) node [left] {\footnotesize 2} -- (3.5,0);
    	\draw (-.8,0) node [left] {\footnotesize 1} -- (-.5,0);
    	\draw (.5,0) edge[bend left=100,looseness=2.5] ($(3,0)+(0:.5)$);
    	\draw ($(3,0)+({1*360/4}:.5)$)  edge[bend right=150,looseness=8]  ($(3,0)+({2*360/4}:.5)$);
    	\node [white!50!black] at ($(3,0)+({360/4-360/8}:.6)$) {\tiny 0};
    	\node [white!50!black] at ($(3,0)+({2*360/4-360/8}:.6)$) {\tiny 1};
    	\node [white!50!black] at ($(3,0)+({3*360/4-360/8}:.6)$) {\tiny 2};
    	\node [white!50!black] at ($(3,0)+({4*360/4-360/8}:.6)$) {\tiny 3}; 	
    	\node [white!50!black] at (90:.6) {\tiny 0};
    	\node [white!50!black] at (-90:.6) {\tiny 1};
    	\node at (1.5,-1) {$d^0_1[G]$};
    	\end{scope}
    	
    	\begin{scope}[yshift=-3.7cm,xshift=3cm];
    	\draw ($(3,0)+({3*360/4}:0.8)$)  circle (.3cm);
    	\draw [line width=1mm,white]($(3,0)+({2*360/4}:.58)$)  edge[bend right=150,looseness=7]  ($(3,0)+({3*360/4}:.68)$);
    	\draw (0,0) circle (.5cm);
    	\draw (3,0) circle (.5cm);
    	\draw [blue] (.2,0) node [left] {\footnotesize 1} -- (.5,0);
    	\draw [blue] (3.2,0) node [left] {\footnotesize 2} -- (3.5,0);
    	\draw (-.8,0) node [left] {\footnotesize 1} -- (-.5,0);
    	\draw (.5,0) edge[bend left=65] ($(3,0)+({360/4}:.5)$);
    	\draw ($(3,0)+({2*360/4}:.5)$)  edge[bend right=150,looseness=8]  ($(3,0)+({3*360/4}:.5)$);
    	\node [white!50!black] at ($(3,0)+({360/4-360/8}:.6)$) {\tiny 0};
    	\node [white!50!black] at ($(3,0)+({2*360/4-360/8}:.6)$) {\tiny 1};
    	\node [white!50!black] at ($(3,0)+({3*360/4-360/8}:.6)$) {\tiny 2};
    	\node [white!50!black] at ($(3,0)+({4*360/4-360/8}:.6)$) {\tiny 3}; 	
    	\node [white!50!black] at (90:.6) {\tiny 0};
    	\node [white!50!black] at (-90:.6) {\tiny 1};
    	\node at (1.5,-1) {$d^2_3[G]$};
    	\end{scope}
    \end{tikzpicture}
  \caption{The top depicts a $5$-cell which is a product of $\Delta^1\times\Delta^4$-simplices in $\SD$, and the bottom two parts of its boundary.  The edges are numbered in grey.}
  \label{ssd_example}
\end{figure} 

We next construct the inverse map $g \colon \Lambda \to \Upsilon_{\textfrak{U}}$. Notice that any non-metric Sullivan diagram has a canonically associated metric Sullivan diagram by assigning all the edges in an admissible cycle the same length.  Moreover any Sullivan diagram has a fat graph representative with all its vertices on the admissible cycles.  A representative of a metric Sullivan diagram with all its vertices on the admissible cycles is given by the following data:
\begin{enumerate}[(i)]
\item A set of $n$ parametrized circles $C_1, C_2,\ldots,C_n$ which are disjoint, ordered, and of length 1.
\item A finite number of chords $l_1, l_2,\ldots,l_s$ where a chord is a graph which consist of two vertices connected by an edge. Let $V$ denote the set of vertices of such chords.
\item A subset $\widetilde{V}\subset V$ such that,  $\widetilde{V}$ contains at least one vertex of each chord and $\vert V \backslash \widetilde{V}\vert=m$.
\item An assignment $\alpha \colon \widetilde{V} \to \sqcup_i C_i$ which will indicate how to attach the chords onto the $n$ circles.  Two or more chords may be attached on the same circle and even on the same point.  The assignment $\alpha$ should attach at least one chord on each circle.
\item For each $x$ in the image of $\alpha$, an ordering of the subset of chords attached to $x$, that is, an ordering of the set $\alpha^{-1}(x)$.
\end{enumerate}

From this data one can construct a metric fat graph with inner vertices of valence greater or equal to $3$. The chords are attached onto the $n$ circles using $\alpha$. This gives the circles the structure of a graph by considering the attaching points as vertices and the intervals between them as edges.  It just remains to give a fat structure at the attaching points.  To do this let $x$ be in the image of $\alpha$.  The parametrization of the circles gives a notion of incoming and outgoing half edges on $x$ say $e_x^-$ and $e_x^+$ respectively. Moreover there is an ordering of the chords attached on $x$ say $(l_{x,1}, l_{x,2}, \ldots ,l_{x,s})$.  The cyclic ordering at $x$ is given by $(e_x^-, l_{x,1}, l_{x,2}, \ldots, l_{x,s}, e_x^+)$ as it is shown in Figure \ref{cyclic_ordering}.  Informally, this is to say all chords are attached on the outside of the circles according to the order given by the data.  The chords that are attached only at one vertex give the leaves of the Sullivan diagram.

\begin{figure}[h]
  \centering
    \begin{tikzpicture}
    \draw [decoration={markings, mark=at position 0.75 with {\arrow{>}}},	postaction={decorate}] (0,-1.5) arc (-90:0:1.5);
    \draw [decoration={markings, mark=at position 0.25 with {\arrow{>}}},	postaction={decorate}] (1.5,0) arc (0:90:1.5);
    \node at (1.5,0) {$\bullet$};
    \node at (1.5,0) [left] {$x$};
    \node at (45:1.8) {$e_x^+$};
    \node at (-45:1.8) {$e_x^-$};
    \draw (1.5,0) -- ($(1.5,0)+(-45:0.75)$);
    \node at ($(1.5,0)+(-45:1.1)$) {$l_{x,1}$};
    \draw (1.5,0) -- ($(1.5,0)+(-25:0.75)$);
    \node at ($(1.5,0)+(-25:1.1)$) {$l_{x,2}$};
    \draw (1.5,0) -- ($(1.5,0)+(-5:0.75)$);
    \node at ($(1.5,0)+(-5:1.1)$) {$l_{x,3}$};
    \draw (1.5,0) -- ($(1.5,0)+(45:0.75)$);
    \node at ($(1.5,0)+(45:1.1)$) {$l_{x,s}$};
    \draw [dotted,thick] ($(1.5,0)+(0:0.6)$) arc (0:45:.6);
    
    \end{tikzpicture}
  \caption{The fat structure induced at vertex $x$ where the cyclic ordering is given by the orientation on the plane.}
  \label{cyclic_ordering}
\end{figure}

From this it is clear what the inverse map $g$ should be.  Given a Sullivan diagram $G$, its associated metric Sullivan diagram gives the data (i) to (v) listed above.  Then, $g(G)=(\zeta,\lambda,\tilde{\omega},\vec{r},\vec{P})$ where $\zeta$ is given by $\alpha$ on the chords attached at both ends, $\lambda$ is given by those chords (i.e. $\lambda(i)=k$ if and only if there is a chord attached on both ends connecting $i$ and $k$), $\vec{P}$ is given by $\alpha$ on the chords attached only at one vertex, and $\tilde{\omega}$ and $\vec{r}$ are completely determined by the ordering of the chords at each attaching point.  This map is well defined since slides along chords correspond to jumps along slit, and it is an inverse to $f$.  

We will show that $\UbRad$ and $\SD$ have homeomorphic CW structures, where the cells are indexed by $\Upsilon_{\textfrak{U}} \cong \Lambda$, by giving cellular homeomorphisms
%\[\xymatrix{\UbRad & \ar[l]_(.6){\varphi}
%\displaystyle\frac{\bigsqcup_{[\cL]\in\Upsilon_{\textfrak{U}}} e_{[\cL]}}{\sim}
%\ar[r]^(.6)\psi &\SD}\]
\[\begin{tikzcd}\UbRad & \lar[swap]{\varphi}
	\displaystyle\frac{\bigsqcup_{[\cL]\in\Upsilon_{\textfrak{U}}} e_{[\cL]}}{\sim}
	\rar{\psi} &\SD.\end{tikzcd}\]

We already saw the map $\varphi$ in Definition \ref{def_radcw}. To construct the map $\psi$ one first observes that any Sullivan diagram $[\Gamma]$ in $\SD$ is uniquely determined by its non-metric underlying Sullivan diagram $G$ and a tuple $(\vec{t}_1,\ldots,\vec{t}_{n_p})$ where $t_{ij}$ is the length of the $j$th edge of the $i$th admissible cycle. Using this we can define \[\psi(e_{[\cL]},(\vec{t}_1,\ldots,\vec{t}_{n_p}))=[\Gamma]=(f([\cL]),(\vec{t}_1,\ldots,\vec{t}_{n_p})).\]
It is easy to show that the map $\psi$ is continuous and by construction the homeomorphism $\varphi \circ \psi^{-1}$ is cellular with respect to the CW structures on $\UbRad$ and $\SD$.\end{proof}

%This map is continuous since it is continuous on each $e_{[\cL]}$.  The inverse map of $\psi$ is given by $\psi^{-1}([\Gamma]=(G,(t_{10}\ldots t_{pn_p})))=(g(G),(t_{10}\ldots t_{pn_p}))$. Moreover, the map $\varphi$ is assembled from families of maps $\varphi_{[\cL]}:e_{[\cL]}\to\UbRad$ which restrict to homeomorphisms from the interior of $e_{[\cL]}$ onto their image.  These images are disjoint and cover $\UbRad$.  For each  $e_{[\cL]}$ the image of their boundary $\varphi_{[\cL]}(\partial(e_{[\cL]}))$ is contained in the image $\varphi_{[\cL]}(\cup_{i,j} d^i_j(e_{[\cL]}))$.  So these maps define a CW structure of $\UbRad$.  The same holds of $\psi_{[\cL]}:e_{[\cL]}\to\sd$ in $\sd$.

%\vskip 0pt plus 1fill
%\begin{remark}
%One can also consider non connected Sullivan diagrams and the proof 
%\end{remark}